\documentclass[11pt,letterpaper,leqno]{amsart}
\usepackage{amssymb, amsmath, amsthm, color}
\usepackage{hyperref}

\title{Asymptotic spectral flow for Dirac operators of disjoint Dehn twists}
\author{Chung-Jun Tsai}
\address{Mathematics Division, National Center for Theoretical Sciences (Taipei Office), National Taiwan University, Taipei 10617, Taiwan}
\email{cjtsai@ntu.edu.tw}
\theoremstyle{plain}
\newtheorem{thm}{Theorem}[section]
  \theoremstyle{plain}
  \newtheorem{lem}[thm]{Lemma}
  \newtheorem*{lem*}{Lemma}
  \theoremstyle{plain}
  \newtheorem{prop}[thm]{Proposition}
  \theoremstyle{definition}
  \newtheorem*{ack*}{Acknowledgements}
  \theoremstyle{definition}
  \newtheorem*{rem*}{Remark}
  \theoremstyle{definition}
  \newtheorem*{ques*}{Question}
  \theoremstyle{definition}
  \newtheorem{defn}[thm]{Definition}
  \theoremstyle{plain}
  \newtheorem{cor}[thm]{Corollary} 
  \theoremstyle{plain}
  \newtheorem*{main_thm}{Main Theorem}

\numberwithin{equation}{section}
\newcommand{\dd}{\mathrm{d}}
\newcommand{\pl}{\partial}
\newcommand{\oh}{\frac{1}{2}}

\DeclareMathOperator{\cl}{cl}

\begin{document}

\date{\today}
\maketitle

\begin{abstract}
Let $Y$ be a compact, oriented $3$-manifold with a contact form $a$.  For any Dirac operator $\mathcal{D}$, we study the asymptotic behavior of the spectral flow between $\mathcal{D}$ and $\mathcal{D}+{\rm cl}(-\frac{ir}{2}a)$ as $r\to\infty$.  If $a$ is the Thurston--Winkelnkemper contact form whose monodromy is the product of Dehn twists along disjoint circles, we prove that the next order term of the spectral flow function is $\mathcal{O}(r)$.
\end{abstract}

%%%%%%%%
\section{Introduction}
\subsection{Asymptotic spectral flow}
For a pair of purely imaginary-valued $1$-forms $A_0$ and $A_1$,  choose a path of $1$-forms $A(s)$ connecting $A_0$ to $A_1$.  For a Dirac operator $\mathcal{D}$, consider the family of Dirac operators $\{\mathcal{D}_{A(s)} = \mathcal{D} + {\rm cl}(\frac{A(s)}{2})\}_{s\in[0,1]}$, which is $\mathcal{D}$ perturbed by the Clifford action of $A(s)$.  The eigenvalues of each $\mathcal{D}_{A(s)}$ are unbounded from above and below, and vary continuously along the path.  The \emph{spectral flow} is the algebraic count of the zero crossings: a zero crossing contributes to the count with $+1$ if the eigenvalue crosses zero from a negative to a positive value as $s$ increases, and count with $-1$ if the eigenvalue crosses zero from a positive to a negative value as $s$ increases.  If the path is suitably generic, only these two cases arise.  This count is the spectral flow function.  Moreover, Atiyah, Patodi and Singer (\cite[p.95]{ref_APS3}) observed that the spectral flow function is equal to a certain index on $[0,1]\times Y$.  They also proved that this index (\cite[(4.3)]{ref_APS}) is path independent (\cite[p.89]{ref_APS3}).  Therefore, the spectral flow function depends only on the ordered pair $(\mathcal{D}_{A_0},\mathcal{D}_{A_1})$.

If we have a real-valued $1$-form $a$, we can consider the spectral flow with $A_0 = 0$ and $A_1 = -\frac{ir}{2}a$.  The spectral flow can be thought as a function of $r$, which we denote by ${\rm sf}_a(\mathcal{D},r)$.  In \cite[section 5]{ref_Taubes_SW_Weinstein} and \cite{ref_Taubes_sf}, Taubes studied the asymptotic behavior of the spectral flow function as $r\to\infty$.  He proved:

\begin{thm}[\cite{ref_Taubes_SW_Weinstein}] \label{thm_Taubes_estimate}
There exists constants $\delta\in(0,1/2)$ and $c$ with the following significance.  Suppose that $Y$ is a compact, oriented $3$-manifold equipped with a $\text{Spin}^{\mathbb{C}}$-structure, and $\mathcal{D}$ is the corresponding $\text{Spin}^{\mathbb{C}}$-Dirac operator.  Then, for any real-valued $1$-form $a$ with $||a||_{\mathcal{C}^3}\leq 1$, the spectral flow function satisfies
\begin{align}
\big|{\rm sf}_a(\mathcal{D},r)-\frac{r^2}{32\pi^2}\int_Y a\wedge\dd a\big| \leq c\,r^{\frac{3}{2}+\delta} \label{eqn_Taubes_estimate}
\end{align}
for all $r\geq c$.
\end{thm}
This theorem specifies the leading order term of the spectral flow function, and gives a bound on the next order term.

As mentioned previously, Atiyah, Patodi and Singer (\cite[p.95]{ref_APS3}) showed that the spectral flow gives the index of an associated Dirac operator.  To give the basic idea, we briefly explain the finite dimensional case.  Suppose that $H_-$ and $H_+$ are two $m\times m$, non-degenerate Hermitian matrices.  Connect them by a path of Hermitian matrices: $H(s)$ with $H(s) = H_-$ for $s\leq-1$ and $H(s) = H_+$ for $s\geq 1$.  The zero crossings of the eigenvalues of $H(s)$ is the spectral flow for $H_-$ and $H_+$.  It is easy to see that the spectral flow only depends on $H_-$ and $H_+$, not on the path $H(s)$.  With $H(s)$, define the operator $\mathfrak{D}: \mathcal{C}^{\infty}(\mathbb{R},\mathbb{C}^m)\to\mathcal{C}^{\infty}(\mathbb{R},\mathbb{C}^m)$ by
\begin{align*}
\mathfrak{D} &= \pl_s + H(s)
\end{align*}
where $s$ is the coordinate on $\mathbb{R}$.  This operator $\mathfrak{D}$ is a Fredholm operator.  Its index is given by the spectral flow between $H_-$ and $H_+$.

\subsection{Next order term on a contact $3$-manifold}
Theorem \ref{thm_Taubes_estimate} is one of the main ingredients in the proof of the Weinstein conjecture (\cite{ref_Taubes_SW_Weinstein}).  It is used to obtain the energy bound.  However, theorem \ref{thm_Taubes_estimate} is established for \emph{any} $1$-form.  When $a$ is a \emph{contact form}, much evidence (see below) suggests that the Dirac operator $\mathcal{D}_{-ira}$ is related to the Reeb vector field, and its spectral flow function ${\rm sf}_a(\mathcal{D},r)$ behaves better.  In this paper, we consider the size of the next order term when $a$ is a \emph{contact form}.

\begin{ques*} \label{eqn_question}
For a contact form $a$ with an adapted Riemannian metric, is the next order term of the spectral flow function ${\rm sf}_a(\mathcal{D},r)$ of order $\mathcal{O}(r)$ as $r\to\infty$?
\end{ques*}

  The construction of Dirac operators requires a Riemannian metric.  We always choose an adapted metric to make it easier to compare the spectral flow function.  On a contact $3$-manifold, an \emph{adapted Riemannian metric} is a metric such that $|a|=1$ and $\dd a = 2*\!a$, where $*$ is the Hodge star operator.  The existence of such a metric is proved by \cite{ref_CH}.  In this paper, we give an affirmative answer to the question in following situation:

\begin{main_thm}[Theorem \ref{thm_sf_main_00}]
Suppose that the monodromy of an open book decomposition is the product of Dehn twists along disjoint circles, and $a$ is the associated Thurston--Winkelnkemper contact form.  Then, with a certain adapted metric, the next order term of the spectral flow function of the canonical Dirac operator $\mathcal{D}$ is of order $\mathcal{O}(r)$.  Namely, there exists a constant $c$ such that
\begin{align*} \big|{\rm sf}_a(\mathcal{D},r)-\frac{r^2}{32\pi^2}\int_Y a\wedge\dd a\big|\leq c\,r ~. \end{align*}
for all $r\geq1$.
\end{main_thm}

Here is why such a bound on the next order term is expected.  As $r\to\infty$, the zero modes of the Dirac operators $\mathcal{D}_{-ira}$ have the following properties:
\begin{enumerate}
\item their derivative along the direction of the Reeb vector field is close to the multiplication by $ir/2$;
\item on the contact hyperplane, they almost satisfy the Cauchy--Riemann equation.
\end{enumerate}
The precise statements appear in proposition \ref{prop_beta_estimate_00}.  In this regard, the Dirac equation is very similar to the almost holomorphic equation in \cite{ref_Donaldson} and \cite{ref_IMP}.  This being the case, the general case looks locally like the circle bundle case.

Suppose that $Y$ is a $U(1)$-bundle over a Riemann surface $\Sigma$ with negative Euler number, and $a$ is a connection $1$-form whose curvature form is nowhere zero.  It follows that $-ia$ is a contact form on $Y$.  The adapted metric is chosen to be invariant under the $U(1)$-action.  Its spectral flow function can be computed by the Riemann--Roch formula, and the next order term is of order $\mathcal{O}(r)$.  A careful study of the Dirac operator on such $3$-manifolds can be found in \cite{ref_Nicolaescu}.

The above explanation also suggests that the zero locus of the zero modes $\psi$ is related to the Reeb vector field.  Since the derivative of $\psi$ along the Reeb vector field is about $ir\psi/2$, the derivative of $|\psi|^2$ along the Reeb vector field is small.  Since $\psi$ almost solves the Cauchy--Riemann equation on the contact hyperplane, the number of zeros of $\psi$ on the contact hyperplane should be dominated by $r$.  Therefore, the question of the next order term of the spectral flow can be viewed as the first step to understand the zero modes of $D_{-ira}$.\\

This paper is organized as follows:

Section \ref{sec_prel} provides the background for this paper.  In section \ref{subsec_sf_contact_00}, we set up the conventions of Dirac operators on contact $3$-manifold and spectral flow functions.  Section \ref{subsec_open_book_00} is a review on open book decompositions and the constructions of Thurston--Winkelnkemper contact forms.

Section \ref{sec_estimate_00} contains some basic estimates on the zero modes of $\mathcal{D}_{-ira}$.  Proposition \ref{prop_beta_estimate_00} is the cornerstone of all the estimates in this paper.

In section \ref{sec_set_up_Dirac}, we writes down the Dirac equation on different regions of the open book decomposition.  For the tubular neighborhood of the binding and the Dehn-twist region, we construct a contact form on $S^1\times S^2$.  This contact form captures the geometry of these two regions, and it is easier to study the Dirac equation on $S^1\times S^2$.  On the region with trivial monodromy, almost zero modes of the Dirac operator are constructed by the Riemann--Roch theorem.

In section \ref{sec_S2S1}, we study the Dirac equation on the model manifold $S^1\times S^2$.  There are two main ingredients.  First, we construct an approximation for the zero eigensections of the Dirac operator in this model case.  Second, we cut the interval $[0,r]$ into subintervals such that there are sparse zero crossings near the endpoints of each subinterval.

In section \ref{sec_lower_bound}, we combine above results to obtain the lower bound of the spectral flow function.

Section \ref{sec_dirac_page_00} is a further study of the Dirac equation on the region where the monodromy is trivial.  We describe the boundary behavior of the solutions carefully.

In the last section, we prove the upper bound of the spectral flow function.  Unlike the eigensections of a fixed Dirac operator, zero eigensections of a family of Dirac operators are not necessarily orthogonal to one another.  This section is devoted to overcoming this difficulty.

\begin{ack*}
This paper forms part of the author's Ph.D.\ thesis.  The author would like to thank his thesis advisor Cliff Taubes for his guidance in this project.  He would also like to thank Po-Ning Chen and Valentino Tosatti for helpful discussions, and to thank the anonymous referee for the many useful comments, suggestions and corrections..  Part of this work was done while the author visited MSRI during the academic year 2009--2010.  He is grateful to MSRI and to the organizers of the ``Symplectic and Contact Geometry and Topology" program for their hospitality.
\end{ack*}

%%%%%%%%
\section{Preliminary}\label{sec_prel}
In this section, we set up the background on the spectral flow function and the open book decomposition.

%%%%
\subsection{Spectral flow of a contact form}\label{subsec_sf_contact_00}
We now review the Dirac operator and its spectral flow.  We focus on the \emph{canonical} Dirac operator perturbed by the contact form.  We will not do the general $\text{\rm Spin}^{\mathbb{C}}$-structure constructions.  See \cite[ch.2 \& ch.3]{ref_Morgan} for a complete treatment for general $\text{\rm Spin}^{\mathbb{C}}$-structures.   Below we mostly follow \cite{ref_Taubes_SW_Weinstein}.

\subsubsection{Canonical $\text{\rm Spin}^{\mathbb{C}}$-structure}\label{subsubsec_sf_contact_00}
Let $Y$ be a closed oriented connected $3$-manifold with a contact form $a$.  Fix an adapted metric on $Y$.  Denote the Reeb vector field by $e_3$.  It has unit length measured by any adapted metric.

Consider the contact hyperplane field $\ker(a)\subset TY$.  With the orientation given by $\dd a$, it is a Hermitian line bundle over $Y$.  More precisely, for any $v\in\ker(a)$, let $J(v)$ be the metric dual of $v\lrcorner\dd a$.  The complex line bundle is spanned by local sections $v-iJ(v)$ with $v\in\ker(a)$.  The Hermitian metric is induced from the adapted metric.  We denote this Hermitian line bundle by $K^{-1}$, and its inverse bundle $K$ is called the \emph{canonical line bundle}.  The bundle $\underline{\mathbb{C}}\oplus K^{-1}$ has a Clifford action
\begin{align*} {\rm cl}: TY \longrightarrow {\rm End}(\underline{\mathbb{C}}\oplus K^{-1})  \end{align*}
defined as follows.  The bundle $\underline{\mathbb{C}}$ is the pull-back of $\mathbb{C}$ over a point by the map $Y\to\text{point}$.  Let ${\bf 1}_{\underline{\mathbb{C}}}$ be the pull-back of $1$ by the map $Y\to\text{point}$.  For any oriented orthonormal frame on the contact hyperplane $\{e_1,e_2\}$, $\frac{1}{\sqrt{2}}(e_1-ie_2)$ is a unit-normed, trivializing section of $K^{-1}$.  With this trivialization, the Clifford action is given by the Pauli matrices
\begin{align*}
\cl(e_1) &= \left[\begin{array}{cc} 0&-1\\1&0 \end{array}\right] ~,
&\cl(e_2) &= \left[\begin{array}{cc} 0&i\\i&0 \end{array}\right] ~,
&\cl(e_3) &= \left[\begin{array}{cc} i&0\\0&-i \end{array}\right] ~.
\end{align*}
The bundle $\underline{\mathbb{C}}\oplus K^{-1}$ together with this Clifford action is called the \emph{canonical $\text{\rm Spin}^{\mathbb{C}}$-structure}.

\subsubsection{Dirac operator}
A \emph{$\text{\rm Spin}^{\mathbb{C}}$-connection} on the canonical $\text{\rm Spin}^{\mathbb{C}}$-structure is a Hermitian connection $\nabla_A$ on $\underline{\mathbb{C}}\oplus K^{-1}$, which is compatible with the Clifford action in the following sense: 
\begin{align*} \nabla_A(\cl(v)\psi) &= \cl(\nabla v)\psi + \cl(v)\nabla_A\psi  \end{align*}
for any vector field $v$ and any section $\psi$ of $\underline{\mathbb{C}}\oplus K^{-1}$.  Here $\nabla v$ is the covariant derivative of $v$ with respect to the Levi-Civita connection on $TY$.

Given a $\text{Spin}^{\mathbb{C}}$-connection, define the \emph{Dirac operator} to be the composition
\begin{align*} \mathcal{C}^{\infty}(\underline{\mathbb{C}}\oplus K^{-1}) \stackrel{\nabla_A}{\longrightarrow} \mathcal{C}^{\infty}(T^*Y\otimes(\underline{\mathbb{C}}\oplus K^{-1})) \stackrel{\cl}{\longrightarrow} \mathcal{C}^{\infty}(\underline{\mathbb{C}}\oplus K^{-1}) \end{align*}
The Clifford action is extended to the cotangent bundle by the metric dual.

According to \cite[lemma 10.1]{ref_Hutchings}, there exists a unique $\text{\rm Spin}^{\mathbb{C}}$-connection such that the section $({\bf 1}_{\underline{\mathbb{C}}},0)$ is annihilated by the associated Dirac operator.  This connection is called the \emph{canonical connection}.  We denote it by $\nabla_0$, and denote its associated Dirac operator by $D_0$.  Perturb the canonical connection by $-ira/2$ with $r\geq0$
\begin{align}\label{eqn_can_conn_00}  \nabla_r = \nabla_0 - \frac{ir}{2}a ~,  \end{align}
and consider the corresponding Dirac operator
\begin{align}\label{eqn_can_Dirac_00}  D_r = D_0 + {\rm cl}(\frac{-ir}{2}a) ~.  \end{align}
Note that ${\rm cl}(a)$ acts as $i$ on the $\underline{\mathbb{C}}$-component and as $-i$ on the $K^{-1}$-component.

\subsubsection{Local expression of the canonical connection}
We now write down the canonical connection in terms of the trivialization explained in \ref{subsubsec_sf_contact_00}.  The complete treatment of the local computation of a $\text{\rm Spin}^{\mathbb{C}}$-connection and its Dirac operator can be found in \cite[section 3.2 \& 3.3]{ref_Morgan}.

Let $\{e_1,e_2,e_3\}$ be a local orthonormal frame as explained in \ref{subsubsec_sf_contact_00}, and let $\{\theta_j^k\}_{j\neq k}$ be the Levi-Civita connection $1$-form.  Namely, $\nabla e_j = \sum_{k\neq j}\theta_j^k\otimes e_k$ where $\nabla$ is the Levi-Civita connection.  Using the trivialization $\{{\bf 1}_{\underline{\mathbb{C}}},\frac{1}{\sqrt{2}}(e_1-ie_2)\}$ of $\underline{\mathbb{C}}\oplus K^{-1}$, we identify the local sections with $\mathbb{C}^2$-valued functions.  We claim that the canonical connection $\nabla_0\psi$ is
\begin{align}\begin{split} \label{eqn_can_conn_01}
\dd \psi - \frac{1}{2}\big(\theta_1^2\cl(e_3) + \theta_3^1\cl(e_2) + \theta_2^3\cl(e_1)\big)(\psi) + \frac{i}{2}( - 2a + \theta_1^2)\psi
\end{split}\end{align}
where $\dd \psi$ is the usual exterior derivative of a $\mathbb{C}^2$-function.  According to \cite[(3.3)]{ref_Morgan}, (\ref{eqn_can_conn_01}) defines a $\text{\rm Spin}^{\mathbb{C}}$-connection.

To prove that the above expression is the canonical connection, it remains to check that $({\bf 1}_{\underline{\mathbb{C}}},0)$ is annihilated by the associated Dirac operator.  Let $\{\omega^1,\omega^2,\omega^3\}$ be the dual coframe of $\{e_j\}$.  Since the metric is \emph{adapted}, $\omega^3$ is the contact form $a$.  By comparing $\dd a = 2\omega^1\wedge\omega^2$ with the structure equation $\dd\omega^3 = -\sum_{j\neq3}\theta_j^3\wedge\omega^j$, we conclude that $\theta_1^3(e_3) = 0 =\theta_2^3(e_3)$ and $\theta_1^3(e_2) - \theta_2^3(e_1) = 2$.  By calculating $\dd^2 a = 0 = 2\dd(\omega^1\wedge\omega^2)$ in terms of the structure eqaution, we conclude that $\theta_1^3(e_1) + \theta_2^3(e_2) = 0$.  With these relations, the Dirac operator of (\ref{eqn_can_conn_01}) is
\begin{align}\label{eqn_can_Dirac_01}
\sum_{i=1}^3\cl(e_i)e_i(\psi) +
\left[\begin{array}{cc}
0 & -i\theta_1^2(e_1) - \theta_1^2(e_2) \\
0 & -2 + \theta_1^2(e_3)
\end{array}\right]\psi ~.
\end{align}
It is clear that $({\bf 1}_{\underline{\mathbb{C}}},0)$ is annihilated by the operator.  Thus, the local expression of the canonical connection $\nabla_0$ is given by (\ref{eqn_can_conn_01}), and the local expression of its associated Dirac operator $D_0$ is given by (\ref{eqn_can_Dirac_01}).

\subsubsection{Spectral flow function}\label{subsec_sf_intro}
We now define the spectral flow function for the family of Dirac operators $\{D_s\}_{s\in[0,r]}$.  The construction is borrowed from \cite[subsection 5.1]{ref_Taubes_SW_Weinstein}.

The spectral flow for the family $\{D_s\}_{s\in[0,r]}$ is defined with the help of a certain stratified, real-analytic set in $\mathbb{R}\times[0,r]$.  This set is denoted by $\mathcal{E}$.  Its stratification is given by
\begin{align*} \mathcal{E}=\mathcal{E}_1\supset\mathcal{E}_2\supset\cdots, \end{align*}
where $\mathcal{E}_l$ of the set of pairs $(\lambda,s)$ such that $\lambda$ is an eigenvalue of $D_s$ with multiplicity $l$ or greater.  Each $\mathcal{E}_l$ is a closed set.  Moreover, as can be proved using the results in \cite[chapter 7]{ref_Kato}, each $\mathcal{E}_{l^*}=\mathcal{E}_l - \mathcal{E}_{l+1}$ is an open and real analytic submanifold of $\mathbb{R}\times[0,r]$.  The collection $\{\mathcal{E}_{l^*}\}$ are called the \emph{smooth strata} of $\mathcal{E}$.  When the $1$-dimensional smooth strata are oriented by the pull-back from $\mathbb{R}\times[0,r]$ of the $1$-form $\dd s$.  The zero dimensional strata can be consistently oriented so that the formal, weighted sum $\mathcal{E}_* = \sum_{l\in\mathbb{N}}\mathcal{E}_{l^*}$ defines a locally \emph{closed} cycle in $\mathbb{R}\times[0,r]$.  It also follows from the results in \cite[chapter 7]{ref_Kato} that
\begin{align*} \sum_{l\in\mathbb{N}}\int_{\mathcal{E}_{l^*}}\dd h = 0 \end{align*}
for any smooth function $h$ on $\mathbb{R}\times(0,r)$ with compact support.

Sard's theorem finds a dense, open set $\mathbb{U}\subset\mathbb{R}$ with the property that the two maps from a point $\star$ to $\mathbb{R}\times[0,r]$ that send $\star$ to $(\lambda,0)$ and to $(\lambda,r)$, respectively, are both transverse to the smooth strata of $\mathcal{E}$ for all $\lambda\in\mathbb{U}$.  With this understood, the spectral flow for $\{D_s\}_{s\in[0,r]}$ is defined as follows.  Fix $\lambda_0\in\mathbb{U}$ and $\lambda_0<0$.  By Sard's theorem, there exist smooth, oriented paths $\sigma\subset\mathbb{R}\times[0,r]$ that start at $(\lambda_0,0)$, end at $(\lambda_0,r)$, and are transverse to the smooth strata of $\mathcal{E}$.  Such a path has the following well-defined intersection number with $\mathcal{E}$:
\begin{align*}
{\rm sf}_a(r,\lambda_0) = \sum_{l\in\mathbb{N}}\sum_{p\in\sigma\cap\mathcal{E}_{l^*}}\iota(p) l ~, \end{align*}
where $\iota(p)\in\{-1,1\}$ is the sign of intersection.  To describe the sign of intersection, suppose that $\sigma$ is the graph of a smooth function from $[0,r]$ to $\mathbb{R}$.  The pull-back of $\dd\lambda$ from $\mathcal{E}$ to $\mathbb{R}\times[0,r]$ at a point $(\lambda,u)$ can be written as $\lambda'\dd u$ with
\begin{align}\label{eqn_dlambda}
\lambda' = \int_Y \langle\psi,{\rm cl}(\frac{-i}{2}a)\psi\rangle ~,
\end{align}
where $\psi$ is a unit-normed eigensection of $D_u$ with eigenvalue $\lambda$.  The sign of $\lambda'$ at an intersection point with the image of a graph $\sigma$ is the factor $\iota(p)$.

If $\lambda_0$ is sufficiently close to $0$, ${\rm sf}_a(r,\lambda_0)$  is independent of $\lambda_0$.  The \emph{spectral flow function} for $\{D_u\}_{u\in[0,r]}$ is defined to be
\begin{align*} {\rm sf}_a(r) = \lim_{\lambda_0\to 0^-}{\rm sf}_a(r,\lambda_0) ~. \end{align*}

%%%%
\subsection{Open book decomposition}\label{subsec_open_book_00}
We briefly review open book decompositions of disjoint Dehn twists.  For a complete discussion of open book decompositions, see \cite[chapter 9]{ref_OS} or \cite{ref_Etnyre} and the references therein.  The notations introduced in this section will be used throughout the paper.

\subsubsection{Open book decomposition}
An \emph{(abstract) open book decomposition} consists of a Riemann surface $\bar{\Sigma}$ with boundaries $\{C_j\}_{j\in J}$, and a self-diffeomorphism $\tau$ which is the identity near the boundary.  The map $\tau$ is called the \emph{monodromy} of the open book.  The $3$-manifold $Y$ is obtained by the following construction.  First, form the mapping torus
$$ \bar{\Sigma} \times_\tau{S}^1 = \frac{\bar{\Sigma}\times[0,2\pi]}{(p,2\pi)\sim(\tau(p),0)} ~. $$  Its boundary is the disjoint union of tori, $\coprod_{j}C_j \times{S}^1 $.  Next, attach solid tori $\coprod_{j}{S}^1\times{D}^2$ to $\bar{\Sigma} \times_\tau{S}^1$, where the longitude is identified with the $C_j$-factor, and the meridian is identified with the ${S}^1$-factor.  Note that there is an ${S}^1$-family of $\bar{\Sigma}$ in $Y$, and they are referred as the \emph{pages}.  The cores of the attached solid tori are called the \emph{bindings}.

In the next section, the handle attaching will be described explicitly in terms of local coordinates.

\subsubsection{Contact form}
Given an open book decomposition, Thurston and Winkelnkemper \cite{ref_Thurston_W} construct a contact form $a$ on it, which has the following significance:
\begin{itemize}
\item on the mapping torus $\bar{\Sigma}\times_\tau{S}^1$, the Reeb vector field is transverse to the pages, and ${\rm d}a$ restricted on the pages is an area form;
\item on the attaching solid tori, $a$ is of a certain standard form, and the bindings are Reeb orbits.
\end{itemize}

By the celebrated work of Giroux \cite{ref_Giroux}, this construction produces every contact structure up to isotopy.

We now describe the contact form.  We first explain the case when the monodromy is the identity map, then do the case when the monodromy is the product of Dehn twists along disjoint circles.  In what follows, $\epsilon$ is a constant smaller than $1/100$.  The precise value of $\epsilon$ will be chosen in the construction of Dehn twists.

\subsubsection*{With trivial monodromy}
If the monodromy is the identity map, the mapping torus is $\bar{\Sigma}\times S^1$.  Near each boundary circle $C_j$, let $\rho e^{it}$ be the coordinate on its collar neighborhood where $\rho\in[1,1+20\epsilon)$ and $C_j = \{\rho=1\}$.  Choose a $1$-form $\mu_{\bar{\Sigma}}$ on $\bar{\Sigma}$ satisfying
\begin{itemize}
\item $\dd\mu_{\bar{\Sigma}}$ is an area form on $\bar{\Sigma}$;\smallskip
\item $2\mu_{\bar{\Sigma}} = (2-\rho)\dd t$ on the collar neighborhood of $C_j$.
\end{itemize}
Such $\mu_{\bar{\Sigma}}$ always exist (\cite[p.141]{ref_OS}).  Let $e^{i\phi}$ be the coordinate on the $S^1$-component of $\bar{\Sigma}\times S^1$.  The contact form on this region is taken to be
\begin{align} \label{eqn_contact_identity_00} a &= V\dd\phi+2\mu_{\bar{\Sigma}} \end{align}
where $V$ is a constant greater than $1$.  Later, $V$ will be adjusted to be a larger constant.  We remark that the total area of $\dd\mu_{\bar{\Sigma}}$ is given by
\begin{align*}
\int_{\bar{\Sigma}}\dd\mu_{\bar{\Sigma}} = \sum_{j}\int_{C_j}\mu_{\bar{\Sigma}} = \pi\cdot\#\{\text{boundary components}\} ~.
\end{align*}

\subsubsection*{Attaching handles}
Let $(e^{it},\rho e^{i\phi})$ be the coordinate on the attaching solid torus $S^1\times D^2$ where $\rho e^{i\phi}$ is the polar coordinate and $\rho\leq1$.  The handle attaching is done by identifying the coordinates with the above region.  Choose two smooth functions $f$ and $g$ which only depend on $\rho$ such that
\begin{itemize}
\item when $\rho\in[1-5\epsilon,1]$, the function $f$ is $V$, and $g$ is $2-\rho$;\smallskip
\item when $\rho\in[0,10\epsilon]$, the function $f$ is $\rho^2$, and $g$ is $2-\rho^2$;\smallskip
\item $f'(\rho)\geq 0$, and $g'(\rho)<0$ except at $\rho=0$.
\end{itemize}
It is not hard to see the existence of $f$ and $g$.  The contact form on the attaching solid torus is taken to be
\begin{align} \label{eqn_contact_binding_00} a &= f(\rho)\dd\phi + g(\rho)\dd t ~. \end{align}
When $\rho<10\epsilon$, it is equal to $x\dd y - y\dd x + (2-x^2-y^2)\dd t$ in terms of the rectangular coordinate $x+iy = \rho e^{i\phi}$.  Hence, the $1$-form $a$ is smooth on the solid torus $S^1\times D^2$.

\subsubsection*{Disjoint Dehn twists}
A {Dehn twist} along a simple closed curve $\Gamma$ is a certain type of monodromy.  Roughly speaking, it is obtained by cutting a collar neighborhood of $\Gamma$, twisting $2\pi$ to the right and re-gluing.  The precise description in terms of local coordinates will be given shortly.

If the monodromy is the product of Dehn twists along disjoint circles $\{\Gamma_l\}$, the contact form only needs to be modified on the tubular neighborhood of $\{\Gamma_l\}$.

Note that $\dd\mu_{\bar{\Sigma}}$ is a symplectic form and $\Gamma_l$ is a Lagrangian submanifold.  By the Weinstein tubular neighborhood theorem, there exists a coordinate $(\rho, e^{it})$ on a tubular neighborhood of $\Gamma_l$ such that $\dd\mu_{\bar{\Sigma}} = \dd t\wedge\dd\rho$ with $\Gamma_l=\{\rho=0\}$.  Choose $\epsilon$ small enough so that $\rho\in(-35\epsilon,35\epsilon)$.  By adding the differential of a smooth function, we may assume that $2\mu_{\bar{\Sigma}}$ is $2(v_l-\rho)\dd t$ on the tubular neighborhood of $\Gamma_l$.  The period constant $2\pi v_l = \int_{\Gamma_l}\mu_{\bar{\Sigma}}$ is determined by the original choice of $\mu_{\bar{\Sigma}}$.

To perform the Dehn twist along $\Gamma_l$, choose a smooth function $\sigma_l(\rho)$ satisfying
\begin{align*}
\sigma_l(\rho) = \begin{cases}
0 &\text{when }\rho\leq -10\epsilon \\
\pm N_l &\text{when }\rho\geq10\epsilon
\end{cases} \end{align*}
where the plus/minus sign corresponds to the right-handed/left-handed (positive/negative) Dehn twist, and $N_l\in\mathbb{N}$ corresponds to the power of the Dehn twist.  With $\sigma_l$ chosen, the ($\pm N_l$)-Dehn twist along $\Gamma_l$ is given by $(\rho,e^{it})\mapsto(\rho,e^{i(t+2\pi\sigma_l(\rho))})$.  The mapping torus of the tubular neighborhood of $\Gamma_l$ is
\begin{align*}
\frac{(-35\epsilon,35\epsilon)\times S^1\times[0,2\pi]}{(\rho,e^{it},2\pi)\sim({\rho},e^{i(t + 2\pi\sigma_l({\rho}))},0)} ~.
\end{align*}
Let $\phi$ be the coordinate on the interval $[0,2\pi]$.  If we take $V$ to be a large constant such that
\begin{align} \label{eqn_contact_Dehn_00}
V &\geq 1+ 2\big|(v_l-\rho)^2\sigma'_l(\rho)\big| + 2\big|(v_l-\rho)\sigma_l(\rho)\big| \end{align}
for $|\rho|<35\epsilon$, then the $1$-form
\begin{align} \label{eqn_contact_Dehn_01}
a &= V\dd\phi + 2(v_l-\rho)\dd t + 2\phi(v_l-\rho)\sigma'_l(\rho)\dd\rho
\end{align}
is a contact form which coincides with (\ref{eqn_contact_identity_00}) when $|\rho|>10\epsilon$.  More precisely, (\ref{eqn_contact_Dehn_01}) is a contact form on $(-35\epsilon,35\epsilon)\times S^1\times[0,2\pi]$, and is invariant under the identification map $(\rho,e^{it},\phi)\mapsto(\rho,e^{i(t+2\pi\sigma_l(\rho))},\phi-2\pi)$.

\begin{defn}
Let $\Sigma$ be the surface obtained by cutting out the tubular neighborhood of $\Gamma_l$, $\{|\rho|<20\epsilon\}$, from $\bar{\Sigma}$.  The monodromy is the identity on $\Sigma$.  The restriction of $\mu_{\bar{\Sigma}}$ on $\Sigma$ is denoted by $\mu_\Sigma$.
\end{defn}

%%%%
\subsection{Spectral flow estimate}  With the above construction, here is the precise statement of our main result:

\begin{thm} \label{thm_sf_main_00}
Suppose that the monodromy of the open book decomposition is the product of Dehn twists along some disjoint circles.  For the contact form $a$ given by (\ref{eqn_contact_identity_00}), (\ref{eqn_contact_binding_00}) and (\ref{eqn_contact_Dehn_01}), there exist an adapted Riemannian metric and a constant $c$ such that
\begin{align*}
\Big| {\rm sf}_a(r) - \frac{r^2}{32\pi^2}\int_Y a\wedge\dd a \Big| \leq c\,r
\end{align*}
for all $r\geq 1$.
\end{thm}
The proof is done by the gluing construction.  For simplicity, we will assume that there is only one Dehn twist and one binding, and we will suppress the subscript $l$.  If there is more than one Dehn twist or binding, the argument is essentially the same.

\begin{rem*}
Throughout this paper, the constants $c$ only depend on the contact form and the adapted Riemannian metric, and do \emph{not} depend on $r$.  Within each proof, the subscript of the constants $c_{(*)}$ is only for indicating that they might change (usually increase) after each step.  The constants $c_{(*)}$ in two different proofs have nothing to do with each other.
\end{rem*}

%%%%%%%%
\section{Some estimates}\label{sec_estimate_00}
The purpose of this section is to derive some basic estimates on the zero eigensections of $D_r$.  It is crucial that the perturbation term of $D_r$ is the contact form.  The estimates in this section do not involve the open book decomposition.

We write a section $\psi$ of the canonical $\text{\rm Spin}^{\mathbb{C}}$ bundle $\underline{\mathbb{C}}\oplus K^{-1}$ as $(\alpha,\beta)$.  We will refer to $\alpha$ and $\beta$ as the first and the second component of $\psi$, respectively.  Remember that under the Clifford action, $a$ acts as $i$ on the first component and acts as $-i$ on another component.

\begin{prop}\label{prop_beta_estimate_00}
For any $\delta_1 \geq 0$, there is a constant $c$ determined by the contact form and the adapted metric such that the following holds.
\begin{enumerate}
\item Suppose that $\psi=(\alpha,\beta)$ is a eigensection of $D_r$ for some $r\geq c$, and the magnitude of the corresponding eigenvalue is less then or equal to $\delta_1$.  Then
\begin{align*}
\int_Y |\beta|^2 + r^{-1}\int_Y|\nabla_r\beta|^2 \leq c\,r^{-1} \int_Y |\alpha|^2 ~.
\end{align*}
Hence, $\int_Y |D_r\beta|^2 \leq c \int_Y |\alpha|^2$.
\item Furthermore, suppose that there is a $\mathcal{C}^1$-family of eigenvalues $\lambda(s)$ of $D_{r+s}$ near some $r\geq c$, and $|\lambda(0)|\leq\delta_1$.  Then
\begin{align*}
|\lambda'(0) - \oh| &\leq \frac{c}{r}.
\end{align*}
Therefore, there exist only positive zero crossings for the spectral flow for $r\geq c$.
\end{enumerate}
\end{prop}
\begin{proof}
The Weitzenb\"ock formula (\cite[proposition 5.1.5]{ref_Morgan}) reads
\begin{align*} D_r^2\psi = \nabla_r^*\nabla_r\psi + \frac{\kappa}{4}\psi + {{\rm cl}}(\frac{F_{A_0}}{2})\psi - ir{{\rm cl}}(*a)\psi \end{align*}
where $\kappa$ is the scalar curvature and $F_{A_0}$ is the curvature of the canonical connection.
 Take the inner product with $\beta$, and integrate the equation over $Y$.  We have
\begin{align*}
\delta_1^2 \int_Y |\beta|^2 &\geq  \int_Y \Big( |\nabla_r\beta|^2+ \langle N(\nabla_r\alpha) + N'(\alpha),\beta\rangle \\
&\qquad\qquad + \frac{\kappa}{4}|\beta|^2+\langle{\rm cl}(\frac{F_{A_0}}{2})\psi,\psi\rangle+r|\beta|^2\Big)
\end{align*}
where $N(\nabla_r\alpha) = i\cl({\rm tr}\langle\nabla a,\nabla_r\alpha\rangle)$ and $N'(\alpha) = \frac{i}{2}\cl(\nabla^*\nabla a)\alpha$.  The term $N(\nabla_r\alpha) + N'(\alpha)$ has the same $K^{-1}$-component as $\nabla_r^*\nabla_r\alpha$.  Notice that $N$ and $N'$ are operators independent of $r$.  After integration by parts, we conclude that
\begin{align*}
\delta_1^2 \int_Y |\beta|^2&\geq \int_Y \big((r-c_1)|\beta|^2 + \frac{1}{2}|\nabla_r \beta|^2 - c_2 |\alpha|^2\big) ~.
\end{align*}
Property (i) of the proposition follows from this inequality.\medskip

To prove property (ii), let $\psi=(\alpha,\beta)$ be a unit-normed eigensection of $D_r$ with eigenvalue $\lambda(0)$.  According to (\ref{eqn_dlambda}),
\begin{align*}
\lambda'(0) &= \oh \int_Y (|\alpha(0)|^2 - |\beta(0)|^2) = \oh - \int_Y |\beta(0)|^2 ~.
\end{align*}
This equation together with property (i) proves property (ii).
\end{proof}

Property (ii) of proposition \ref{prop_beta_estimate_00} says that the rate of change of the spectrum near a zero crossing gets closer to $1/2$ as $r\to\infty$.  It has the following consequence.
\begin{cor}\label{cor_sf_exist_00}
For any $\delta_1>0$, there exist a constant $c>0$ determined by the contact form and the adapted metric such that the following holds.  Suppose that $\lambda$ is an eigenvalue of $D_r$ with $r\geq c$ and $4|\lambda|\leq\delta_1$.  Then, $(r,\lambda)$ belongs to a trajectory of eigenvalues which contributes to the spectral flow with $+1$ somewhere in the interval
\begin{align*} [r-2\lambda-\frac{c}{r}, r-2\lambda+\frac{c}{r}] ~. \end{align*}
\end{cor}

\begin{proof}
According to \ref{subsec_sf_intro}, we may choose a trajectory $(r+s,\mu(s))$ for $s\in(-3\lambda,3\lambda)$ such that $\mu(0)=\lambda$ and $\mu(s)$ is continuous and piecewise differentiable.  Such a trajectory is not unique, but any choice will suffice.  By (\ref{eqn_dlambda}), $|\mu'(s)|\leq\oh$ as long as $\mu(s)$ is differentiable at $s$.  It follows that $|\mu(s)|\leq\frac{5}{2}|\lambda|<\delta_1$ for any $s\in(-3\lambda,3\lambda)$.  Therefore, property (ii) of proposition \ref{prop_beta_estimate_00} implies that $|\mu'(s)-\oh|\leq\frac{c}{r}$ provided $\mu$ is differentiable at $s$.  The corollary follows from this estimate.
\end{proof}

Two zero eigensections at different $r$ might not be orthogonal to each other.  However, property (i) of proposition \ref{prop_beta_estimate_00} implies that they are close to being orthogonal.  More precisely, we have:
\begin{prop} \label{prop_small_ip_00}
There exists a constant $c>0$ which determined by the contact form and the adapted metric such  that the following holds.  Suppose that $\psi_1$ and $\psi_2$ are zero modes of $D_{r_1}$ and $D_{r_2}$ with $r_1\geq c$, $r_2\geq c$ and $r_1\neq r_2$ and they are of unit $L^2$-norm.  Then
\begin{align*}
\big|\int_Y \langle \psi_1,\psi_2\rangle\big| \leq c (r_1 r_2)^{-\oh} ~.
\end{align*}

\end{prop}
\begin{proof}
We write $\psi_1=(\alpha_1,\beta_1)$ and $\psi_2=(\alpha_2,\beta_2)$, and compute
\begin{align*}
0 &= \int_Y \big(\langle D_{r_1}\psi_1,\psi_2\rangle - \langle \psi_1,D_{r_2}\psi_2\rangle \big) = \int_Y \langle (D_{r_1}-D_{r_2})\psi_1,\psi_2 \rangle \\
&= (r_1-r_2) \int_Y \big(\langle\alpha_1,\alpha_2\rangle - \langle\beta_1,\beta_2\rangle \big) ~.
\end{align*}
It follows from $r_1\neq r_2$ that $\int_Y \langle\alpha_1,\alpha_2\rangle = \int_Y \langle\beta_1,\beta_2\rangle$.  Hence, $\int_Y\langle\psi_1,\psi_2\rangle = 2\int_Y\langle\beta_1,\beta_2\rangle$.  By property (i) of proposition \ref{prop_beta_estimate_00} and the Cauchy--Schwarz inequality, this completes the proof of the proposition.
\end{proof}

%%%%%%%%
\section{The Dirac equations}\label{sec_set_up_Dirac}
In this section, we will write down the Dirac equations on different regions of the $3$-manifold of an open book decomposition.  The adapted metric for theorem \ref{thm_sf_main_00} will also be specified.

%%%%
\subsection{The tubular neighborhood of the binding}\label{subsec_binding_00}
The tubular neighborhood of the binding has coordinate $(e^{it},\rho e^{i\phi})\in S^1\times D^2$.  The contact form is given by (\ref{eqn_contact_binding_00}),
\begin{align*}
a &= f\dd\phi + g\dd t ~.
\end{align*}
The contact form with the following two coframes
\begin{align*}
\omega^1 &= \cos\phi\dd\rho - \sin\phi\big(\frac{f'}{2}\dd\phi + \frac{g'}{2}\dd t\big) ~, \\
\omega^2 &= \sin\phi\dd\rho + \cos\phi\big(\frac{f'}{2}\dd\phi + \frac{g'}{2}\dd t\big)
\end{align*}
specifies the Riemannian metric, $a^2+(\omega^1)^2+(\omega^2)^2$.  It is easy to see that the coframe is smooth except at $\rho=0$.  For $\rho<10\epsilon$, $\omega^1 = \dd x + y\dd t$ and $\omega^2 = \dd y - x\dd t$ in terms of the rectangular coordinate $x+iy = \rho e^{i\phi}$.  Thus, they form a smooth orthonormal coframe on $S^1\times D^2$.  As explained in \ref{subsubsec_sf_contact_00}, their dual frame induces a trivialization of $\underline{\mathbb{C}}\oplus K^{-1}$.  On the tubular neighborhood of the binding, this trivialization identifies the sections of $\underline{\mathbb{C}}\oplus K^{-1}$ with $\mathbb{C}^2$-valued functions.  By (\ref{eqn_can_Dirac_00}) and (\ref{eqn_can_Dirac_01}), the Dirac operator $D_r$ on $\psi=(\alpha,\beta)$ is
\begin{align}\label{eqn_Dirac_neck_01}
D_r\psi = \left\{\begin{aligned}
&\frac{r}{2}\alpha + \frac{i}{2\Delta}(-g'\pl_\phi\alpha + f'\pl_t\alpha) \\
&\qquad + e^{-i\phi}\big( -\pl_\rho\beta + \frac{i}{\Delta}(g\pl_\phi\beta - f\pl_t\beta) + \frac{g-\Delta'}{\Delta}\beta \big) ~, \\
&e^{i\phi}\big( \pl_\rho\alpha + \frac{i}{\Delta}(g\pl_\phi\alpha - f\pl_t\alpha) \big) \\
&\qquad - (\frac{r}{2} + 1 - \frac{g'}{2\Delta} + \frac{f''g'-f'g''}{8\Delta})\beta - \frac{i}{2\Delta}(-g'\pl_\phi\beta + f'\pl_t\beta)
\end{aligned}\right.\end{align}
where $\Delta = \oh(f'g - fg')$.  The volume form is $\Delta\dd\rho\wedge\dd\phi\wedge\dd t$.  Note that the Dirac operator is invariant under the two $S^1$-actions in $e^{i\phi}$ and $e^{it}$.

At $\rho = 0$, the coordinate has a singularity, and we shall use the rectangular coordinate $z = \rho e^{i\phi}$.  Remember that $f=\rho^2$ and $g=2-\rho^2$ when $0\leq\rho\leq10\epsilon$.  The Dirac operator is
\begin{align}\label{eqn_Dirac_binding_00}
D_r\psi = \left\{\begin{aligned}
&\frac{r}{2}\alpha + \frac{i}{2}(\pl_\phi\alpha + \pl_t\alpha) + (-2\pl_z\beta - \frac{i}{2}\bar{z}(\pl_\phi\beta + \pl_t\beta) - \frac{\bar{z}}{2}\beta) ~, \\
&(2\pl_{\bar{z}}\alpha - \frac{i}{2}z(\pl_\phi\alpha + \pl_t\alpha)) - \frac{r+3}{2}\beta - \frac{i}{2}(\pl_\phi\beta + \pl_t\beta) ~.
\end{aligned}\right. \end{align}
The volume form is $2\rho\dd\rho\wedge\dd\phi\wedge\dd t = i\dd z\wedge\dd\bar{z}\wedge\dd t$.  We can regard (\ref{eqn_Dirac_binding_00}) as an operator on $\mathbb{C}\times S^1$.

%%%%
\subsection{The Dehn-twist region} \label{subsec_Dehn_00}
For the Dehn-twist region, we work on $(-35\epsilon,35\epsilon)\times S^1\times[0,2\pi]$, and consider the operators and functions which are invariant under the identification map.  Before making the identification, we have coordinate $(\rho,e^{it},\phi)\in (-35\epsilon,35\epsilon)\times S^1\times [0,2\pi]$.  With this understood, the contact form is given by (\ref{eqn_contact_Dehn_01}),
\begin{align*}
a &= V\dd\phi + 2(v-\rho)\dd t + 2\phi(v-\rho)\sigma'\dd\rho ~.
\end{align*}
The following two coframes
\begin{align*}
\omega^1 &= \cos\phi\dd\rho - \sin\phi(-\dd t - \phi \sigma'\dd\rho - (v-\rho)\sigma'\dd\phi) ~, \\
\omega^2 &= \sin\phi\dd\rho + \cos\phi(-\dd t - \phi \sigma'\dd\rho - (v-\rho)\sigma'\dd\phi)
\end{align*}
together with the contact form specify the Riemannian metric, $a^2 + (\omega^1)^2 + (\omega^2)^2$.  Note that $\omega^1$ and $\omega^2$ are both smooth and invariant under the identification map.  As explained in \ref{subsubsec_sf_contact_00}, their dual frame gives a trivialization of $\underline{\mathbb{C}}\oplus K^{-1}$.  On the Dehn-twsit region, we use this trivialization to identify the sections of $\underline{\mathbb{C}}\oplus K^{-1}$ with $\mathbb{C}^2$-valued functions on $(-35\epsilon,35\epsilon)\times S^1\times [0,2\pi]$ satisfying
$$  \psi(\rho,e^{i(t+2\pi\sigma(\rho))},0) = \psi(\rho,e^{it},2\pi) ~.  $$
By (\ref{eqn_can_Dirac_00}) and (\ref{eqn_can_Dirac_01}), the Dirac operator $D_r$ on $\psi=(\alpha,\beta)$ is
\begin{align*} D_r\psi =
\left\{\begin{aligned}
&\frac{r}{2}\alpha + \frac{i}{\tilde{\Delta}}\big(\pl_\phi\alpha - (v-\rho)\sigma'\pl_t\alpha\big) + e^{-i\phi}\frac{2(v-\rho) - \tilde{\Delta}'}{\tilde{\Delta}}\beta \\
&\qquad + e^{-i\phi}\big( -(\pl_{\rho}\beta - \phi \sigma'\pl_t\beta) + \frac{iV}{\tilde{\Delta}}(-\pl_t\beta + \frac{2(v-\rho)}{V}\pl_\phi\beta) \big) ~, \\
&e^{i\phi}\big( (\pl_{\rho}\alpha - \phi \sigma'\pl_t\alpha) + \frac{iV}{\tilde{\Delta}}(-\pl_t\alpha + \frac{2(v-\rho)}{V}\pl_\phi\alpha) \big) \\
&\qquad - (\frac{r}{2} + 1 + \frac{1}{\tilde{\Delta}} + \frac{((v-\rho)\sigma)''}{2\tilde{\Delta}})\beta - \frac{i}{\tilde{\Delta}}(\pl_\phi\beta - (v-\rho)\sigma'\pl_t\beta)
\end{aligned}\right.\end{align*}
where $\tilde{\Delta} = V - 2(v-\rho)^2\sigma'$.  The volume form is $\tilde{\Delta}\dd\phi\wedge\dd t\wedge\dd\rho$.  Consider the \emph{untwisting} of $\alpha$ and $\beta$:
\begin{align}\label{def_untwisting}
\tilde{\alpha}(\rho, t, \phi) &= \alpha(\rho,t-\phi{\sigma(\rho)}, \phi) ~, & \tilde{\beta}(\rho, t, \phi) &= \beta(\rho, t-\phi{\sigma(\rho)}, \phi) ~.
\end{align}
After the untwisting, the Dirac operator on $\tilde{\psi}=(\tilde{\alpha},\tilde{\beta})$ is
\begin{align}\label{eqn_Dirac_Dehn_00}
\tilde{D}_r\tilde{\psi} = \left\{\begin{aligned}
&\frac{r}{2}\tilde{\alpha} + \frac{i}{\tilde{\Delta}}\big(\pl_\phi\tilde{\alpha} - ((v-\rho)\sigma)'\pl_t\tilde{\alpha}\big)  + e^{-i\phi}\frac{2(v-\rho) - \tilde{\Delta}'}{\tilde{\Delta}}\tilde{\beta} \\
&\qquad + e^{-i\phi}\big( -\pl_{\rho}\tilde{\beta} + \frac{i}{\tilde{\Delta}}((-V + 2(v-\rho)\sigma)\pl_t\tilde{\beta} + 2(v-\rho)\pl_\phi\tilde{\beta}) \big) ~, \\
&e^{i\phi}\big( \pl_{\rho}\tilde{\alpha} + \frac{i}{\tilde{\Delta}}((-V + 2(v-\rho)\sigma)\pl_t\tilde{\alpha} + 2(v-\rho)\pl_\phi\tilde{\alpha}) \big) \\
&\qquad - (\frac{r}{2} + 1 + \frac{1}{\tilde{\Delta}} + \frac{((v-\rho)\sigma)''}{2\tilde{\Delta}})\tilde{\beta} - \frac{i}{\tilde{\Delta}}(\pl_\phi\tilde{\beta} - ((v-\rho)\sigma)'\pl_t\tilde{\beta}) ~.
\end{aligned}\right.\end{align}

The untwisting operator (\ref{def_untwisting}) and the Dirac equation (\ref{eqn_Dirac_Dehn_00}) have the following features.
\begin{enumerate}
\item The untwisting operator (\ref{def_untwisting}) is only defined locally.  In general, it cannot extend to the whole $3$-manifold.\smallskip
\item Note that $\tilde{\alpha}(\rho,t,2\pi) = \alpha(\rho,t-2\pi \sigma(\rho),2\pi) = \alpha(\rho,t,0) = \tilde{\alpha}(\rho,t,0)$.  Hence, after the untwisting $\tilde{\alpha}$ and $\tilde{\beta}$ are $2\pi$-periodic in both $t$ and $\phi$.  With this understood, the Dirac operator (\ref{eqn_Dirac_Dehn_00}) is invariant under these two $S^1$-actions.  Namely, for any $(e^{it_0},e^{i\phi_0})\in S^1\times S^1$,
\begin{align*}  \tilde{D}_r\big(\tilde{\psi}(\rho,t+t_0,\phi+\phi_0)\big) = (\tilde{D}_r\tilde{\psi})(\rho,t+t_0,\phi+\phi_0) ~.  \end{align*}
\item The Dirac operator (\ref{eqn_Dirac_Dehn_00}) is of the same form as that on the tubular neighborhood of the binding (\ref{eqn_Dirac_neck_01}).  More precisely, it corresponds to $\tilde{f} = V - 2(v-\rho)\sigma$ and $\tilde{g} = 2(v-\rho)$ in (\ref{eqn_Dirac_neck_01}), and $\tilde{\Delta}$ is equal to $\oh(\tilde{f}'\tilde{g}-\tilde{f}\tilde{g}')$.
\end{enumerate}

%%%%
\subsection{Associated contact form on $S^1\times S^2$}\label{subsec_asso_contact_00}
In this section, we construct $S^1\times S^2$'s by compatifying the tubular neighborhood of the binding and the Dehn-twist region.  We also construct contact forms on these $S^1\times S^2$, associated to (\ref{eqn_contact_binding_00}) and (\ref{eqn_contact_Dehn_01}).  It is convenient to regard the Dirac operators in section \ref{subsec_binding_00} and \ref{subsec_Dehn_00} as being defined on $S^1\times S^2$.

\subsubsection{On the tubular neighborhood of the binding.}  The tubular neighborhood of the binding is a solid torus, $S^1\times D^2$.  Topologically, we construct the $S^1\times S^2$ by attaching another solid torus, $S^1\times D^2$.  The attaching map on $\pl(S^1\times D^2) = S^1\times S^1$ is the identity map.

We now describe the $S^1\times S^2$ in terms of the coordinate.  Let $e^{it}$ be the coordinate for the $S^1$-factor.  Let $(\rho,e^{i\phi})\in[0,2]\times S^1$ be the (re-parametrized) spherical coordinate for the $S^2$-factor.  To be more precise, choose a positive smooth function $\nu(\rho)$ in $\rho\in[0,2]$ such that $\nu(\rho)=1$ when $\rho\leq10\epsilon$ or $\rho\geq2-10\epsilon$, and $\int_0^2\nu(\rho)\dd\rho=\pi/2$.  The parametrization of the unit sphere is given by $\big( \sin(\int_0^\rho\nu(s)\dd s)\cos\phi,\sin(\int_0^\rho\nu(s)\dd s)\sin\phi, \cos(\int_0^\rho\nu(s)\dd s) \big)$.

To construct the contact form, choose two smooth functions $f(\rho)$ and $g(\rho)$ in $\rho\in[0,2]$ such that:
\begin{itemize}
\item when $\rho\in[0,1]$, the functions $f(\rho)$ and $g(\rho)$ coincide with the functions constructed in section \ref{subsec_open_book_00};
\item when $\rho\in(1,1+10\epsilon]$, $f(\rho) = V$ and $g(\rho) = 2-\rho$;
\item when $\rho\in[2-10\epsilon,2]$, $f(\rho) = (2-\rho)^2$ and $g(\rho) = -2 + (2-\rho)^2$;
\item the functions $f$ and $f'g-fg'$ are positive when $\rho\in(0,2)$.
\end{itemize}
It is not hard to see that there always exist such $f$ and $g$.  The contact form is taken to be
\begin{align} \label{eqn_S2S1_contact_00}
a &= f(\rho)\dd\phi + g(\rho)\dd t ~.
\end{align}
The above conditions on $f$ and $g$ guarantee that $a$ is a smooth contact form.  The volume form $\frac{1}{2}a\wedge\dd a$ is equal to $\Delta \dd\rho\wedge\dd\phi\wedge\dd t$ where
\begin{align}\label{eqn_S2S1_Delta_00}
\Delta = \oh(f'g-fg') ~.
\end{align}

\subsubsection{On the Dehn-twist region.}\label{subsec_asso_contact_01}  After the untwisting (\ref{def_untwisting}), the Dehn-twist region is the solid torus $(e^{it},\rho,e^{i\phi})\in S^1\times[-35\epsilon,35\epsilon]\times S^1$.  Topologically, $S^1\times S^2$ is obtained by attaching two solid torus, $S^1\times D^2$.  The original solid tori has two boundary components, $S^1\times\{\pm35\epsilon\}\times S^1$.  The attaching map is the identity map on $S^1\times S^1$ .

In this case, we take a similar coordinate on the $S^1\times S^2$.  Let $e^{it}$ be the coordinate for the $S^1$-factor.  Let $(\rho,e^{i\phi})\in[-2,2]\times S^1$ be the (re-parametrized) spherical coordinate for the $S^2$-factor.

To construct the contact form, choose two smooth functions $\tilde{f}(\rho)$ and $\tilde{g}(\rho)$ in $\rho\in[0,2]$ satisfying the following properties:
\begin{itemize}
\item when $\rho\in(-35\epsilon,35\epsilon)$, $\tilde{f}(\rho) = V - 2(v-\rho)\sigma(\rho)$ and $\tilde{g}(\rho) = 2(v-\rho)$;
\item when $\rho\in[2-10\epsilon,2]$, $\tilde{f}(\rho) = (2-\rho)^2$ and $\tilde{g}(\rho) = -2|v|-2+(2-\rho)^2$;
\item when $\rho\in[-2,-2+10\epsilon]$, $\tilde{f}(\rho) = (\rho+2)^2$ and $\tilde{g}(\rho) = 2|v|+2-(\rho+2)^2$;
\item the functions $\tilde{f}$ and $\tilde{f}'\tilde{g}-\tilde{f}\tilde{g}'$ are positive when $\rho\in(-2,2)$.
\end{itemize}
With these two functions, the contact form is taken to be $a = \tilde{f}(\rho)\dd\phi+\tilde{g}(\rho)\dd t$.  The volume form $\frac{1}{2}a\wedge\dd a$ is $\tilde{\Delta}\dd\rho\wedge\dd\phi\wedge\dd t$ where $\tilde{\Delta} = \frac{1}{2}(\tilde{f}'\tilde{g}-\tilde{f}\tilde{g}')$.

\subsubsection{The canonical $\text{\rm Spin}^{\mathbb{C}}$-structure of the associated contact form}

\begin{defn}
For each boundary component and Dehn-twist region of the page, the above construction gives a contact form (\ref{eqn_S2S1_contact_00}) on $S^1 \times S^2$.  These contact forms will be referred as the \emph{associated contact forms}.
\end{defn}

We now choose an adapted metric and fix a trivialization of $\underline{\mathbb{C}}\oplus K^{-1}$ of the associated contact form on the $S^1\times S^2$.  We focus on the associated contact form of the tubular neighborhood of the binding.  For the associated contact form of the Dehn-twist region, the construction is essentially the same.  The only difference is that $f$ and $g$ are replaced by $\tilde{f}$ and $\tilde{g}$.

Consider the following two coframes
\begin{align*}
\omega^1 &= \cos\phi\dd\rho - \sin\phi(\frac{f'}{2}\dd\phi + \frac{g'}{2}\dd t) ~, &\omega^2 &= \sin\phi\dd\rho + \cos\phi(\frac{f'}{2}\dd\phi + \frac{g'}{2}\dd t) ~.
\end{align*}
As explained in section \ref{subsec_binding_00}, they are smooth on the whole space, $S^1\times S^2$.  The adapted metric is taken to be $a^2+(\omega^1)^2+(\omega^2)^2$.  Let $\{e_1,e_2,e_3\}$ be the dual frame of $\{\omega^1,\omega^2,a\}$.  According to section \ref{subsubsec_sf_contact_00}, $\{{\bf1}_{\underline{\mathbb{C}}},\frac{1}{\sqrt{2}}(e_1-ie_2)\}$ induces a \emph{global} trivialization of the canonical $\text{\rm Spin}^{\mathbb{C}}$-bundle $\underline{\mathbb{C}}\oplus K^{-1}$.  When working with the Dirac operator on the associated $S^1\times S^2$, we will always use this trivialization to identify the sections of $\underline{\mathbb{C}}\oplus K^{-1}$ with $\mathbb{C}^2$-valued functions on $S^1\times S^2$.

With such a choice of the metric and the trivialization, the local expression of the Dirac operator is given by (\ref{eqn_Dirac_neck_01}) and (\ref{eqn_Dirac_Dehn_00}), respectively.  We will study the Dirac equations of the associated contact forms carefully in section \ref{sec_S2S1}.

%%%%
\subsection{The part with trivial monodromy}\label{sec_page_id_00}
On the part of the page where the monodromy is the identity map, the contact form is
\begin{align*}
a = V\dd\phi + 2\mu_\Sigma ~.
\end{align*}
Choose a Riemannian metric $\dd s^2_\Sigma$ on $\Sigma$ such that
\begin{itemize}
\item the area form is $\dd\mu_\Sigma$;\smallskip
\item near the tubular neighborhood of the binding, $\dd s^2_\Sigma = \dd\rho^2 + \frac{1}{4}\dd t^2$ in terms of the coordinates in section \ref{subsec_binding_00};\smallskip
\item near the Dehn-twist region, $\dd s^2_\Sigma = \dd\rho^2 + \dd t^2$ in terms of the coordinates in section \ref{subsec_Dehn_00}.
\end{itemize}
The Riemannian metric on $\Sigma\times S^1$ is taken to be $a^2 + \dd s^2_\Sigma$.  Near the binding and the Dehn-twist region, it is not hard to check that this metric agrees with the metric defined in \ref{subsec_binding_00} and \ref{subsec_Dehn_00}.

Any locally defined, oriented orthonormal frame on $\Sigma$, $u_1$ and $u_2$, gives rise to the following frame for the contact hyperplane on the $3$-manifold:
\begin{align*}\begin{split}
e_1 &= \cos\phi u_1 - \sin\phi u_2 - \frac{2}{V}\mu_\Sigma(\cos\phi u_1 - \sin\phi u_2)\pl_\phi ~, \\
e_2 &= \sin\phi u_1 + \cos\phi u_2 - \frac{2}{V}\mu_\Sigma(\sin\phi u_1 + \cos\phi u_2)\pl_\phi ~.
\end{split}\end{align*}
As discussed in \ref{subsubsec_sf_contact_00}, the dual frame induces a trivialization of $\underline{\mathbb{C}}\oplus K^{-1}$.  By examining the transition function of $K^{-1}$ in terms of this trivialization, we find that $K^{-1} = \pi^* K_{\Sigma}^{-1}$, where $K_{\Sigma}^{-1}$ is the anti-canonical bundle of $\Sigma$ determined by the metric and $\dd\mu_\Sigma$.  More precisely, let $\theta^1$ and $\theta^2$ be the dual coframe of $u_1$ and $u_2$ on $\Sigma$.  Then $e_1-ie_2$ is identified as $\theta^1-i\theta^2$.  By (\ref{eqn_can_Dirac_00}) and (\ref{eqn_can_Dirac_01}), the Dirac operator $D_r$ on $\psi=(\alpha,\beta)$ is
\begin{align*}
D_r\psi = \left\{\begin{aligned}
&\frac{r}{2}\alpha + \frac{i}{V}\pl_\phi\alpha + e^{-i\phi}\big( -u_1(\beta) + iu_2(\beta) - \frac{2}{V}\mu_\Sigma(-u_1+iu_2)\pl_\phi\beta \\
&\qquad\qquad\qquad\qquad\qquad\qquad + \frac{2i}{V}\mu_\Sigma(-u_1+iu_2)\beta + i\theta_1^2(-u_1+iu_2)\beta \big) ~, \\
&e^{i\phi}\big( u_1(\alpha) + iu_2(\alpha) -\frac{2}{V}\mu_\Sigma(u_1+iu_2)\pl_\phi\alpha \big) - (\frac{r}{2} + 1 + \frac{1}{V})\beta - \frac{i}{V}\pl_\phi\beta
\end{aligned}\right.\end{align*}
where $\theta_1^2$ is the Levi-Civita connection for the metric $\dd s^2_\Sigma$ on $\Sigma$; namely, $\nabla e_1 = \theta_1^2\otimes e_2$.

Consider the separation of variables:
\begin{align}\label{eqn_Fourier_page_00} \alpha &= \alpha_{n} e^{in\phi}(2\pi V)^{-\oh}, &\beta &= \beta_n e^{i(n+1)\phi}(2\pi V)^{-\oh}. \end{align}
Before separation of variables, $\alpha$ is a function on $S^1\times\Sigma$, and $\beta$ is a section of $K^{-1}$ over $S^1\times\Sigma$.  After the separation of variable, $\alpha_n$ is a function on $\Sigma$, and $\beta_n$ is a section of $K_{\Sigma}^{-1}$ over $\Sigma$.  The Dirac operator on the frequency $n$ components is
\begin{align}\label{eqn_Dirac_page_00}\left\{\begin{aligned}
&(\frac{r}{2} - \frac{n}{V})\alpha_n + \bar{\pl}^*_n \beta_n ~, \\
&\bar{\pl}_n \alpha_n - (\frac{r}{2} + 1 - \frac{n}{V})\beta_n 
\end{aligned}\right.\end{align}
where $\bar{\pl}_n$ and $\bar{\pl}^*_n$ are the Cauchy--Riemann operators on $\underline{\mathbb{C}}\oplus K_\Sigma^{-1}$ with the connection perturbed by $-\frac{2in}{V}\mu_\Sigma$.  Subject to suitable boundary conditions, their index is given by \cite[(4.3)]{ref_APS}.  The boundary conditions will be explained later.  By the computation in \cite[p.148-149]{ref_BGV} and \cite[(4.5)]{ref_APS}, the characteristic class term can be expressed in terms of the curvature of the connection and the Euler characteristic of $\Sigma$.  The formula reads
\begin{align}\label{eqn_APS_00}
\dim\ker\bar{\pl}_n - \dim\ker\bar{\pl}^*_n = \frac{n}{V\pi}\int\!\!\!\int_\Sigma \dd\mu_\Sigma + \oh\chi(\Sigma) + \oh(\eta_n+h_n)(\pl\Sigma)
\end{align}
where $\chi(\Sigma)$ is the Euler characteristic of $\Sigma$, and $\eta_n(\pl\Sigma)$ and $h_n(\pl\Sigma)$ are the correction terms from the boundary.  Since the boundary of $\Sigma$ is a disjoint union of \emph{circles}, the correction terms $\eta_n(\pl\Sigma)$ and $h_n(\pl\Sigma)$ are uniformly bounded for all $n$.

In \cite{ref_APS}, the connection is required to depend only on $\partial\Sigma$ in a small neighborhood of $\pl\Sigma$.  With the notation in subsections \ref{subsec_binding_00} and \ref{subsec_Dehn_00}, the connection is required to be independent of $\rho$ near $\pl\Sigma$.  Our connection $\mu_\Sigma$ does not satisfy this property.  However, $\mu_\Sigma$ is affine in $\rho$, and $\pl\Sigma$ is a disjoint union of $S^1$'s.  With a slight modification, the index formula (\ref{eqn_APS_00}) still holds in our setting.  We will explain the modification at the end of this section. 

In \cite{ref_APS}, there are adjoint boundary conditions for $\bar{\pl}_n$ and $\bar{\pl}_n^*$.  
On the collar neighborhood of the boundary,
\begin{align*}
\bar{\pl}_n = \pl_\rho + {\pl \hskip -2.2mm \slash_n} \qquad\text{ and }\qquad 
\bar{\pl}_n^* = -\pl_\rho + {\pl \hskip -2.2mm \slash_n} ~.
\end{align*}
where $\rho$ is the coordinate transverse to the boundary.  The operator ${\pl \hskip -2.2mm \slash_n}$ is the restriction of $\bar{\pl}_n$ and $\bar{\pl}_n^*$ on the boundary, and it is a Dirac operator on the boundary.  The Atiyah--Patodi--Singer (APS for short) boundary condition says that the restriction of $\alpha_n$ on the boundary only has components in the negative eigenspaces of ${\pl \hskip -2.2mm \slash_n}$, and $\beta_n$ only has non-negative ones.  We now give an explicit description of the boundary conditions.

\subsubsection*{Adjacent to the tubular neighborhood of the binding.}  We follow the notations in section \ref{subsec_binding_00}, and take $u_1 = \pl_\rho$, $u_2 = 2\pl_t$.  The surface $\Sigma$ is described by $\rho\geq 1$.  The operators $\bar{\pl}_n$ and $\bar{\pl}^*_n$ are
\begin{align}\begin{split}\label{eqn_Dirac_bdry_neck_00}
\bar{\pl}_n \alpha_n &= \pl_\rho\alpha_n - 2i\pl_t \alpha_n - \frac{2n(2-\rho)}{V}\alpha_n ~, \\
\bar{\pl}^*_n \beta_n &= -\pl_\rho\beta_n - 2i\pl_t \beta_n - \frac{2n(2-\rho)}{V}\beta_n ~,
\end{split}\end{align}
and ${\pl \hskip -2.2mm \slash_n}$ is $-2i\pl_t - \frac{2n}{V}$.  Let $\alpha_n = \alpha_{n,m}(\rho)e^{imt}$ and $\beta_n = \beta_{n,m}(\rho)e^{imt}$.  The APS boundary condition is
\begin{align}\label{eqn_APS_bdry_00}\left\{\begin{aligned}
\alpha_{n,m}(1) = 0 &\qquad\text{ when } m\geq\frac{n}{V} ~,\\
 \beta_{n,m}(1)=0 &\qquad\text{ when } m<\frac{n}{V} ~.
\end{aligned}\right.\end{align}

\subsubsection*{Adjacent to the Dehn-twist region.}  We follow the notations in section \ref{subsec_Dehn_00}, and take $u_1 = \pl_\rho$, $u_2 = \pl_t$.  The region $\Sigma$ is the union of where $\rho\geq 20\epsilon$ and $\rho\leq-20\epsilon$.  The operators $\bar{\pl}_n$ and $\bar{\pl}^*_n$ are
\begin{align}\begin{split}\label{eqn_Dirac_bdry_Dehn_00}
\bar{\pl}_n \alpha_n &= \pl_\rho\alpha_n - i\pl_t \alpha_n - \frac{2n(v-\rho)}{V}\alpha_n ~, \\
\bar{\pl}^*_n \beta_n &= -\pl_\rho\beta_n - i\pl_t \beta_n - \frac{2n(v-\rho)}{V}\beta_n ~.
\end{split}\end{align}
At $\rho = 20\epsilon$, ${\pl \hskip -2.2mm \slash_n}$ is $-i\pl_t - \frac{2n(v-20\epsilon)}{V}$.  At $\rho = -20\epsilon$, since $\pl_\rho$ does not point inward, ${\pl \hskip -2.2mm \slash_n}$ is $i\pl_t + \frac{2n(v+20\epsilon)}{V}$.  Let $\alpha_n = \alpha_{n,m}(\rho)e^{imt}$ and $\beta_n = \beta_{n,m}(\rho)e^{imt}$.  The APS boundary condition at $\rho = 20\epsilon$ is
\begin{align}\label{eqn_APS_bdry_01}\left\{\begin{aligned}
\alpha_{n,m}(20\epsilon) = 0 &\qquad\text{ when } m\geq2n\frac{v -20\epsilon}{V} ~, \\
\beta_{n,m}(20\epsilon)=0 &\qquad\text{ when } m<2n\frac{v-20\epsilon}{V} ~.
\end{aligned}\right.\end{align}
The condition at $\rho = -20\epsilon$ is
\begin{align}\label{eqn_APS_bdry_02}\left\{\begin{aligned}
\alpha_{n,m}(-20\epsilon) = 0 &\qquad\text{ when } m\leq2n\frac{v +20\epsilon}{V} ~, \\
\beta_{n,m}(-20\epsilon)=0 &\qquad\text{ when } m>2n\frac{v+20\epsilon}{V} ~.
\end{aligned}\right.\end{align}\\

In order to use the index formula (\ref{eqn_APS_00}) to compute $\dim\ker\bar{\pl}_n$, we need to know that $\dim\ker\bar{\pl}^*_n = 0$.
\begin{lem}\label{lem_beta_vanish_00}
There exists a constant $c$ such that the following holds.  For all $n\geq c$, if $\beta_n$ satisfies the APS boundary condition,
\begin{align*}
\int_\Sigma|\beta_n|^2 \leq c n^{-1} \int_\Sigma|\bar{\pl}_n^*\beta_n|^2 ~.
\end{align*}
In particular, $\bar{\pl}^*_n$ only has trivial solution for any $n\geq c$.
\end{lem}
\begin{proof}
The integration by parts formula gives
\begin{align*}
\int_\Sigma |\bar{\pl}^*_n\beta_n|^2 &= \int_\Sigma |\nabla_n \beta_n|^2 + \int_\Sigma(\frac{2n}{V} + \frac{\kappa_\Sigma}{4}) |\beta_n|^2 + \int_{\partial\Sigma}\langle {\pl \hskip -2.2mm \slash_n}\beta_n, \beta_n \rangle 
\end{align*}
where $\kappa_\Sigma$ is the scalar curvature.  The APS boundary condition for $\beta_n$ implies that $\int_{\pl\Sigma}\langle {\pl \hskip -2.2mm \slash_n}\beta_n,\beta_n\rangle$ is non-negative.  Hence, if $n \geq V\max{|\kappa_\Sigma|}$, we obtain the inequality of the lemma.
\end{proof}

For any integer $n\geq c$ in lemma \ref{lem_beta_vanish_00}, the dimension of $\ker\bar{\pl}_n$ is given by the right hand side of (\ref{eqn_APS_00}).  Solutions of $\bar{\pl}_n$ automatically solve the Dirac equation (\ref{eqn_Dirac_page_00}) with $r$ given by
\begin{align} \label{defn_gamma_n} \gamma_n = \frac{2n}{V} ~. \end{align}
However, they only solve the Dirac equation on $\Sigma\times S^1$.  In order to get smooth sections on the $3$-manifold $Y$, we need to do some modifications.

\begin{defn}\label{defn_Sigma_ext_curt}
For any $\delta\in(-15\epsilon,5\epsilon)$, let $\Sigma_\delta$ be the extension/curtailment of $\Sigma$ defined by
\begin{itemize}
\item $\{\rho\geq 1-\delta\}$ for the part adjacent to the tubular neighborhood of the binding, in terms of the coordinate in section \ref{subsec_binding_00}.\smallskip
\item $\{\rho\leq -20\epsilon+\delta\text{ or }\rho\geq20\epsilon-\delta\}$ for the part adjacent to the Dehn-twist region, in terms of the coordinate in section \ref{subsec_Dehn_00}.
\end{itemize}
Positive $\delta$ corresponds to the extension, and negative $\delta$ corresponds to the curtailment.  When $\delta=0$, $\Sigma_0=\Sigma$.
\end{defn}

\begin{defn}
Let $\chi_\Sigma$ be the cut-off function which is equal to $1$ on $\Sigma_{\epsilon}\times S^1$ and equal to $0$ on $Y\backslash(\Sigma_{2\epsilon}\times S^1)$, and only depends on $\rho$ over $(\Sigma_{2\epsilon}\backslash\Sigma_{\epsilon})\times S^1$ in terms of the coordinate in sections \ref{subsec_binding_00} and \ref{subsec_Dehn_00},
\end{defn}

Suppose that $\alpha_n$ solves $\bar{\pl}_n$.  On the part adjacent to the tubular neighborhhood of the binding, $\alpha_n$ is equal to
\begin{align}\label{eqn_sln_extension_00}
\sum_{m<\frac{n}{V}} \mathfrak{c}_{n,m} \exp\big( -\frac{n}{V}(\rho-2+\frac{Vm}{n})^2\big) e^{imt}
\end{align}
where $\mathfrak{c}_{n,m}$ are constants.  The expression also solves $\bar{\pl}_n$ on the region where $\rho\geq1-2\epsilon$, and it also obeys the corresponding APS boundary condition.  On the part adjacent to the Dehn-twist region, the situation is similar.  Therefore, any solution of $\bar{\pl}_n$ on $\Sigma$ can be extended uniquely to a solution on $\Sigma_{2\epsilon}$, and the extension obeys the corresponding APS boundary condition on $\Sigma_{2\epsilon}$.  When $m<\frac{n}{V}$, the function $(\rho-2+\frac{Vm}{n})^2$ is monotone decreasing for $\rho\in(1-2\epsilon, 1]$.  It implies that
\begin{align*}  \int_{\Sigma_{2\epsilon}\backslash\Sigma}|\alpha_n|^2 \leq \int_{\Sigma\backslash\Sigma_{-2\epsilon}}|\alpha_n|^2 \leq \int_\Sigma |\alpha_n|^2 ~, \end{align*}
and thus the extension is still square integrable.

Consider the following construction of the almost eigensections: for each solution of $\bar{\pl}_n$, extend it to $\Sigma_{2\epsilon}$, and multiply it by the cut-off function $\chi_\Sigma$.  This process is linear, and it ends up with a vector space of the same dimension as $\ker{\bar{\pl}_n}$.  Choose an orthonormal basis with  respect to the $L^2$-inner product on $\Sigma_{2\epsilon}$.  Denote the basis by $\{\xi_{n,l}\}$, where $l$ runs from $1$ to the number on the right hand side of (\ref{eqn_APS_00}).  They are smooth functions on $Y$.  Their properties are summarized in the following proposition.

\begin{prop}\label{prop_page_sln_00}
There exists a constant $c$ which has the following significance:  For any integer $n\geq c$, let $\psi_{n,l}$ be the section of $\underline{\mathbb{C}}\oplus K^{-1}$ over $\Sigma\times S^1$ whose first component is $$\xi_{n,l}e^{in\phi}(2\pi V)^{-\oh}$$ and second component is zero.  Here, $\xi_{n,l}$ is given by the above construction.  Then,
\begin{align*}
\int_Y |D_r \psi_{n,l} - \frac{r-\gamma_n}{2}\psi_{n,l}|^2\leq c\exp(-\frac{n}{c})
\end{align*}
for any $r>0$, and $\gamma_n$ is defined by (\ref{defn_gamma_n}).  Moreover,
\begin{align*}
&\int_Y\langle\psi_{n,l},\psi_{n,l'}\rangle = \int_Y \langle D_r\psi_{n,l}, \psi_{n,l'} \rangle = 0 ~, \\
&\big| \int_Y \langle D_r\psi_{n,l}, D_r\psi_{n,l'} \rangle \big| \leq c\exp(-\frac{n}{c})
\end{align*}
for any $l\neq l'$.
\end{prop}
\begin{proof}
For each section $\psi_{n,l}$, there exists a function $\alpha_{n,l}$ which solves $\bar{\pl}_n$ on $\Sigma_{2\epsilon}$ and is extended from $\Sigma$, such that $\xi_{n,l} = \chi_{\Sigma}\alpha_{n,l}$.  From the expression (\ref{eqn_sln_extension_00}), there exists a constant $c_1$ such that
\begin{align}\label{eqn_bdry_ext_small_00}
\int_{\Sigma_{2\epsilon}\backslash\Sigma_{\epsilon}}|\alpha_{n,l}|^2 \leq c_1\exp(-\frac{n}{c_1})\int_{\Sigma_{\epsilon}}|\alpha_{n,l}|^2 ~.
\end{align}
By (\ref{eqn_Dirac_page_00}), the first component of $D_r\psi_{n,l}$ is
\begin{align*}
\frac{r-\gamma_n}{2}\chi_\Sigma\alpha_{n,l}e^{in\phi}(2\pi V)^{-\oh} = \frac{r-\gamma_n}{2}\xi_{n,l}e^{in\phi}(2\pi V)^{-\oh} ~.
\end{align*}
The second component of $D_r\psi_{n,l}$ is equal to
\begin{align*}
\chi'_\Sigma\alpha_{n,l}e^{i(n+1)\phi}(2\pi V)^{-\oh}
\end{align*}
which is supported on $\Sigma_{2\epsilon}\backslash\Sigma_{\epsilon}$.  Since $\{\chi_\Sigma\alpha_{n,l}\} =\{\xi_{n,l}\}$ forms an orthonormal set in $L^2(\Sigma_{2\epsilon})$, the above expression of $D_r\psi_{n,l}$ with (\ref{eqn_bdry_ext_small_00}) proves the proposition.
\end{proof}

\subsubsection{A remark on the APS index theorem.}
We now explain why the index formula (\ref{eqn_APS_00}) still holds in our setting.  For simplicity, we only emphasize it for the boundary component adjacent to the tubular neighborhood of the binding.  To start, choose a smooth function $h(\rho)$ in $\rho\in[1,1+20\epsilon)$ such that
\begin{itemize}
\item $h(\rho) = 1$ for $\rho\in[1,1\epsilon]$;
\item $h(\rho) = 2-\rho$ for $\rho\in[1+2\epsilon,1+20\epsilon)$;
\item $h(\rho)$ is non-increasing in $\rho$.
\end{itemize}
Let $\mu_h$ be the $1$-form which is equal $h(\rho)\dd t$ near the boundary of $\Sigma$, and is equal to $\mu_\Sigma$ away from the boundary of $\Sigma$.  Stokes theorem implies that $\int\!\int_{\Sigma}\dd\mu_\Sigma = \int\!\int_\Sigma\dd\mu_h$.

We take the same metric $\dd s^2_\Sigma$ on $\Sigma$.  Let $\bar{\pl}_{n,h}$ and $\bar{\pl}_{n,h}^*$ be the Cauchy--Riemann operators on $\underline{\mathbb{C}}\oplus K_{\Sigma}^{-1}$ with the connection perturbed by $\frac{-2in}{V}\mu_h$.  Since $h(\rho)$ is equal to $1$ on the tubular neighborhood of $\pl\Sigma$, it meets the requirement of \cite[(4.3)]{ref_APS}.  Thus, the index is given by the right hand side of (\ref{eqn_APS_00}).

We claim that $\dim\ker\bar{\pl}_n = \dim\ker\bar{\pl}_{n,h}$.  The operators $\bar{\pl}_n$ and $\bar{\pl}_{n,h}$ share the same boundary condition, as described by (\ref{eqn_APS_bdry_00}).  For any solution $\alpha_n$ of $\bar{\pl}_n$, it can be expressed as (\ref{eqn_sln_extension_00}) on $\Sigma\backslash\Sigma_{-20\epsilon}$.  Let $\alpha_{n,h}$ be equal to $\alpha_n$ on $\Sigma_{-20\epsilon}$, and equal to
\begin{align*}
\sum_{m<\frac{n}{V}} \mathfrak{c}_{n,m} \exp\big( -\frac{2n}{V}\int_{1+2\epsilon}^\rho(h(s)-\frac{Vm}{n})\dd s\big) \exp\big(-\frac{n}{V}(2\epsilon-1+\frac{Vm}{n})^2\big) e^{imt}
\end{align*}
on $\Sigma\backslash\Sigma_{-20\epsilon}$.  It is not hard to see that $\alpha_{n,h}$ is smooth and solves $\bar{\pl}_{n,h}$.  The construction of $\alpha_{n,h}$ from $\alpha_n$ gives a linear map from $\ker\bar{\pl}_n$ to $\ker\bar{\pl}_{n,h}$, which we denote by $\Pi_h$.  On the other hand, the inverse map $\Pi_h^{-1}$ is given by solving the ordinary differential equation on $\Sigma\backslash\Sigma_{-20\epsilon}$.  Therefore, $\ker\bar{\pl}_n$ and $\ker\bar{\pl}_{n,h}$ are isomorphic to each other.  A similar construction implies that $\ker\bar{\pl}^*_n$ and $\ker\bar{\pl}^*_{n,h}$ are isomorphic to each other.

%%%%%%%%
\section{The model case: $S^1\times S^2$}\label{sec_S2S1}
In this section, we study the Dirac equations of the associated contact forms (\ref{eqn_S2S1_contact_00}) on $S^1\times S^2$.  The $S^1\times S^2$ associated to the tubular neighborhood of the binding will be denoted by $\check{Y}$.  The Dirac operators will be denoted by $\check{D}_r$, and the sections will be denoted by $\check{\psi}$.  The $S^1\times S^2$ associated to the Dehn-twist region will be denoted by $\tilde{Y}$.  The Dirac operators will be denoted by $\tilde{D}_r$, and the sections will be denoted by $\tilde{\psi}$.  Their spectral flow function will be denoted by $\check{\rm sf}_a(r)$ and $\tilde{\rm sf}_a(r)$, respectively.  We will focus on the associated contact form of the tubular neighborhood of the binding.  For the associated contact form of the Dehn-twist region, the argument is completely parallel, and the details will be omitted.

Recall that we fix a {global} trivialization of $\underline{\mathbb{C}}\oplus K^{-1}$ to identify its sections with $\mathbb{C}^2$-valued functions.  The corresponding Dirac operator is given by (\ref{eqn_Dirac_neck_01}), and it is invariant under the two \emph{global} $S^1$-actions in $e^{i\phi}$ and $e^{it}$.  Hence, the eigenspaces of the Dirac operator split according to the frequencies with respect to these two $S^1$-actions.  The splitting allows us to study the spectral flow function directly.  Let $\mathcal{S}_{k,m}$ be the following space of sections:
\begin{align*}
\big\{ \psi=(\alpha,\beta) ~\big|~ \pl_\phi\psi = ik\psi+i(0,\beta),\pl_t\psi = im\psi \big\}.
\end{align*}

The following notions will be used throughout the paper.
\begin{defn}\label{def_gamma_nm_00}
For the associated contact form (\ref{eqn_S2S1_contact_00}) of the tubular neighborhood of the binding, the function $g/f$ is monotone decreasing in $\rho$.  For each positive integer $k$ and integer $m$, there is a unique $\check{\rho}_{k,m}\in(0,2)$ such that $k g(\check{\rho}_{k,m}) = m f(\check{\rho}_{k,m})$.  Let $\check{\gamma}_{k,m}$ be
\begin{align*} %\check{\gamma}_{k,m} =
\frac{mf'(\check{\rho}_{k,m})-kg'(\check{\rho}_{k,m})}{\Delta(\check{\rho}_{k,m})} = \frac{2k}{f(\check{\rho}_{k,m})} = \frac{2m}{g(\check{\rho}_{k,m})} \end{align*}
where $\Delta$ is defined by (\ref{eqn_S2S1_Delta_00}).  The last equality only makes sense at where $g(\check{\rho}_{k,m})\neq0$.  If $k=0$ and $m>0$, let $\check{\rho}_{k,m}=0$ and $\check{\gamma}_{k,m}-m$.  If $k=0$ and $m<0$, let $\check{\rho}_{k,m}=2$ and $\check{\gamma}_{k,m}=-m$.
\end{defn}

For the associated contact form of the Dehn-twist region, $\tilde{\gamma}_{k,m}$ is defined in the same way:  replace $f$, $g$ and $\Delta$ by $\tilde{f}$, $\tilde{g}$ and $\tilde{\Delta}$, and $\tilde{\rho}_{k,m}$ lies in the interval $[-2,2]$.  For $k=0$ and $m\neq0$, $\tilde{\gamma}_{k,m} = \frac{{\rm sign}(m)m}{|v|+1}$.

We will have various cut-off functions for different purpose.  They will be denoted by $\chi$ with some sub/superscript.  If there is no sub/superscript, it is the following one:
\begin{defn}
Let $\chi(x)$ be the cut-off function on $\mathbb{R}$ with $\chi(x)=1$ when $|x|\leq\oh$ and $\chi(x)=0$ when $|x|\geq1$.
\end{defn}

%%%%
\subsection{Uniqueness of zero crossing}
In order to prove the upper bound in theorem \ref{thm_sf_main_00} for the associated contact forms, we need to know that the zero crossing of $\check{D}_r$ on each $\mathcal{S}_{k,m}$ is unique.

\begin{prop} \label{prop_S2S1_unique_00}
For the associated contact form (\ref{eqn_S2S1_contact_00}) on $\check{Y}$, there exists a constant $c>0$ such that the following holds.
\begin{enumerate}
\item For each $k$ and $m$, the Dirac operator $\check{D}_r$ on $\mathcal{S}_{k,m}$ has at most one zero crossing for $r\geq c$.
\item If $k$ is negative, or both $k$ and $m$ are zero, there is no zero crossing for $r\geq c$.
\item If the Dirac operator does have a zero crossing on $\mathcal{S}_{k,m}$ at some $r\geq c$, then
$$ r\in[\check{\gamma}_{k,m}-c, \check{\gamma}_{k,m}+c] ~. $$
\end{enumerate}
\end{prop}

\begin{proof}  Suppose there is a $\check{\psi} = (\check{\alpha},\check{\beta})\in\mathcal{S}_{k,m}$ such that $\check{D}_r\psi=0$ and $\int_{\check{Y}}|\check{\psi}|^2=1$.  We will first prove that $\check{\psi}$ is small except near $\check{\rho}_{k,m}$, then prove that $\check{\psi}$ is similar to the unique solution of the linearized equation at $\check{\rho}_{k,m}$.

By proposition \ref{prop_beta_estimate_00} and (\ref{eqn_Dirac_neck_01}), there exists a constant $c_1$ such that
\begin{align}
\int_{\check{Y}} \big| (r + \frac{kg'-mf'}{\Delta})\check{\alpha} \big|^2  &=\int|{\rm pr}_{1}(\check{D}_r\check{\beta})|^2  \leq c_1\int_{\check{Y}} |\check{\alpha}|^2 ~, \label{eqn_S2S1_estimate_0a}\\
\int_{\check{Y}} \big| e^{i\phi}(\pl_\rho\check{\alpha} - \frac{kg-mf}{\Delta}\check{\alpha}) \big|^2 &=\int|{\rm pr}_{2}(\check{D}_r\check{\beta})|^2\leq c_1\int_{\check{Y}} |\check{\alpha}|^2 \label{eqn_S2S1_estimate_0b}
\end{align}
provided $r\geq c_1$.  Here, ${\rm pr}_1$ and ${\rm pr}_2$ stand for the projection onto the $\underline{\mathbb{C}}$ and $K^{-1}$-component, respectively.  We denote $(r + \frac{kg'-mf'}{\Delta})\check{\alpha}$ by $\mathcal{D}_1\check{\alpha}$, and $e^{i\phi}(\pl_\rho\check{\alpha} - \frac{kg-mf}{\Delta}\check{\alpha})$ by $\mathcal{D}_2\check{\alpha}$.

\subsubsection*{Rough relations between $k$, $m$ and $r$}  Consider the integral of the real part of $\langle\mathcal{D}_1\check{\alpha}, f\check{\alpha}\rangle + \langle\mathcal{D}_2\check{\alpha}, e^{i\phi}f'\check{\alpha}\rangle$.  With (\ref{eqn_S2S1_estimate_0a}), (\ref{eqn_S2S1_estimate_0b}), we have
\begin{align*}
-c'_1\int_{\check{Y}}|\check{\alpha}|^2 \leq \int_{\check{Y}} \big((rf-2k)|\check{\alpha}|^2 + \frac{1}{2}f'\pl_\rho|\check{\alpha}|^2\big) \leq c'_1\int_{\check{Y}} |\check{\alpha}|^2 ~.
\end{align*}
Note that $f'\pl_\rho$ is a globally defined vector field on $\check{Y}$.  With integration by parts, there exist a constant $c_2$ such that
\begin{align} \begin{split} \label{eqn_S2S1_estimate_01}
-c_2\int_{\check{Y}}|\check{\alpha}|^2 \leq \int_{\check{Y}} (rf-2k)|\check{\alpha}|^2 \leq c_2\int_{\check{Y}} |\check{\alpha}|^2 ~, \\
-c_2\int_{\check{Y}}|\check{\alpha}|^2 \leq \int_{\check{Y}} (rg-2m)|\check{\alpha}|^2 \leq c_2\int_{\check{Y}} |\check{\alpha}|^2
\end{split} \end{align}
provided $r\geq c_1$.  The second line is obtained by the same argument on $\langle\mathcal{D}_1\check{\alpha}, g\check{\alpha}\rangle + \langle\mathcal{D}_2\check{\alpha}, e^{i\phi}g'\check{\alpha}\rangle$.  Let $c_3 = c_2+\max \{f,|g|\}$.  It follows from (\ref{eqn_S2S1_estimate_01}) that
\begin{align}\label{eqn_S2S1_estimate_02} 2k&\leq r c_3 ~, &2|m|&\leq r c_3 \end{align}
provided $r\geq c_1$.  On the other hand, it is straightforward to bound $r$ in terms of $k$ and $m$.  By (\ref{eqn_S2S1_estimate_0a}), there exists a constant $c_4$ such that
\begin{align} \label{eqn_S2S1_estimate_06}
r\leq c_4(|k|+|m|)
\end{align}
provided $r\geq c_4$.  It follows that $k$ and $m$ cannot both be zero.

\subsubsection*{Some estimates on $\alpha$}  Consider the integral of $|f\mathcal{D}_1\check{\alpha} + e^{-i\phi}f'\mathcal{D}_2\check{\alpha}|^2$.  By (\ref{eqn_S2S1_estimate_0a}) and (\ref{eqn_S2S1_estimate_0b}), we have
\begin{align*}
\int_{\check{Y}} (rf-2k)^2|\check{\alpha}|^2 + |f'\pl_\rho\check{\alpha}|^2 + (rf-2k)f'\pl_\rho|\check{\alpha}|^2 \leq c''_1\int_{\check{Y}}|\check{\alpha}|^2 ~.
\end{align*}
Throw away the second term, and perform integration by parts on the third term.  With (\ref{eqn_S2S1_estimate_02}), there exists a constant $c_5$ such that
\begin{align}
\int_{\check{Y}} (rf-2k)^2|\check{\alpha}|^2 &\leq rc_5\int_{\check{Y}}|\check{\alpha}|^2 \label{eqn_S2S1_estimate_03}
\end{align}
provided $r\geq c_1$.  The same argument on $|g\mathcal{D}_1\check{\alpha} + e^{-i\phi}g'\mathcal{D}_2\check{\alpha}|^2$ and $|e^{-i\phi}\Delta\mathcal{D}_2\check{\alpha}|^2$ implies that
\begin{align}
\int_{\check{Y}} (rg-2m)^2|\check{\alpha}|^2 &\leq rc_5\int_{\check{Y}}|\check{\alpha}|^2 ~,\label{eqn_S2S1_estimate_04} \\
\int_{\check{Y}} (kg-mf)^2|\check{\alpha}|^2 &\leq rc_5\int_{\check{Y}}|\check{\alpha}|^2 \label{eqn_S2S1_estimate_05}
\end{align}
provided $r\geq c_1$.

We separate the discussion into two cases according to whether
\begin{align*} |m|&<(\frac{1}{32\epsilon^2}-1)k &\text{or}& &|m|&\geq(\frac{1}{32\epsilon^2}-1)k~.\end{align*}

\subsubsection*{Case 1}  When $|m|<(\frac{1}{32\epsilon^2}-1)k$, $k$ can only be positive.  We are going to use (\ref{eqn_S2S1_estimate_05}) to obtain a refined estimate on $\alpha$.  Note that $\check{\rho}_{n,m}$ (given by definition \ref{def_gamma_nm_00}) lies within $(8\epsilon,2-8\epsilon)$.  The function $|kg-mf|$ can only be small near $\check{\rho}_{n,m}$.  More precisely, there is a constant $c_6>0$ such that
\begin{align} \label{eqn_S2S1_estimate_07}
|kg-mf|&\geq
\begin{cases}
\frac{1}{c_6} r &\text{when } |\rho-\check{\rho}_{k,m}|\geq\epsilon ~,\medskip\\
\frac{1}{c_6} r |\rho-\rho_{k,m}| &\text{when } |\rho-\check{\rho}_{k,m}|<\epsilon ~.
\end{cases}
\end{align}
provided $r\geq c_6$.  Here is the proof of (\ref{eqn_S2S1_estimate_07}):  The condition $|m|<(\frac{1}{32\epsilon^2}-1)k$ and (\ref{eqn_S2S1_estimate_06}) implies that $k$ is greater than some multiple of $r$.  When $\rho\leq7\epsilon$ or $\rho\geq2-7\epsilon$, it is straightforward to verify (\ref{eqn_S2S1_estimate_07}).  When $\rho\in(7\epsilon, 2-7\epsilon)$, (\ref{eqn_S2S1_estimate_07}) follows from Taylor's theorem on $\frac{1}{f}(kg-mf)$ and the monotonicity of $\frac{1}{f}(kg-mf)$.

With (\ref{eqn_S2S1_estimate_05}) and (\ref{eqn_S2S1_estimate_07}), there exist a constant $c_7$ such that
\begin{align} \label{eqn_S2S1_estimate_08}
\int_{\check{Y}}|\check{\alpha}|^2 \leq c_7\int_{|\rho-\check{\rho}_{k,m}|\leq c_7 r^{-\oh}} |\check{\alpha}|^2
\end{align}
provided $r\geq c_7$.

\subsubsection*{Refined estimate on $r$}  By (\ref{eqn_S2S1_estimate_0a}) and (\ref{eqn_S2S1_estimate_08}),
\begin{align*}
c_1 \int_{\check{Y}}|\check{\alpha}|^2 &\geq \int_{|\rho-\check{\rho}_{k,m}|\leq c_7 r^{-\oh}}\big|(r-\frac{mf'-kg'}{\Delta})\check{\alpha}\big|^2 \\
&\geq \int_{|\rho-\check{\rho}_{k,m}|\leq c_7 r^{-\oh}}\Big( \frac{(r-\check{\gamma}_{k,m})^2}{2}|\check{\alpha}|^2 - \big| (\check{\gamma}_{k,m}-\frac{mf'-kg'}{\Delta})\check{\alpha} \big|^2 \Big) \\
&\geq \int_{\check{Y}} \frac{(r-\check{\gamma}_{k,m})^2}{2 c_7}|\check{\alpha}|^2 - c_8 \int_{\check{Y}} |\check{\alpha}|^2 ~.
\end{align*}
For the last inequality, note that the derivative of $\frac{mf'-kg'}{\Delta}$ at $\check{\rho}_{k,m}$ is zero.  Taylor's theorem on $\check{\gamma}_{k,m} - \frac{mf'-kg'}{\Delta}$ implies that
\begin{align*}
\big|\check{\gamma}_{k,m}-\frac{mf'-kg'}{\Delta})\big|^2 &\leq c_{8}'r^2|\rho-\check{\rho}_{k,m}|^4 \leq c_8
\end{align*}
for any $\rho$ with $|\rho-\check{\rho}_{k,m}|\leq c_7 r^{-\oh}$.

Thus, there exits a constant $c_9$ such that
\begin{align} \label{eqn_S2S1_estimate_09}
|r-\check{\gamma}_{k,m}|\leq c_9
\end{align}
provided $r\geq c_9$.  It follows that any zero crossing on $\mathcal{S}_{k,m}$ must happen somewhere very close to $\check{\gamma}_{k,m}$.  It also implies that $\check{\gamma}_{k,m}$ and $r$ are of the same order.

\subsubsection*{Zeroth order approximation of $\check{\alpha}$}  Consider the linearized operator of $e^{-i\phi}\mathcal{D}_2$ at $\check{\rho}_{k,m}$:
\begin{align*} \mathcal{L}_{k,m} = \pl_x + \check{\gamma}_{k,m}x \end{align*}
where $x$ is $\rho-\check{\rho}_{k,m}$.  If we regard $\mathcal{L}_{k,m}$ as an operator on $\mathbb{R}$, the theory of the $1$-dimensional harmonic oscillator applies.  See \cite[chapter 9]{ref_Roe} for the properties of the harmonic oscillator.  If $\check{\gamma}_{k,m}>0$, $\mathcal{L}_{k,m}$ has the following properties.  Its kernel is $1$-dimensional, and is spanned by
\begin{align}\label{sln_Dirac_neck_00} \check{\xi}_{k,m} &= (\frac{\check{\gamma}_{k,m}}{\pi})^{\frac{1}{4}}\exp(-\frac{\check{\gamma}_{k,m}}{2} x^2) ~. \end{align}
It has a right inverse operator $G_{k,m}: \mathcal{C}^\infty_{\text{cpt}}(\mathbb{R})\to\mathcal{C}^\infty(\mathbb{R})$ which satisfies
\begin{align*}
\int_{\mathbb{R}} \langle G_{k,m} \eta,\check{\xi}_{k,m}\rangle &= 0 &\text{and}& &\int_{\mathbb{R}} |G_{k,m} \eta|^2 &\leq \frac{1}{\check{\gamma}_{k,m}}\int_{\mathbb{R}} |\eta|^2
\end{align*}
for any $\eta\in\mathcal{C}^\infty_{\text{cpt}}(\mathbb{R})$.  More precisely, the operator $-\pl^2_x + \check{\gamma}^2_{k,m}x^2 + \check{\gamma}_{k,m}$ has positive spectrum, and induce an spectral decomposition.  According to \cite[(9.3)]{ref_Roe}, $G_{k,m}$ is given by
\begin{align*}
\big(-\pl_x + \check{\gamma}_{k,m}x\big) \circ \big(-\pl^2_x + \check{\gamma}^2_{k,m}x^2 + \check{\gamma}_{k,m}\big)^{-1} ~.
\end{align*}

Consider the cut-off function $\chi(r^{\frac{1}{3}}x)$.  By (\ref{eqn_S2S1_estimate_05}) and (\ref{eqn_S2S1_estimate_07}),
\begin{align}
\int_{\check{Y}} \Big| \big(1-\chi(r^{\frac{1}{3}}x)\big)\check{\alpha} \Big|^2 &\leq c_{10} r^{-\frac{1}{3}} \int_{\check{Y}} |\check{\alpha}|^2 ~. \label{eqn_S2S1_estimate_10}
\end{align}
We compute $\mathcal{L}_{k,m}(\chi(r^{\frac{1}{3}}x)\check{\alpha})$:
\begin{align*}
\int_{\check{Y}} \Big| \mathcal{L}_{k,m}\big(\chi(r^{\frac{1}{3}}x)\check{\alpha}\big) \Big|^2 &\leq 2\int \big|\pl_x(\chi(r^{\frac{1}{3}}x))\,\check{\alpha}|^2 + \big(\chi(r^{\frac{1}{2}}x)\big)^2 \big|\mathcal{L}_{k,m}(\check{\alpha})\big|^2  \\
&\leq 2\int \big|\pl_x(\chi(r^{\frac{1}{3}}x))\,\check{\alpha}|^2 + 4\int \big(\chi(r^{\frac{1}{2}}x)\big)^2|\mathcal{D}_2\check{\alpha}|^2 \\
&\quad + 4\int \big(\chi(r^{\frac{1}{2}}x)\big)^2 \big|(\mathcal{L}_{k,m}-e^{-i\phi}\mathcal{D}_2)(\check{\alpha})\big|^2 ~.
\end{align*}
The second term is controlled by (\ref{eqn_S2S1_estimate_0b}).  The operator $\mathcal{L}_{k,m}-e^{-i\phi}\mathcal{D}_2$ does not involve taking derivatives, and $|\mathcal{L}_{k,m}-e^{-i\phi}\mathcal{D}_2|\leq c'_{10} rx^2$ on the support of $\chi(r^{\frac{1}{2}}x)$ by the Taylor series expansion.  With (\ref{eqn_S2S1_estimate_0b}) and (\ref{eqn_S2S1_estimate_09}), we have
\begin{align}
\label{eqn_S2S1_estimate_11} &\int_{\check{Y}} \Big| \mathcal{L}_{k,m}\big(\chi(r^{\frac{1}{3}}x)\check{\alpha}\big) \Big|^2 \\
\notag\leq\; & c''_{10}\Big( r^{\frac{2}{3}}\int_{\oh r^{-\frac{1}{3}}\leq|x|\leq r^{-\frac{1}{3}}}|\check{\alpha}|^2 + \int_{\check{Y}}|\check{\alpha}|^2 + r^\frac{4}{3} \int_{|x|\leq r^{-\frac{1}{3}}} x^2 |\check{\alpha}|^2 \Big) \\
\notag\leq\; & c'''_{10}r^{\frac{1}{3}}\int_{\check{Y}}|\check{\alpha}|^2
\end{align}
provided $r\geq c_{10}$.  Let
$$\chi(r^{\frac{1}{3}}x)\check{\alpha} = \check{\alpha}_{k,m}(x)e^{i(k\phi+mt)}\Delta^{-\oh}(2\pi)^{-1} ~. $$
If we regard $\check{\alpha}_{k,m}(x)$ as being defined on $\mathbb{R}$ and apply $G_{k,m}$ on $\mathcal{L}_{k,m}\big(\check{\alpha}_{k,m}(x)\big)$, we conclude that
\begin{align}
\check{\alpha}_{k,m} &= \mathfrak{c}_{k,m}\check{\xi}_{k,m} + \check{\xi}_{k,m}^\perp \label{eqn_S2S1_estimate_12}\\
& \text{with } \int_{\mathbb{R}}|\check{\xi}_{k,m}^{\perp}|^2\leq c_{11}r^{-\frac{2}{3}}\int_{\check{Y}}|\check{\alpha}|^2 ~, \notag \\
& \text{and } |\mathfrak{c}_{k,m}|^2 = |\check{\alpha}_{k,m} - \check{\xi}^\perp_{k,m}|^2\geq (1-c_{11}r^{-\frac{1}{3}})\int_{\check{Y}}|\check{\alpha}|^2 \notag
\end{align}
for some constant $c_{11}$.

Now, (\ref{eqn_S2S1_estimate_12}), (\ref{eqn_S2S1_estimate_10}) and (\ref{eqn_S2S1_estimate_09}) imply that there exists a constant $c_{12}$ with the following significance.  Suppose that $\check{D}_r$ on $\mathcal{S}_{k,m}$ has two zero modes $\check{\psi}_1$ and $\check{\psi}_2$ at $r_1\geq c_{12}$ and $r_2\geq c_{12}$, respectively.  Then,
\begin{align*}
\big| \int_{\check{Y}}\langle\check{\alpha}_1,\check{\alpha}_2\rangle \big|^2 \geq (1-c_{12}r^{-\frac{1}{3}}) \int_{\check{Y}}|\check{\alpha}_1|^2 \int_{\check{Y}}|\check{\alpha}_2|^2 ~.
\end{align*}
This contradicts proposition \ref{prop_small_ip_00}, and the uniqueness in case 1 follows.\\

\subsubsection*{Case 2 with $m>0$}  When $|m|\geq(\frac{1}{32\epsilon^2}-1)k$, let us further assume that $m>0$.  The case when $m<0$ will be discussed later.  The first task is to show that $\check{\alpha}$ is small on the region where $\rho\geq9\epsilon$.  To start, (\ref{eqn_S2S1_estimate_04}) implies that
\begin{align*}
\int_{\rho\geq2-9\epsilon}|\check{\alpha}|^2\leq c_5 r^{-1}\int_{\check{Y}}|\check{\alpha}|^2 ~.
\end{align*}

If $k\leq0$, (\ref{eqn_S2S1_estimate_03}) implies that
\begin{align*}
\int_{9\epsilon\leq\rho\leq2-9\epsilon}|\check{\alpha}|^2\leq c_{13} r^{-1}\int_{\check{Y}}|\check{\alpha}|^2 ~.
\end{align*}
for some constant $c_{13}$.

If $k\geq0$, it is easy to see that there exists a constant $c_{14}>0$ such that
\begin{align*} kg-mf &\leq -c_{14}m &\text{when}\quad &9\epsilon\leq\rho\leq2-9\epsilon ~. \end{align*}
By (\ref{eqn_S2S1_estimate_05}) and (\ref{eqn_S2S1_estimate_06}), there exists a constant $c_{15}$ such that
\begin{align*}
\int_{9\epsilon\leq\rho\leq2-9\epsilon}|\check{\alpha}|^2\leq c_{15} r^{-1}\int_{\check{Y}}|\check{\alpha}|^2
\end{align*}
provide $r\geq c_{15}$.

The above estimates finds a constant $c_{16}$ such that
\begin{align} \label{eqn_S2S1_estimate_13}
\int_{\rho\geq9\epsilon}|\check{\alpha}|^2\leq c_{16} r^{-1}\int_{\check{Y}}|\check{\alpha}|^2
\end{align}
provided $r\geq c_{16}$.

\subsubsection*{Refined estimate on $r$}  When $\rho\leq 10\epsilon$, the function $\frac{1}{\Delta}(mf'-kg')$ is identically equal to $\check{\gamma}_{k,m} = k+m$.  By (\ref{eqn_S2S1_estimate_0a}) and (\ref{eqn_S2S1_estimate_13}),
\begin{align*}
\int_{\rho\leq9\epsilon} (r-\check{\gamma}_{k,m})^2|\check{\alpha}|^2 \leq c_1\int_{\check{Y}} |\check{\alpha}|^2 \leq c_1(1+\frac{c_{16}r^{-1}}{1-c_{16}r^{-1}})\int_{\rho\leq9\epsilon}|\check{\alpha}|^2
\end{align*}
Therefore, there exists a constant $c_{17}$ such that
\begin{align}\label{eqn_S2S1_estimate_14}
|r-\check{\gamma}_{k,m}|\leq c_{17}
\end{align}
provided $r\geq c_{17}$.

\subsubsection*{Zeroth order approximation of $\check{\alpha}$}  When $\rho\leq 10\epsilon$, the operator $\mathcal{D}_2$ is
\begin{align*} \mathcal{L}_{k,m} &= 2\pl_z + \frac{\check{\gamma}_{k,m}}{2} \bar{z}  \end{align*}
where $z = \rho e^{i\phi}$.  If we regard $\mathcal{L}_{k,m}$ as an operator on $\mathbb{C}$, the theory of $2$-dimensional harmonic oscillator applies.  If $\check{\gamma}_{k,m}>0$, $\mathcal{L}_{k,m}$ has the following properties.  $\mathcal{C}^\infty(\mathbb{C})$ splits according to the frequency with respect to the $S^1$-action by $-i\pl_\phi = z\pl_z - \bar{z}\pl_{\bar{z}}$.  For any $l\in\mathbb{Z}$, $\mathcal{L}_{k,m}$ maps the frequency $l$ subspace to the frequency $l+1$ subspace, and we only care about $\mathcal{L}_{k,m}$ on the frequency $k$ subspace.

When $k< 0$, the kernel of $\mathcal{L}_{k,m}$ is trivial.  It has a right inverse operator $G_{k,m}$ which maps the frequency $k+1$ subspace of $\mathcal{C}^\infty_{\text{cpt}}(\mathbb{C})$ to the frequency $k$ subspace of $\mathcal{C}^\infty(\mathbb{C})$.  $G_{k,m}$ satisfies
\begin{align*}
\int_{\mathbb{C}} |{G}_{k,m}\eta|^2 &\leq \frac{1}{\check{\gamma}_{k,m}}\int_\mathbb{C} |\eta|^2
\end{align*}
for any $\eta\in \mathcal{C}^\infty_{\text{cpt}}(\mathbb{C})$ with frequency $k+1$.
 
When $k\geq 0$, ${\mathcal{L}}_{k,m}$ has a $1$-dimensional kernel spanned by
\begin{align}\label{sln_Dirac_binding_00}
\check{\xi}_{k,m} &= (\frac{1}{k})^\oh(\frac{\check{\gamma}_{k,m}}{2})^\frac{k+1}{2}\frac{z^{k}}{\sqrt{\pi}}\exp(-\frac{\check{\gamma}_{k,m}}{4}|z|^2) ~.
\end{align}
It has a right inverse operator ${G}_{k,m}$ satisfying
\begin{align*}
\int_{\mathbb{C}} |{G}_{k,m}\eta|^2 &\leq \frac{1}{\check{\gamma}_{k,m}}\int_\mathbb{C} |\eta|^2 &\text{, and }& &\int_{\mathbb{C}} \langle{G}_{k,m} \eta,\check{\xi}_{k,m}\rangle &= 0
\end{align*}
for any $\eta\in\mathcal{C}^\infty_{\text{cpt}}(\mathbb{C})$ with frequency $k+1$.

Let $\chi_B$ be the cut-off function depending only on $\rho = |z|$, with $\chi_B(\rho)=1$ when $\rho\leq9\epsilon$ and $\chi_B(\rho)=0$ when $\rho\geq 10\epsilon$.  By (\ref{eqn_S2S1_estimate_13}) and (\ref{eqn_S2S1_estimate_0b}), there exists a constant $c_{18}$ such that
\begin{align}\label{eqn_S2S1_estimate_15}
\int_{\check{Y}} |{\mathcal{L}}_{k,m}(\chi_B\check{\alpha})|^2 \leq c_{18}\int_{\check{Y}}|\check{\alpha}|^2
\end{align}
provided $r\geq c_{18}$.

If $k<0$, we apply $G_{k,m}$ on (\ref{eqn_S2S1_estimate_15}) to find a constant $c_{19}$ so that
\begin{align*} \int_{\check{Y}} |\chi_B\check{\alpha}|^2 = \int_{\mathbb{C}\times S^1} |\chi_B\check{\alpha}|^2 \leq c_{19}r^{-1}\int_{\check{Y}}|\check{\alpha}|^2 ~. \end{align*}
This contradicts (\ref{eqn_S2S1_estimate_13}).  Thus, $k$ can only be nonnegative.  If $k\geq 0$, we apply $G_{k,m}$ on (\ref{eqn_S2S1_estimate_15}) to obtain a similar zeroth order approximation as that in case 1.  By the same token, we end with a contradiction to proposition \ref{prop_small_ip_00}.\\

\subsubsection*{Case 2 with $m<0$}
Similar estimates imply that $\check{\alpha}$ peaks on the region where $\rho\geq2-9\epsilon$.  When $\rho\geq 2-10\epsilon$, let $w = (2-\rho)e^{i\phi}$.  The Dirac operator is
\begin{align*}\left\{\begin{aligned}
\frac{r-2}{2}\check{\alpha} + \frac{i}{2}(\pl_\phi\check{\alpha} - \pl_t\check{\alpha}) + \big( 2\pl_w\check{\beta} + \frac{i}{2}\bar{w}(\pl_\phi\check{\beta} - \pl_t\check{\beta}) + \frac{\bar{w}}{2}\check{\beta} \big) ~, \\
\big( -2\pl_{\bar{w}}\check{\alpha} + \frac{i}{2}w(\pl_\phi\check{\alpha} - \pl_t\check{\alpha}) \big) - \frac{r+1}{2}\check{\beta} - \frac{i}{2}(\pl_\phi\check{\beta} - \pl_t\check{\beta}) ~.
\end{aligned}\right.\end{align*}
With the same argument, there can be at most one zero crossing happening near $\check{\gamma}_{k,m} = k-m$.  This completes the proof of proposition \ref{prop_S2S1_unique_00}.
\end{proof}

%%%%
\subsection{Second order approximation of eigensections}\label{subsec_2nd_order_S2S1}
We need a further understanding of eigensections in this model.  In particular, we need to know where the zero crossing happens up to an error of $\mathcal{O}(r^{-1})$.  It will be achieved by the second order approximation of eigensections with small eigenvalues.

There are two ingredients.  The first ingredient is that the true Dirac operator $\check{D}_r$ can have at most one small eigenvalue on ${\mathcal{S}}_{k,m}$.  The second ingredient is an iteration scheme to construct an approximation by the linearized operator.  The following two lemmata constitute the first ingredient.

\begin{lem}\label{lem_beta_estimate_00}
There exists a constant $c>1$ which has the following significance.  Suppose that $\check{\psi}$ be a eigensection of $\check{D}_r$ for some $r\geq c$, and the magnitude of the corresponding eigenvalue is less than $\sqrt{\frac{r}{2}}$.  Then
\begin{align*}
\int_{\check{Y}}|\check{\beta}|^2 + r^{-1}\int_{\check{Y}}|\nabla_r\check{\beta}|^2 \leq cr^{-1} \int_{\check{Y}} |\check{\alpha}|^2 ~.
\end{align*}
\end{lem}
\begin{proof}
The same as proof of proposition \ref{prop_beta_estimate_00}.
\end{proof}

\begin{lem}\label{lem_S2S1_hot_00}
There exists a constant $c>10$ such that the following holds.  For any $r\geq c$, the Dirac operator $\check{D}_r$ on ${\mathcal{S}}_{k,m}$ has at most one eigenvalue $\lambda$ whose magnitude is less than $\sqrt{\frac{r}{2}}$.  If $k$ is negative or $k=m=0$, there is no such eigenvalue.  Moreover, if there does exist such eigenvalue, then
\begin{align*} |\lambda-(\frac{r}{2}-\frac{\check{\gamma}_{k,m}}{2})|\leq c ~. \end{align*}
\end{lem}

\begin{proof}
The proof is parallel to the proof of proposition \ref{prop_S2S1_unique_00}.  We explain it briefly.

Suppose that $\psi\in\mathcal{S}_{k,m}$ is an eigensection whose eigenvalue has magnitude less than $\sqrt{\frac{r}{2}}$.  By lemma \ref{lem_beta_estimate_00}, (\ref{eqn_S2S1_estimate_0a}) and (\ref{eqn_S2S1_estimate_0b}) would be replaced by
\begin{align*}
\int_{\check{Y}} \big| \lambda - (\frac{r}{2} + \frac{kg'-mf'}{2\Delta})\check{\alpha} \big|^2 &\leq c_1\int_{\check{Y}} |\check{\alpha}|^2 ~,\\
\int_{\check{Y}} \big| e^{i\phi}(\pl_\rho\check{\alpha} - \frac{kg-mf}{\Delta}\check{\alpha}) \big|^2 &\leq c_1\int_{\check{Y}} |\check{\alpha}|^2 ~.
\end{align*}
Based on these two estimates, the bound in (\ref{eqn_S2S1_estimate_01}) becomes $c_2\sqrt{r}$.  The estimates (\ref{eqn_S2S1_estimate_03}), (\ref{eqn_S2S1_estimate_04}), (\ref{eqn_S2S1_estimate_05}) and (\ref{eqn_S2S1_estimate_07}) remain the same.  They imply that
\begin{itemize}
\item the refined estimate on $r$: $|\lambda - (\frac{r}{2}-\frac{\check{\gamma}_{k,m}}{2})|\leq c_3$ for some constant $c_3$;
\item the same zeroth order  approximation of $\check{\alpha}$.
\end{itemize}

If $\check{D}_r$ has two such eigensections, their orthogonality with lemma \ref{lem_beta_estimate_00} implies that the inner product between their first components is small.  It contradicts to the zeroth order approximation of their first components. 
\end{proof}\medskip

Suppose that $\check{D}_r$ on ${\mathcal{S}}_{k,m}$ has an eigenvalue $\lambda_0$ with $|\lambda_0|\leq1$.  Similar to (\ref{eqn_S2S1_estimate_09}) and (\ref{eqn_S2S1_estimate_14}), there exists a constant $c_4$ such that
\begin{align}\label{eqn_hot_rbound_00} |r-\check{\gamma}_{k,m}|\leq c_4 \end{align}
provided $r\geq c_4$.  Because of (\ref{eqn_hot_rbound_00}), $r$ and $\check{\gamma}_{k,m}$ are of the same order.  It follows from lemma \ref{lem_S2S1_hot_00} that if $\lambda$ is an eigenvalue of $\check{D}_r$ on $\mathcal{S}_{k,m}$ other than $\lambda_0$, then
\begin{align}\label{eqn_hot_operatornorm_00} |\lambda-\lambda_0|\geq\sqrt{\frac{r}{8}}. \end{align}
We are going to approximate the eigensection of $\lambda_0$ to the second order.  Again, we separate it into two cases according to whether
\begin{align*} |m|&<(\frac{1}{32\epsilon^2}-1)k &\text{or}& &|m|&\geq(\frac{1}{32\epsilon^2}-1)k ~. \end{align*}

\subsubsection*{Case 1}  When $|m|<(\frac{1}{32\epsilon^2}-1)k$, let
\begin{align}\begin{split}\label{eqn_Fourier_def_rec_00}
\check{\alpha} &= \check{\alpha}_{k,m}(\rho) e^{i(k\phi+m t)}\Delta^{-\oh}(2\pi)^{-1} ~,\\
\check{\beta} &= \check{\beta}_{k,m}(\rho) e^{i((k+1)\phi+mt)}\Delta^{-\oh}(2\pi)^{-1} ~.
\end{split}\end{align}
The Dirac operator on $\check{\alpha}_{k,m}$ and $\check{\beta}_{k,m}$ is
\begin{align}\label{eqn_Dirac_neck_00}\left\{\begin{aligned}
(\frac{r}{2} + \frac{kg'-mf'}{2\Delta})\check{\alpha}_{k,m} + (-\check{\beta}'_{k,m} - \frac{kg-mf}{\Delta}\check{\beta}_{k,m} - \frac{\Delta'}{2\Delta}\check{\beta}_{k,m}) ~, \\
(\check{\alpha}'_{k,m} - \frac{kg-mf}{\Delta}\check{\alpha}_{k,m} - \frac{\Delta'}{2\Delta}\check{\alpha}_{k,m}) - (\frac{r}{2} + 1 + \frac{kg'-mf'}{2\Delta} + \frac{f''g'-f'g''}{8\Delta})\check{\beta}_{k,m} ~.
\end{aligned}\right.\end{align}
For simplicity, change the variable by $x = \rho-\check{\rho}_{k,m}$.  The Taylor series expansion of the eigensection equation at $x=0$ is of the following form
\begin{align}\label{eqn_hot_approx_00}\left\{\begin{aligned}
\lambda \check{\alpha}_{k,m} &= (\frac{r}{2}-\frac{\check{\gamma}_{k,m}}{2}+\mathfrak{r}_1 x^2+\mathfrak{r}_3 x^3 + \mathfrak{R}^\alpha_1)\check{\alpha}_{k,m}\\
&\qquad + (-\frac{\dd}{\dd x} + \check{\gamma}_{k,m} x + \mathfrak{c}_1 + \mathfrak{r}_2 x^2 + \mathfrak{R}^\beta_1)\check{\beta}_{k,m} ~, \\
\lambda \check{\beta}_{k,m} &= (\frac{\dd}{\dd x} + \check{\gamma}_{k,m} x + \mathfrak{c}_1 + \mathfrak{r}_2 x^2 + \mathfrak{c}_2 x + \mathfrak{r}_4 x^3 + \mathfrak{R}^\alpha_2)\check{\alpha}_{k,m} \\
&\qquad - (\frac{r}{2}-\frac{\check{\gamma}_{k,m}}{2} + \mathfrak{c}_3 + \mathfrak{r}_1 x^2 + \mathfrak{R}^\beta_2)\check{\beta}_{k,m} ~.
\end{aligned}\right.\end{align}
By Taylor's theorem, there exists a constant $c_5>0$ such that the coefficients $|\mathfrak{c}_j|\leq c_5$ and $|\mathfrak{r}_j|\leq c_5 r$; the remainder terms $|\mathfrak{R}^\alpha_j|\leq c_5(x^2 + rx^4)$ and $|\mathfrak{R}^\beta_j|\leq c_5(|x| + r|x|^3)$ for $|x|\leq \epsilon$.  These coefficients and the remainder terms depend on $k$ and $m$.  Similar to \ref{eqn_S2S1_estimate_02}), the assumption on the existence of small eigenvalue implies that $k$ and $|m|$ are less than some multiple of $r$.

Equation (\ref{eqn_hot_approx_00}) can be used to construct higher order approximation of the eigensection by the following procedure.  Rewrite the equation as
\begin{align*}\left\{\begin{aligned}
(-\frac{\dd}{\dd x} + \check{\gamma}_{k,m}x)\check{\beta}_{k,m} &= \mathfrak{F}_1(\check{\alpha}_{k,m},\check{\beta}_{k,m}) + (\lambda - \frac{r}{2})\check{\alpha}_{k,m} ~,\\
(\frac{\dd}{\dd x} + \check{\gamma}_{k,m}x)\check{\alpha}_{k,m} &= \mathfrak{F}_2(\check{\alpha}_{k,m},\check{\beta}_{k,m}) + (\lambda+\frac{r}{2})\check{\beta}_{k,m} ~.
\end{aligned}\right.\end{align*}
Start with $\check{\alpha}_{k,m} = (\frac{\check{\gamma}_{k,m}}{\pi})^{\frac{1}{4}}\exp(-\frac{\check{\gamma}_{k,m}}{2}x^2)$, $\check{\beta}_{k,m}=0$, and $\lambda = 0$.  The next order term of $\check{\alpha}_{k,m}$ is determined by the second equation.  The next order term of $\lambda$ is determined by the condition that $\mathfrak{F}_1(\check{\alpha}_{k,m},\check{\beta}_{k,m}) + (\lambda-\frac{r}{2})\check{\alpha}_{k,m}$ is $L^2$-orthogonal to $\exp(-\frac{\check{\gamma}_{k,m}}{2}x^2)$.  The next order term of $\check{\beta}_{k,m}$ is determined by the first equation.

Following this procedure, we have the second order approximation:
\begin{align}\label{eqn_hot_approx_01}\left\{\begin{aligned}
\check{\alpha}_{k,m} &= \big( 1 + \mathfrak{a}_1(x) + \mathfrak{a}_2(x,r) \big) \chi(\epsilon x) \big(\frac{\check{\gamma}_{k,m}}{\pi}\big)^{\frac{1}{4}} \exp(-\frac{\check{\gamma}_{k,m}}{2}x^2) ~,\\
\check{\beta}_{k,m} &= \big( \mathfrak{b}_1(x) + \mathfrak{b}_2(x) \big) \chi(\epsilon x) \big(\frac{\check{\gamma}_{k,m}}{\pi}\big)^{\frac{1}{4}} \exp(-\frac{\check{\gamma}_{k,m}}{2}x^2) ~,\\
\lambda &= \frac{r}{2} - \frac{\check{\gamma}_{k,m}}{2} + \frac{\mathfrak{r}_1}{2\check{\gamma}_{k,m}}
\end{aligned}\right.\end{align}
where
\begin{align*}
\mathfrak{a}_1(x) &= - \mathfrak{c}_1 x - \frac{\mathfrak{r}_2}{3} x^3 ~,\\
\mathfrak{a}_2(x,r) &= \Big(\frac{\mathfrak{c}_1^2-\mathfrak{c}_2}{2} - \frac{\mathfrak{r}_1}{4\check{\gamma}_{k,m}}(r-\check{\gamma}_{k,m}+\frac{\mathfrak{r}_1}{2\check{\gamma}_{k,m}}+\mathfrak{c}_3)\Big) \big( x^2-\frac{1}{2\check{\gamma}_{k,m}} \big) \\
&\qquad + \big( \frac{\mathfrak{c}_1\mathfrak{r}_2}{3} - \frac{\mathfrak{r}_4}{4} - \frac{\mathfrak{r}_1^2}{8\check{\gamma}_{k,m}} \big) \big( x^4 - \frac{3}{4\check{\gamma}_{k,m}^2} \big) + \frac{\mathfrak{r}_2^2}{18} \big( x^6 - \frac{15}{8\check{\gamma}_{k,m}^3} \big) ~,\\
\mathfrak{b}_1(x) &= - \frac{\mathfrak{r}_1}{2\check{\gamma}_{k,m}}x ~,\\
\mathfrak{b}_2(x) &= \big(\frac{\mathfrak{c}_1\mathfrak{r}_1-\mathfrak{r}_3}{2\check{\gamma}_{k,m}} + \frac{\mathfrak{r}_1\mathfrak{r}_2}{4\check{\gamma}_{k,m}^2}\big) \big( x^2 + \frac{1}{2\check{\gamma}_{k,m}} \big) + \frac{\mathfrak{r}_1\mathfrak{r}_2}{6\check{\gamma}_{k,m}}x^4 ~.
\end{align*}
To estimate the error term, note that
\begin{align*} \int_{\mathbb{R}} \big| x^l\exp(-\frac{\check{\gamma}_{k,m}}{2}x^2) \big|^2 \leq ((l+3)!) (\check{\gamma}_{k,m})^{-(l+\oh)} \end{align*}
for any integer $l\geq 0$.  If we plug (\ref{eqn_hot_approx_01}) into (\ref{eqn_hot_approx_00}), the size of the error term can be computed directly.  Let $\check{\psi}_{n,m}$ be the section whose components are given by (\ref{eqn_Fourier_def_rec_00}) and (\ref{eqn_hot_approx_01}), then there exists a constant $c_6$ such that
\begin{align} \label{eqn_S2S1_almost_eigen_00}
\int_{\check{Y}} \big|\check{D}_r \check{\psi}_{k,m} - (\frac{r}{2} - \frac{\check{\gamma}_{k,m}}{2} + \frac{\mathfrak{r}_1}{2\check{\gamma}_{k,m}}) \check{\psi}_{k,m}\big|^2 \leq c_6 r^{-2}\int_{\check{Y}}|\check{\psi}_{k,m}|^2
\end{align}
provided $r\geq c_6$.  The properties of this second order approximation are summarized in the following proposition.

\begin{prop}\label{prop_small_eigensection_00}
There exists a constant $c$ which has the following significance.  For any $r\geq c$ and $|m|<(\frac{1}{32\epsilon^2}-1)k$, suppose that the Dirac operator $\check{D}_r$ on ${\mathcal{S}}_{k,m}$ has an eigenvalue $\lambda_0$ with $|\lambda_0|\leq1$.  Then the corresponding eigensection is
\begin{align*} \check{\psi}_{k,m}^{\text{\rm eig}} &= \check{\mathfrak{q}}_{k,m}\check{\psi}_{k,m} + \check{\psi}_{k,m}^{(3)} \end{align*}
with $\check{\psi}_{k,m}$ is given by (\ref{eqn_Fourier_def_rec_00}) and (\ref{eqn_hot_approx_01}), and
\begin{align*}
\int_{\check{Y}}|\check{\psi}_{k,m}^{\text{\rm eig}}|^2&=1 ~,
&\int_{\check{Y}}|\check{\psi}_{k,m}^{(3)}|^2&\leq c r^{-3} ~,
&\text{and}& &|\check{\mathfrak{q}}_{k,m}-1|&\leq cr^{-1} ~.
\end{align*}
Moreover,
\begin{align*} \big|\lambda_0 - (\frac{r}{2} - \frac{\check{\gamma}_{k,m}}{2} + \frac{\mathfrak{r}_1}{2\check{\gamma}_{k,m}})\big|\leq c r^{-1} \end{align*}
where $\mathfrak{r}_1$ is the coefficient of the second order term in the Taylor series expansion of $\frac{1}{2\Delta}(kg'-mf')$ at $\check{\rho}_{k,m}$.
\end{prop}
\begin{proof}
Let ${{\rm pr}_{\lambda_0}}$ be the $L^2$-orthogonal projection onto the eigenspace of $\lambda_0$.  Write $\check{\psi}_{k,m}$ as ${{\rm pr}_{\lambda_0}}(\check{\psi}_{k,m}) + (\check{\psi}_{k,m}-{{\rm pr}_{\lambda_0}}(\check{\psi}_{k,m}))$.  By (\ref{eqn_S2S1_almost_eigen_00}) and (\ref{eqn_hot_operatornorm_00}),
\begin{align*} \int_{\check{Y}} |\check{\psi}_{k,m}-{{\rm pr}_{\lambda_0}}(\check{\psi}_{k,m})|^2 &\leq 8c_6r^{-3}\int_{\check{Y}} |\check{\psi}_{k,m}|^2 ~. \end{align*}
From the expression (\ref{eqn_hot_approx_01}), there exists a constant $c_7$ such that
\begin{align*} \big| 1 - \int_{\check{Y}}|\check{\psi}_{k,m}|^2 \big|\leq c_7 r^{-1} ~. \end{align*}
After normalizing the $L^2$-norm of $\check{\psi}_{k,m} + ({{\rm pr}_{\lambda_0}}(\check{\psi}_{k,m}))-\check{\psi}_{k,m})$, we obtain the desired expression of $\check{\psi}^{\text{\rm eig}}_{k,m}$.  The estimate on $\lambda$ follows from (\ref{eqn_S2S1_almost_eigen_00}).
\end{proof}

There is a subtlety about the section $\check{\psi}_{k,m}$ given by (\ref{eqn_hot_approx_01}):  it depends on $r$, particularly the $\mathfrak{a}_2(x,r)$ term.  However, that term is small, and it implies that the difference between eigensections at different $r$ is small.

\begin{cor}\label{cor_eigensection_differ_r_00}
There exists a constant $c$ such that the following holds.  Suppose that $\check{D}_{r_1}$ and $\check{D}_{r_2}$ both have an eigenvalue whose magnitude is less than or equal to $1$ on $\mathcal{S}_{k,m}$, and $r_1\geq c$, $r_2\geq c$ and $|m|<(\frac{1}{32\epsilon^2}-1)k$.  Let $\check{\psi}_{k,m}^{\text{\rm eig}}(r_1)$ and $\check{\psi}_{k,m}^{\text{\rm eig}}(r_2)$ be the corresponding eigensections of unit $L^2$-norm.  Then
\begin{align*}
\big| \int_{\check{Y}}\langle \check{\eta}, \check{\psi}_{k,m}^{\text{\rm eig}}(r_1) \rangle \big| \leq c (\min\{r_1,r_2\})^{-\frac{3}{2}} \big(\int_{\check{Y}} |\check{\eta}|^2 \big)^\oh
\end{align*}
for any $\check{\eta}$ with $\int_\Sigma \langle\check{\eta},\check{\psi}_{k,m}^{\text{\rm eig}}(r_2)\rangle = 0$.
\end{cor}
\begin{proof}
By proposition \ref{prop_small_eigensection_00}, $|r_1-r_2|\leq 8$.  From the expression (\ref{eqn_hot_approx_01}),
\begin{align*} \int_{\check{Y}} \big| (\check{\psi}_{k,m}(r_1)-\check{\psi}_{k,m}(r_2)) \big|^2 &\leq c_8(r_2-r_1)^2 \int_\mathbb{R}x^4(\check{\gamma}_{k,m})^{\frac{1}{2}} \exp(-{\check{\gamma}_{k,m}}x^2) \\
&\leq c_9 (\min\{r_1,r_2\})^{-4}  \end{align*}
for some constant $c_8$ and $c_9$.  Since the terms $\mathfrak{a}_1(x)$ and $\mathfrak{b}_1(x)$ in (\ref{eqn_hot_approx_01}) do not depend on $r$, there exists a constant $c_{10}$ such that
\begin{align*}
\Big| 1 + ( \frac{\mathfrak{c}_1^2}{2\check{\gamma}_{k,m}} + \frac{\mathfrak{c}_1\mathfrak{r}_2}{2\check{\gamma}_{k,m}^2} + \frac{5\mathfrak{r}_2^2}{24\check{\gamma}_{k,m}^3} + \frac{\mathfrak{r}_1^2}{8\check{\gamma}_{k,m}^3} ) - \int_{\check{Y}}|\check{\psi}_{k,m}(r_j)|^2 \Big| \leq c_{10}r_j^{-2}
\end{align*}
for $j=1,2$.  Using this improved estimate in the proof of proposition \ref{prop_small_eigensection_00}, we find that
\begin{align*} \big|\check{\mathfrak{q}}_{k,m}(r_1) - \check{\mathfrak{q}}_{k,m}(r_2)\big|\leq c_{11} (\min\{r_1,r_2\})^{-2} \end{align*}
for some constant $c_{11}$.  The difference between these eigensections is
\begin{align*}
\check{\psi}_{k,m}^{\text{\rm eig}}(r_1) - \check{\psi}_{k,m}^{\text{\rm eig}}(r_2) &= \big(\check{\mathfrak{q}}_{k,m}(r_1)-\check{\mathfrak{q}}_{k,m}(r_2)\big)\check{\psi}_{k,m}(r_1) + \check{\psi}_{k,m}^{(3)}(r_1)\\
&\qquad + \check{\mathfrak{q}}_{k,m}(r_1)\big(\check{\psi}_{k,m}(r_1)-\check{\psi}_{k,m}(r_2)\big) - \check{\psi}_{k,m}^{(3)}(r_2).
\end{align*}
Take the inner product of $\check{\eta}$ with the above expression, and integrate over $\check{Y}$.  With proposition \ref{prop_small_eigensection_00} and some simple manipulations, it completes the proof of the corollary.
\end{proof}

Throughout the discussion for proposition \ref{prop_small_eigensection_00}, $r$ is fixed.  We can forget the assumption of proposition \ref{prop_small_eigensection_00}, and look at (\ref{eqn_hot_approx_01}) independently.  The inequality (\ref{eqn_S2S1_almost_eigen_00}) still holds as long as $r$ and $\gamma_{k,m}$ differ by $\mathcal{O}(1)$.  In other words, we can rephrase it as the existence of small eigenvalues.

\begin{lem}\label{lem_exist_sl_00}
There exist a constant $c$ which has the following significance.  For any $k$, $m$ and $r$ with $|m|< (\frac{1}{32\epsilon^2}-1)k$, $\check{\gamma}_{k,m}\geq c$ and $|r-\check{\gamma}_{k,m}+\frac{\mathfrak{r}_1}{\check{\gamma}_{k,m}}|\leq 1$, the section $\check{\psi}_{k,m}$ defined by (\ref{eqn_Fourier_def_rec_00}) and (\ref{eqn_hot_approx_01}) satisfies
\begin{align*}
\int_{\check{Y}} \big|\check{D}_r \check{\psi}_{k,m} - (\frac{r}{2}-\frac{\check{\gamma}_{k,m}}{2}+\frac{\mathfrak{r}_1}{2\check{\gamma}_{k,m}})\check{\psi}_{k,m} \big|^2 &\leq c r^{-2} \int_{\check{Y}} |\check{\psi}_{k,m}|^2 \\
\big| 1-\int|\check{\psi}_{k,m}|^2 \big| &\leq cr^{-1}
\end{align*}
with the same $\mathfrak{r}_1$ as that in proposition \ref{prop_small_eigensection_00}.  Therefore, there exists an eigenvalue $\lambda$ of $\check{D}_r$ on $\mathcal{S}_{k,m}$ with
$$ \big|\lambda - (\frac{r}{2}-\frac{\check{\gamma}_{k,m}}{2}+\frac{\mathfrak{r}_1}{2\check{\gamma}_{k,m}})\big|\leq cr^{-1} ~. $$ 
\end{lem}
\begin{proof}
We only need to check that the coefficients and remainder terms of equation (\ref{eqn_hot_approx_00}) remains the same order.  According to the equation of $\check{\gamma}_{k,m}$ in definition \ref{def_gamma_nm_00}, $k\leq c_{12} \check{\gamma}_{k,m}$ for some constant $c_{12}$, and $|m|< (\frac{1}{32\epsilon^2}-1) c_{12}\check{\gamma}_{k,m}$.  Hence, the bound on the coefficients and remainder terms of equation (\ref{eqn_hot_approx_00}) remains the same order.  By considering the Rayleigh quotient of $\check{D}_r - (\frac{r}{2}-\frac{\check{\gamma}_{k,m}}{2}+\frac{\mathfrak{r}_1}{2\check{\gamma}_{k,m}})$ on $\mathcal{S}_{k,m}$, we conclude the existence of such an eigenvalue $\lambda$.
\end{proof}

\subsubsection*{Case 2}  When $|m|\geq(\frac{1}{32\epsilon^2}-1)k$, the linearized equation can be solved completely.  Let us further assume that $m>0$.  The discussion for $m<0$ is completely parallel, and will be omitted.

The functions $\check{\xi}_{k,m}$ defined by (\ref{sln_Dirac_binding_00}) are the almost eigensections.  They are exponentially small when $\rho\geq9\epsilon$.
\begin{lem} \label{lem_estimate_binding_00}
There exist a constant $c$ such that the following holds.  For any integers $k$ and $m$ with $m\geq (\frac{1}{32\epsilon^2}-1)k\geq0$, the function $\check{\xi}_{k,m}$ defined by (\ref{sln_Dirac_binding_00}) satisfies
\begin{align*}
|\check{\xi}_{k,m}|^2 \leq c \exp(-\frac{\check{\gamma}_{k,m}}{c}|z|^2)
\end{align*}
for any $z$ with $|z|\geq 9\epsilon$.
\end{lem}
\begin{proof}
Remember that $\check{\gamma}_{k,m}$ is $k+m$ when $m\geq (\frac{1}{32\epsilon^2}-1)k\geq0$.  Within this proof, we simply denote it by $\gamma$, and we are going to think $\gamma$ as a variable with $\gamma\geq \frac{k}{32\epsilon^2}$.  From the expression
\begin{align*}
4\pi^2|\check{\xi}_{k,m}|^2 &= \frac{1}{k!}(\frac{\gamma}{2})^{k+1}|z|^{2k} \exp(-\frac{\gamma}{2}|z|^2),
\end{align*}
$|\check{\xi}_{k,m}|^2$ is monotone increasing in $k$, for any fixed $\gamma\geq \frac{k}{32\epsilon^2}$ and $|z|\geq9\epsilon$.  Thus,
\begin{align*}
&4\pi^2|\check{\xi}_{k,m}|^2 \\
\leq& \frac{1}{\Gamma(32\epsilon^2 \gamma+1)}(\frac{\gamma}{2})^{32\epsilon^2 \gamma +1}|z|^{64\epsilon^2 \gamma} \exp(-\frac{\gamma}{2}|z|^2) \\
=& \left( \frac{1}{\Gamma(32\epsilon^2 \gamma+1)}(\frac{\gamma}{2})^{32\epsilon^2 \gamma +1}|z|^{64\epsilon^2 \gamma} \exp\big(-(\frac{1}{2}-\frac{1}{c}) \gamma|z|^2\big) \right) \exp(-\frac{\gamma}{c} |z|^2)
\end{align*}
where $\Gamma(s) = \int_0^\infty x^{s-1}e^{-x}\dd x$ is the usual gamma function.
If $c$ is sufficiently large, Stirling's formula implies that the whole expression in front of $\exp(-\frac{\gamma}{c} |z|^2)$ is uniformly bounded for all $\gamma>0$ and $|z|\geq9\epsilon$.
\end{proof}

Recall that $\chi_B$ is the cut-off function depending only on $\rho = |z|$, with $\chi_B(\rho)=1$ when $\rho\leq9\epsilon$ and $\chi_B(\rho)=0$ when $\rho\geq 10\epsilon$.  Let $\check{\psi}_{k,m}$ be the section whose first component is
\begin{align}\label{eqn_hot_approx_02}
\check{\alpha}_{k,m} &= \chi_B \check{\xi}_{k,m} e^{imt} ~,
\end{align}
and whose second component is zero.  By lemma \ref{lem_estimate_binding_00} and (\ref{eqn_Dirac_binding_00}), there exists a constant $c_{13}$ such that
\begin{align}\label{eqn_S2S1_almost_eigen_01}
\int_{\check{Y}} \big| \check{D}_r\check{\psi}_{k,m} - (\frac{r}{2}-\frac{\check{\gamma}_{k,m}}{2}-1)\check{\psi}_{k,m} \big|^2 \leq c_{13} \exp(-\frac{r}{c_{13}}) ~.
\end{align}
With (\ref{eqn_hot_operatornorm_00}) and (\ref{eqn_hot_rbound_00}), we conclude that:

\begin{prop}\label{prop_small_eigensection_01}
There exists a constant $c$ which has the following significance.  For any $r\geq c$ and $m\geq(\frac{1}{32\epsilon^2}-1)k\geq0$, suppose that the Dirac operator $\check{D}_r$ on ${\mathcal{S}}_{k,m}$ has an eigenvalue $\lambda_0$ with $|\lambda_0|\leq1$.  Then the corresponding eigensection is
\begin{align*} \check{\psi}_{k,m}^{\text{\rm eig}} &= \check{\mathfrak{q}}_{k,m} \check{\psi}_{k,m} + \check{\psi}_{k,m}^{(3)} \end{align*}
with $\check{\psi}_{k,m}$ given by (\ref{eqn_hot_approx_02}), and
\begin{align*}
\int_{\check{Y}}|\check{\psi}_{k,m}^{\text{\rm eig}}|^2&=1 ~,
&\int_{\check{Y}}|\check{\psi}_{k,m}^{(3)}|^2 &\leq c \exp(-\frac{r}{c}) ~,
&\text{and} & &|\check{\mathfrak{q}}_{k,m}-1|&\leq c \exp(-\frac{r}{c}) ~. 
\end{align*}
Moreover,
\begin{align*} \Big|\lambda_0 - (\frac{r}{2} - \frac{\check{\gamma}_{k,m}}{2})\Big|\leq c \exp(-\frac{r}{c}) ~. \end{align*}
\end{prop}

In this case, $|m|\geq(\frac{1}{32\epsilon^2}-1)k$, the section $\check{\psi}_{k,m}$ does not depend on $r$, but $\check{\psi}_{k,m}^{\text{\rm eig}}$ does depend on $r$.  However, the dependence is small.  Similar to corollary \ref{cor_eigensection_differ_r_00}, we have:
\begin{cor}\label{cor_eigensection_differ_r_01}
There exists a constant $c$ such that the following holds.  Suppose that $\check{D}_{r_1}$ and $\check{D}_{r_2}$ both have an eigenvalue whose magnitude is less than or equal to $1$ on $\mathcal{S}_{k,m}$, and $r_1\geq c$, $r_2\geq c$ and $m\geq(\frac{1}{32\epsilon^2}-1)k\geq0$.  Let $\check{\psi}_{k,m}^{\text{\rm eig}}(r_1)$ and $\check{\psi}_{k,m}^{\text{\rm eig}}(r_2)$ be the corresponding eigensections of unit $L^2$-norm.  Then
\begin{align*}
\big| \int_{\check{Y}}\langle \check{\eta}, \check{\psi}_{k,m}^{\text{\rm eig}}(r_1) \rangle \big| \leq c\exp(-\frac{\min\{r_1,r_2\}}{c}) \big(\int_{\check{Y}} |\check{\eta}|^2 \big)^\oh
\end{align*}
for any $\check{\eta}$ with $\int_\Sigma \langle\check{\eta},\check{\psi}_{k,m}^{\text{\rm eig}}(r_2)\rangle = 0$.
\end{cor}

In this case, the sections $\check{\psi}_{k,m}$ also guarantees the existence of small eigenvalues.

\begin{lem}\label{lem_exist_sl_01}
There exist a constant $c$ which has the following significance:  For any $k$, $m$ and $r$ with $m \geq (\frac{1}{32\epsilon^2}-1)k\geq0$, $\check{\gamma}_{k,m}\geq c$ and $|r-\check{\gamma}_{k,m}|\leq1$, the section $\check{\psi}_{k,m}$ defined by (\ref{eqn_hot_approx_02}) satisfies
\begin{align*}
\int_{\check{Y}} \big|D_r \check{\psi}_{k,m} - (\frac{r}{2}-\frac{\check{\gamma}_{k,m}}{2})\check{\psi}_{k,m} \big|^2 &\leq c \exp(-\frac{r}{c}) \int_{\check{Y}} |\check{\psi}_{k,m}|^2
\end{align*}
and $\big| 1-\int_{\check{Y}}|\check{\psi}_{k,m}|^2 \big| \leq c \exp(-\frac{r}{c})$.  Therefore, there exists an eigenvalue $\lambda$ of $\check{D}_r$ on $\mathcal{S}_{k,m}$ with
$$ \big|\lambda - (\frac{r}{2}-\frac{\check{\gamma}_{k,m}}{2})\big|\leq c \exp(-\frac{r}{c}) ~. $$ 
\end{lem}
Since the coefficient $\mathfrak{r}_1$ in lemma \ref{lem_exist_sl_00} is equal to $0$ for $m \geq (\frac{1}{32\epsilon^2}-1)k\geq0$, we can also formally put the $-\frac{\mathfrak{r}_1}{2\check{\gamma}_{k,m}}$ term in proposition \ref{prop_small_eigensection_00} and lemma \ref{lem_exist_sl_01}.

%%%%
\subsection{Estimating the spectral flow function}
We now prove theorem \ref{thm_sf_main_00} for this model case.

\begin{thm}\label{thm_sf_S1S2_00}
For the associated contact form (\ref{eqn_S2S1_contact_00}) of the tubular neighborhood of the binding, there exists a constant $c$ such that
\begin{align*} \big| \check{{\rm sf}}_a(r) - \frac{r^2}{4}\int_0^2 \Delta\dd\rho \big| \leq cr \end{align*}
for all $r \geq c$.  The function $\Delta$ is defined by (\ref{eqn_S2S1_Delta_00}).
\end{thm}
\begin{proof}
Proposition \ref{prop_S2S1_unique_00} implies the following upper bound of the spectral flow function
\begin{align*}
\check{{\rm sf}}_a(r) &\leq \#\big\{\text{integers }k\geq 0 \text{ and } m ,\text{ with }\check{\gamma}_{k,m}\leq r+c_1\big\} + c_1 \\
&= \#\big\{\text{integers }k\geq 1 \text{ and } m ,\text{ with }\check{\gamma}_{k,m}\leq r+c_1\big\} + 2r + 3c_1
\end{align*}
for some constant $c_1$.  Thus, it suffices to count the total number of the lattice points $(k,m)$ with $\check{\gamma}_{k,m}\leq r+c_1$.

Consider the reparameterized polar coordinate $(s,\rho)$ on the right half-plane:
$ k = s \big(f^2(\rho)+g^2(\rho)\big)^{-\oh} f(\rho) $, $ m = s \big(f^2(\rho)+g^2(\rho)\big)^{-\oh} g(\rho) $.  Let
\begin{align*} \check{\gamma}(s,\rho) = 2s \big(f^2(\rho)+g^2(\rho)\big)^{-\oh} ~. \end{align*}
Then $\check{\gamma}(s,\rho)$ is equal to $\check{\gamma}_{k,m}$ at any lattice points $(k,m)$.  From the expression of $\check{\gamma}(s,\rho)$, it is not hard to see that there exists a constant $c_2>0$ such that the total number of lattice points with $\check{\gamma}_{k,m}\leq r$ is less than the area of where $\check{\gamma}\leq r+c_2$.  The area can be evaluated directly
\begin{align*}
\int_0^2 \int_0^{\frac{\sqrt{f^2+g^2}}{2}(r+c_{2})} \frac{2\Delta}{f^2+g^2} s \dd s\dd\rho &=  \frac{r^2 + 2c_{2}r + c_{2}^2}{4} \int_0^2 \Delta\dd\rho ~.
\end{align*}
This proves the assertion on the upper bound of the spectral flow function.

Lemma \ref{lem_exist_sl_00} and lemma \ref{lem_exist_sl_01} finds a constant $c_3>0$ such that $\check{D}_r$ on $\mathcal{S}_{k,m}$ has a zero eigenvalue within $[\check{\gamma}_{k,m}-c_3, \check{\gamma}_{k,m}+c_3]$, provided $\check{\gamma}_{k,m}\geq c_3$.  Therefore,
\begin{align*}
\check{\rm sf}_a(r) &\geq \#\{\text{integers }k\geq 0 \text{ and } m ,\text{ with } c_3\leq\check{\gamma}_{k,m}\leq r+c_3\} ~.
\end{align*}
The same area computation gives the desired lower bound on the spectral flow function.
\end{proof}

Theorem \ref{thm_sf_S1S2_00} can be used to find a sequence of numbers such that there are not too many zero crossings near these numbers.  Here is the precise statement:

\begin{lem}\label{lem_divide_interval_00}
For any $\delta_3>0$, there exist a constant $c$ and a sequence of numbers $\{s_n\}_{n\in\mathbb{N}}$ determined by the associated contact forms and with the following significance:
\begin{enumerate}
\item the total number of zero crossings of $\check{D}_r$ and $\tilde{D}_r$ happening between $s_n-\frac{\delta_3}{s_{n-1}}$ and $s_n+\frac{\delta_3}{s_{n-1}}$ is less than $c$ for all $n\in\mathbb{N}$;
\item for each $n\in\mathbb{N}$, $ |s_n - \gamma_n - \frac{1}{V}|\leq \frac{1}{4V} $ where $\gamma_n=\frac{2n}{V}$.
\end{enumerate}
\end{lem}
\begin{proof}
For $n\leq 8(\delta_3+1) V^2$, take $s_n$ to be $\gamma_n+\frac{1}{V}$.  For $n > 8(\delta_3+1) V^2$, we are going to define $s_n$ inductively.  By theorem \ref{thm_sf_S1S2_00}, there exists a constant $c_1>0$ such that the total number of zero crossings happening within the interval
$$ I_n = [\gamma_n+\frac{3}{4V},\gamma_n+\frac{5}{4V}]$$
is less than $c_1n \leq c_1 V s_{n-1}$.  Divide $I_n$ into sub-intervals whose lengths are $\frac{2\delta_3}{s_{n-1}}$.  It follows that the total number of sub-intervals is greater than $\frac{s_{n-1}}{8\delta_3 V}$.  There must be a sub-interval which contains less than $8c_1\delta_3 V^2$ zero crossings.  Let $s_n$ be the midpoint of that sub-interval.  By the construction, $\{s_n\}_{n\in\mathbb{N}}$ has the desired properties.
\end{proof}

%%%%
\subsection{Higher order approximation on certain regions}\label{subsec_higher_order_S2S1}
This section is a remark on the approximation eigensections.  On the region where $1-5\epsilon\leq\rho\leq1+15\epsilon$ of $\check{Y}$, the function $f = V$ and $g = 2-\rho$, and the Dirac operator on $\mathcal{S}_{k,m}$ is already linear on this region, see (\ref{eqn_Dirac_neck_00}).

If $\frac{k}{V}(1-11\epsilon)\leq m\leq\frac{k}{V}$, $\check{\rho}_{k,m}$ lies between $1$ and $1+11\epsilon$, and $\check{\gamma}_{k,m}$ is equal to $\frac{2k}{V}$.  All the higher order terms in (\ref{eqn_hot_approx_00}) are equal to zero, and the correction terms in (\ref{eqn_hot_approx_01}) are also equal to zero.  Meanwhile, the almost solution (\ref{eqn_hot_approx_01}) only supports within the interval $[1-\epsilon,1+12\epsilon]$.  Therefore, these almost solutions solve the Dirac equation up to an exponentially small error term.  We have similar statements as proposition \ref{prop_small_eigensection_01} and lemma \ref{lem_exist_sl_01}.

\begin{lem}\label{lem_refined_r_neck_00}
There exists a constant $c$ which has the following significance:  For any $r\geq c$ and $\frac{k}{V}(1-11\epsilon)\leq m\leq\frac{k}{V}$, suppose that the Dirac operator $\check{D}_r$ on ${\mathcal{S}}_{k,m}$ has an eigenvalue $\lambda_0$ with $|\lambda_0|\leq1$, then
\begin{align*} \Big|\lambda_0 - (\frac{r}{2} - \frac{\check{\gamma}_{k,m}}{2})\Big|\leq c \exp(-\frac{r}{c}) ~. \end{align*}
On the other hand, if $\check{\gamma}_{k,m}\geq c$, $|r-\check{\gamma}_{k,m}|<1$, and $k$ and $m$ satisfy the same constraint, the Dirac operator $\check{D}_r$ on ${\mathcal{S}}_{k,m}$ does have an eigenvalue $\lambda_0$ satisfying the above estimate.
\end{lem}

Similarly, for the associated contact form of the Dehn-twist region, the Dirac operator is already linear on the region where $-35\epsilon\leq\rho\leq-10\epsilon$ or $10\epsilon\leq\rho\leq35\epsilon$.
\begin{lem}\label{lem_refined_r_neck_01}
There exists a constant $c$ with the following significance.  For and $r\geq c$ and
\begin{align*}
&{\textstyle\frac{2(v-31\epsilon)}{V+2(v-31\epsilon)\sigma(20\epsilon)}k\leq m\leq\frac{2(v-20\epsilon)}{V+2(v-20\epsilon)\sigma(20\epsilon)}k} \quad \text{or}\\
&{\textstyle\frac{2(v+20\epsilon)}{V+2(v+20\epsilon)\sigma(-20\epsilon)}k\leq m\leq\frac{2(v+31\epsilon)}{V+2(v+31\epsilon)\sigma(-20\epsilon)}k} ~,
\end{align*}
suppose that the Dirac operator $\tilde{D}_r$ on ${\mathcal{S}}_{k,m}$ has an eigenvalue $\lambda_0$ with $|\lambda_0|\leq1$, then
\begin{align*} \Big|\lambda_0 - (\frac{r}{2} - \frac{\tilde{\gamma}_{k,m}}{2})\Big|\leq c \exp(-\frac{r}{c}) . \end{align*}
On the other hand, if $\tilde{\gamma}_{k,m}\geq c$, $|r-\tilde{\gamma}_{k,m}|<1$, and $k$ and $m$ satisfy the same constraint, the Dirac operator $\tilde{D}_r$ on ${\mathcal{S}}_{k,m}$ does have an eigenvalue $\lambda_0$ satisfying the above estimate.
\end{lem}
Note that for $k$ and $m$ in the first range, $\tilde{\gamma}_{k,m} = \frac{2(k+m \sigma(20\epsilon))}{V}$.  For $k$ and $m$ in the second range, $\tilde{\gamma}_{k,m} = \frac{2(k+m \sigma(-20\epsilon))}{V}$.

%%%%
\subsection{Contact forms with two $S^1$-symmetry} \label{subsec_other_contact_00}
The method in this section works for the contact forms that are invariant under two global $S^1$-actions.  For instance, one can use the same method to prove theorem \ref{thm_sf_main_00} for the overtwisted contact form in \cite{ref_Taubes_S1S2}, or some contact forms on $T^3$.  There are two main differences:
\begin{enumerate}
\item the frequency $k$ might be negative;
\item there might be more than one zero crossing on each $\mathcal{S}_{k,m}$.  But the number is decided by $f(\rho)$ and $g(\rho)$, and the zero modes peak at different region.
\end{enumerate}

With this understood, the condition (\ref{eqn_contact_Dehn_00}) is a shortcut for dealing the Dehn-twist region.  It ensures the positivity of $\tilde{f}$.  If $\tilde{f}$ is not always positive, we can still extend (the untwisting of) $a$ to a contact form of the type (\ref{eqn_S2S1_contact_00}) on $S^1\times S^2$.  The frequency $k$ can be negative, and it requires more work to discuss it.

Here is a remark from the viewpoint of contact topology.  With the terminology of Giroux correspondence \cite{ref_Giroux}, our associated contact form is supported by an annulus with the identity map, and thus Stein fillable.  If $\tilde{f}$ is not always positive, the extension ends up with an overtwisted contact form on $S^1\times S^2$.

%%%%%%%%
\section{Lower bound of the spectral flow}\label{sec_lower_bound}
We are going to prove a stronger statement which implies the lower bound in theorem \ref{thm_sf_main_00}.
\begin{defn}\label{defn_index_part_00}
Let us introduce the following notions:
\begin{enumerate}
\item for the associated contact form of the tubular neighborhood of the binding, let $\check{I}(r',r)$ be the total number of zero crossings of $\check{D}_r$ happening within the interval $(r',r]$ and on $\mathcal{S}_{k,m}$ with $m\geq\frac{k}{V}$;
\item for the associated contact form of the Dehn-twist region, let $\tilde{I}(r',r)$ be the total number of zero crossings of $\tilde{D}_r$ happening within the interval $(r',r]$ and on $\mathcal{S}_{k,m}$ with
$${ \frac{2(v-20\epsilon)}{V-2(v-20\epsilon)\sigma(20\epsilon)} k\leq m\leq\frac{2(v+20\epsilon)}{V-2(v+20\epsilon)\sigma(-20\epsilon)} k} ~;$$
\item for any $n\in\mathbb{N}$, let $I_\Sigma(n)$ be the dimension of $\ker{\bar{\pl}_n}$.
\end{enumerate}
\end{defn}

If we fix $r'=0$, these functions obey the following estimates:
\begin{lem}\label{lem_bound_counting_00}
There exists a constant $c>0$ which has the following significance:
\begin{align*}
\Big| \check{I}(0,r) - \frac{r^2}{4}\int_0^1 \Delta\dd\rho \Big| &\leq c r ~, \\
\Big| \tilde{I}(0,r) - \frac{r^2}{4}\int_{-20\epsilon}^{20\epsilon} \tilde{\Delta}\dd\rho \Big| &\leq c r ~, \\
\Big| \sum_{n=1}^{[\frac{V}{2}r]}I_\Sigma(n) - \frac{V r^2}{8\pi} \int_\Sigma\dd\mu_\Sigma \Big| &\leq c r
\end{align*}
for all $r\geq c$.  The functions $\Delta$ and $\tilde{\Delta}$ are defined by (\ref{eqn_S2S1_Delta_00}).
\end{lem}
\begin{proof}
The proof for the first two inequalities is similar to that of theorem \ref{thm_sf_S1S2_00}.  We briefly explain it for $\check{I}(0,r)$, and use the same notation as that in the proof of theorem \ref{thm_sf_S1S2_00}.  There exists a constant $c_1$ such that $\check{I}(0,r)$ is less than
\begin{align*}
&\text{Area}\big(\{\rho\leq 1 \text{ and }\check{\gamma}(s,\rho)\leq r + c_1\}\big) + c_1r \\
&\quad + \text{Area}\big(\{\frac{k}{V}\leq m \leq \frac{k}{V}+1 \text{ and }\check{\gamma}(s,\rho)\leq r + c_1\}\big)
\end{align*}
for all $r\geq c_1$.  The first term is given by the integral of $\Delta$.  It is not hard to see that the third term is less than $c_2 r$ for some constant $c_2$.

The third inequality on $I_\Sigma(n)$ follows directly from the index formula (\ref{eqn_APS_00}) and lemma \ref{lem_beta_vanish_00}.
\end{proof}

Here comes the main theorem of this section.
\begin{thm}\label{thm_sf_lower_bound}
There exist a constant $c$ and a sequence of numbers $\{s_n\}_{n\in\mathbb{N}}$ which have the following significance:
\begin{enumerate}
\item for all $n\geq 2$,
\begin{align*} {\rm sf}_a(s_n) - {\rm sf}_a(s_{n-1}) \geq I_\Sigma(n) + \check{I}(s_{n-1},s_n) + \tilde{I}(s_{n-1},s_n) - c ~; \end{align*}
\item for each $n\in\mathbb{N}$, $ |s_n - \gamma_n -\frac{1}{V}|\leq \frac{1}{4V} $, where $\gamma_n = \frac{2n}{V}$ as in (\ref{defn_gamma_n}).
\end{enumerate}
\end{thm}

It is clear that the lower bound of the spectral flow function claimed in theorem \ref{thm_sf_main_00} follows from theorem \ref{thm_sf_lower_bound} and lemma \ref{lem_bound_counting_00}.

We will prove theorem \ref{thm_sf_lower_bound} by constructing almost eigensections.  The following lemma measures the difference between true eigenvalues and almost eigenvalues.  It is an issue of linear algebra, but we state it for a Dirac operator.
\begin{lem}\label{lem_true_eigenvalue_00}
Let $\mathcal{D}$ be a Dirac operator on the bundle $\mathbb{S}$.  If there exist a constant $\delta_4$, a finite number of smooth sections $\{\psi_l\}_{l=1}^L$ of $\mathbb{S}$, and real numbers $\{\mu_l\}_{l=1}^L$ with the properties:
\begin{enumerate}
\item $\{\psi_l\}_{l=1}^L$ is an orthonormal set with respect to the $L^2$-inner product;\smallskip
\item $\int | \mathcal{D}\psi_l - \mu_l\psi_l |^2 \leq \delta_4$ for all $l$;\smallskip
\item $\int \langle \mathcal{D}\psi_l, \psi_{l'} \rangle = 0$ for all $l\neq l'$;\smallskip
\item $\sum\limits_{\substack{1\leq l,l'\leq L \\ \text{and } l\neq l'}} |\int \langle \mathcal{D}\psi_l, \mathcal{D}\psi_{l'} \rangle| \leq \delta_4$.
\end{enumerate}
Then, there exist $L$ eigenvalues (counting multiplicity) $\{\lambda_l\}_{l=1}^L$ of $\mathcal{D}$ such that $|\lambda_l - \mu_l| \leq \sqrt{2 \delta_4}$ for all $l$.
\end{lem}
\begin{proof}
Clearly the lemma is true for $L=1$.  Suppose the lemma holds for $L-1$, we are going to show that it is true for $L$.  Without loss of generality, we may assume that $\{\mu_l\}_{l=1}^L$ is non-decreasing in $l$.

For each $l\in\{1,2,\cdots,L\}$, remove $\psi_l$ and $\mu_l$, and apply the lemma.  If there are $L$ eigenvalues (counting multiplicity), we are done.  If there are only $(L-1)$ eigenvalues, $\{\lambda_l\}_{l=1}^{L-1}$, they must satisfy
\begin{align*}
|\lambda_l - \mu_l| &\leq \sqrt{2 \delta_4} &&\text{and} &|\lambda_l - \mu_{l+1}| &\leq \sqrt{2 \delta_4}
\end{align*}
for all $l\in\{1,2,\cdots,L-1\}$.  The triangle inequalities implies that
$$|\mu_l - \mu_{l+1}| \leq 2\sqrt{2 \delta_4}$$
for all $l\in\{1,2,\cdots,L-1\}$.

Suppose that $\{e_l\}_{l=1}^{L-1}$ are the eigensections corresponding to $\{\lambda_l\}_{l=1}^{L-1}$.  There exist complex numbers $\{\mathfrak{c}_l\}_{l=1}^{L}$ such that $\sum_{l=1}^{L}|\mathfrak{c}_l|^2 = 1$, and $\sum_{l=1}^{L}\mathfrak{c}_l\psi_l$ is orthogonal to $e_l$ for all $l\in\{1,2,\cdots,L-1\}$.  For any real number $\nu$, the operator $\mathcal{D}-\nu$ on $\sum_{l=1}^{L}\mathfrak{c}_l\psi_l$ satisfies the estimate:
\begin{align*}
\int \big| (\mathcal{D}-\nu)(\sum_{l=1}^{L}\mathfrak{c}_l\psi_l) \big|^2 &\leq \sum_{l=1}^L |\mathfrak{c}_l|^2 \int \big| (\mathcal{D}-\nu)\psi_l \big|^2 + \sum_{\substack{1\leq l,l'\leq L \\ \text{and } l\neq l'}} \int \big|\langle \mathcal{D}\psi_l, \mathcal{D}\psi_{l'} \rangle\big| \\
&\leq \sum_{l=1}^L |\mathfrak{c}_l|^2 (\sqrt{\delta_4} + |\nu-\mu_l|)^2 + \delta_4 \\
&\leq (\sqrt{2 \delta_4} + \max_l |\nu-\mu_l|)^2 ~.
\end{align*}
Consider $\nu = \frac{\mu_1+\mu_L}{2}$, the above inequality produces another eigenvalue $\lambda_L$ with
\begin{align*}
\mu_1 - \sqrt{2 \delta_4} \leq \lambda_L \leq \mu_L + \sqrt{2 \delta_4} ~,
\end{align*}
and the eigensection associated to $\lambda_L$ is orthogonal to $\{e_l\}_{l=1}^{L-1}$.  Therefore, there exist some $l\in\{1,2,\cdots,L\}$ such that $|\lambda_L - \mu_l| \leq \sqrt{2 \delta_4}$.  After re-numbering the indices of $\{\lambda_l\}_{l=1}^L$, these $L$ eigenvalues satisfy the assertion of the lemma.
\end{proof}

The following proposition is the prototype of theorem \ref{thm_sf_lower_bound}.

\begin{prop}\label{prop_sf_lower_bound}
There exist a constant $c>0$ such that the following holds.  For any integer $n\geq c$ and any $\delta_5^-, \delta_5^+\in[\frac{1}{2V},\frac{3}{2V}]$, let $\gamma_n = \frac{2n}{V}$ as in (\ref{defn_gamma_n}), then
\begin{align*}
&{\rm sf}_a(\gamma_n+(\delta_5^++\frac{c}{\gamma_n})) - {\rm sf}_a(\gamma_n-(\delta_5^-+\frac{c}{\gamma_n})) \\
\geq\, & I_\Sigma(n) + \check{I}(\gamma_n-\delta_5^-,\gamma_n+\delta_5^+) + \tilde{I}(\gamma_n-\delta_5^-,\gamma_n+\delta_5^+)
\end{align*}
\end{prop}
\begin{proof}
The proof contains three steps.

\subsubsection*{Step 1}
This step constructs almost eigensections of $D_{\gamma_n}$ from those three terms on the right hand side.

\emph{From the page}.  By proposition \ref{prop_page_sln_00}, for any $n\geq c_2$, there exists a $L^2$-orthonormal set of sections $\{\psi_{n,l}\}$ where $l\in\{1,2,\cdots,I_\Sigma(n)\}$, with
\begin{align} \label{eqn_glue_sln_01} \int_Y|D_{\gamma_n}\psi_{n,l}|^2\leq c_2\exp(-\frac{\gamma_n}{c_2}) ~, \end{align}
$\int_Y\langle D_{\gamma_n}\psi_{n,l},\psi_{n,l'}\rangle = 0$ and $|\int_Y\langle D_{\gamma_n}\psi_{n,l},D_{\gamma_n}\psi_{n,l'}\rangle|\leq c_2\exp(-\frac{\gamma_n}{c_2})$ for any $l\neq l'$.

\emph{From the tubular neighborhood of the binding}.  If $\check{D}_r$ has a zero crossing at $\gamma\in(\gamma_n-\delta_5^-,\gamma_n+\delta_5^+]$ on $\mathcal{S}_{k,m}$ with $m\geq\frac{k}{V}$, proposition \ref{prop_small_eigensection_00} and \ref{prop_small_eigensection_01} for $r=\gamma$ imply that there exists a constant $c_3$ such that
\begin{align*} \big| \frac{\gamma}{2} - \frac{\check{\gamma}_{k,m}}{2} + \frac{\mathfrak{r}_1}{2\check{\gamma}_{k,m}} \big|\leq c_3\gamma^{-1}\leq {2c_3}{\gamma_n^{-1}} \end{align*}
provided $n\geq c_3$.  Then apply lemma \ref{lem_exist_sl_00} and lemma \ref{lem_exist_sl_01} for $r=\gamma_n$ to find a constant $c_4$ and a section $\check{\psi}_{k,m}$ such that
\begin{align*}
\int_{\check{Y}} \big| \check{D}_{\gamma_n} \check{\psi}_{k,m} - (\frac{\gamma_n}{2}-\frac{\check{\gamma}_{k,m}}{2}+\frac{\mathfrak{r}_1}{2\check{\gamma}_{k,m}})\check{\psi}_{k,m} \big|^2 &\leq c_4 \gamma_n^{-2} \int_{\check{Y}} |\check{\psi}_{k,m}|^2
\end{align*}
and $\big| 1 - \int_{\check{Y}} |\check{\psi}_{k,m}|^2 \big| \leq c_4 \gamma_n^{-1}$.  Using the triangle inequality, we have
\begin{align}\label{eqn_glue_sln_00}
\int_Y \big| D_{\gamma_n} \check{\psi}_{k,m} - (\frac{\gamma_n}{2}-\frac{\gamma}{2})\check{\psi}_{k,m} \big|^2 &\leq c_5 \gamma_n^{-2} \int_Y |\check{\psi}_{k,m}|^2
\end{align}
for some constant $c_5$.  The section $\check{\psi}_{k,m}$ can be regarded as being defined on $Y$.

\emph{From the Dehn-twist region}.  The same construction as the tubular neighborhood of the binding produces sections $\tilde{\psi}_{k,m}$.  After undoing the untwisting (\ref{def_untwisting}) on $\tilde{\psi}_{k,m}$, they can be regarded as sections on $Y$, and also obey (\ref{eqn_glue_sln_00}).

\subsubsection*{Step 2}  In order to apply lemma \ref{lem_true_eigenvalue_00} on $D_{\gamma_n}$ and the sections constructed in step 1, we need to check that they meet the conditions of lemma \ref{lem_true_eigenvalue_00}.

\emph{Condition} (i).  The orthogonality is clear between any two sections constructed from the same region, and between one section from $\check{Y}$ and another section from $\tilde{Y}$.

For a section $\psi_{n,l}$ from the page and another section $\check{\psi}_{k,m}$ from $\check{Y}$, their $L^2$-inner product can be nonzero only when $k=n$.  For the section $\check{\psi}_{n,m}$, $m$ is required to be greater than or equal to $\frac{n}{V}$.  It is complementary to the APS boundary condition for $\psi_{n,l}$, see (\ref{eqn_APS_bdry_00}) and (\ref{eqn_sln_extension_00}).  Therefore, $\int_Y\langle\psi_{n,l},\check{\psi}_{k,m}\rangle = 0$.

For a section from the page and another section from $\tilde{Y}$, the argument is basically the same.  At a first glance, the condition in definition \ref{defn_index_part_00} (ii) does not seem to match with the boundary conditions (\ref{eqn_APS_bdry_01}) and (\ref{eqn_APS_bdry_02}).  However, the untwisting operator (\ref{def_untwisting}) shifts the frequency.  A direct computation shows that these conditions are complementary to each other.

\emph{Condition} (ii).  By (\ref{eqn_glue_sln_01}), the almost eigenvalues of $\{\psi_{n,l}\}$ are all equal to $0$.  By (\ref{eqn_glue_sln_00}), the almost eigenvalues are equal to $\frac{1}{2}(\gamma_n-\gamma)$ for $\check{\psi}_{k,m}$ and $\tilde{\psi}_{k,m}$, and $\gamma$ is where the zero crossing happens of $\check{D}_r$ or $\tilde{D}_r$ on $\mathcal{S}_{k,m}$.  The error term $\delta_4$ is $c_6\gamma_n^{-2}$ for some constant $c_6$.

\emph{Condition} (iii).  For any two sections from the page, it is given by proposition \ref{prop_page_sln_00}.  The arguments for other situations are the same as that for requirement (i).

\emph{Condition} (iv). The $L^2$-inner product can only be nonzero between two sections from $\Sigma$.  With the help of proposition \ref{prop_page_sln_00}, the summation is less than $2I_\Sigma(n) c_2 \exp(-\frac{\gamma_n}{c_2})\leq c_6\gamma_n^{-2}$.

The $L^2$-norm of some sections are not equal to $1$, but almost.  It can be easily fixed by normalizing these almost eigensections.  The constant $c_6$ is replaced by $c_6(1+c_6'\gamma_n^{-1})$, which is still uniformly bounded.

\subsubsection*{Step 3}  With these almost eigensections, lemma \ref{lem_true_eigenvalue_00} gives the following eigenvalues for $D_{\gamma_n}$:
\begin{itemize}
\item there are $I_\Sigma(n)$ eigenvalues whose magnitude is less than $ \sqrt{2c_6}\gamma_n^{-1}$;
\item if $\check{I}(\gamma_n-\delta_5^-,\gamma_n+\delta_5^+)$ or $\tilde{I}(\gamma_n-\delta_5^-,\gamma_n+\delta_5^+)$ gets a spectral flow count at $\gamma$, there is an eigenvalue $\lambda$ associating to it, with
$$\big| \lambda - \frac{\gamma_n-\gamma}{2} \big| \leq \sqrt{2c_6}{\gamma_n^{-1}} ~.; $$
\item all the above eigenvalues are different.
\end{itemize}
Note that the magnitude of these eigenvalues are less than $\delta_5^+ + \delta_5^-\leq\frac{3}{C}$ provided $n\geq 10 c_5 V^2$.  Let $c_7$ be the constant given by corollary \ref{cor_sf_exist_00} for $\delta_1 = \frac{12}{V}$, then
\begin{align*}
&{\rm sf}_a(\gamma_n+(\delta_5^++\frac{\sqrt{2c_6}+c_7}{\gamma_n})) - {\rm sf}_a(\gamma_n-(\delta_5^-+\frac{\sqrt{2c_6}+c_7}{\gamma_n})) \\
\geq\, & I_\Sigma(n) + \check{I}(\gamma_n-\delta_5^-,\gamma_n+\delta_5^+) + \tilde{I}(\gamma_n-\delta_5^-,\gamma_n+\delta_5^+) ~.
\end{align*}
This completes the proof of proposition \ref{prop_sf_lower_bound}. 
\end{proof}

We now prove the main theorem of this section.
\begin{proof}[Proof of thoerem \ref{thm_sf_lower_bound}]
Let $c_8$ be the constant given by proposition \ref{prop_sf_lower_bound}.  Lemma \ref{lem_divide_interval_00} with $\delta_3 = c_8$ gives a constant $c_9$ and a sequence $\{s_n\}_{n\in\mathbb{N}}$ such that $|s_n-\gamma_n-\frac{1}{V}|\leq \frac{1}{4V}$ and
\begin{align}\begin{split}\label{eqn_pf_sf_lower_00}
\check{I}(s_n - \frac{c_8}{s_{n-1}}, s_n + \frac{c_8}{s_{n-1}}) \leq c_9 ~, \\
\tilde{I}(s_n - \frac{c_8}{s_{n-1}}, s_n + \frac{c_8}{s_{n-1}}) \leq c_9
\end{split}\end{align}
for all $n\in\mathbb{N}$.  If $n\geq10c_8 V^2$, $\delta_5^- = \gamma_n - s_{n-1} - \frac{c_8}{\gamma_n}$ and $\delta_5^+ = s_n - \gamma_n - \frac{c_8}{\gamma_n}$ meet the requirement of proposition \ref{prop_sf_lower_bound}.  Hence,
\begin{align*}
&{\rm sf}_a(s_n) - {\rm sf}_a(s_{n-1}) \\
\geq\,& I_\Sigma(n) + \check{I}(s_{n-1} + \frac{c_8}{\gamma_n}, s_n - \frac{c_8}{\gamma_n}) + \tilde{I}(s_{n-1} + \frac{c_8}{\gamma_n}, s_n - \frac{c_8}{\gamma_n}) \\
\geq\,& I_\Sigma(n) + \check{I}(s_{n-1},s_n) + \tilde{I}(s_{n-1},s_n) - 4c_9
\end{align*}
provided $n\geq10c_8 V^2$.  The last inequality follows from (\ref{eqn_pf_sf_lower_00}) and $\gamma_n>s_{n-1}>s_{n-2}$.  This completes the proof of theorem \ref{thm_sf_lower_bound}.
\end{proof}

%%%%%%%%
\section{The Dirac operator of trivial monodromy}\label{sec_dirac_page_00}
In order to prove the upper bound on the spectral flow function, we need a further understanding of the Dirac operator on $\Sigma\times S^1$.  With the help of the $S^1$-action, the Dirac operator reduces to Cauchy--Riemann operators (\ref{eqn_Dirac_page_00}), which we denote by $\mathcal{D}_{r,n}$.  This section follows the same notations as that in section \ref{sec_page_id_00}.

We first establish two estimates on $(\hat{\alpha}_n,\hat{\beta}_n)\in\mathcal{C}^\infty(\underline{\mathbb{C}}\oplus K_\Sigma^{-1})$ in terms of $\mathcal{D}_{r,n}(\hat{\alpha}_n,\hat{\beta}_n)$.  Roughly speaking, the estimates imply that if $\mathcal{D}_{r,n}(\hat{\alpha}_n,\hat{\beta}_n)$ is small, then $r$ is close to $\gamma_n={2n}/{V}$, and $\check{\alpha}_n$ almost solves $\bar{\pl}_n$.

We first consider the case when $|r-\gamma_n|\geq r^\oh$.

\begin{lem}\label{lem_Drn_rn_large}
There exists a constant $c$ with the following property.  For any $r\geq c$ and integer $n$ with $|r-\gamma_n|\geq r^\oh$,
\begin{align*} \int_\Sigma|\hat{\alpha}_n|^2+|\hat{\beta}_n|^2 \leq {c}{r^{-1}} \int_\Sigma |\mathcal{D}_{r,n}(\hat{\alpha}_n,\hat{\beta}_n)|^2 \end{align*}
for any $\hat{\alpha}_n$ and $\hat{\beta}_n$ satisfying the APS boundary condition for $\bar{\pl}_n$ and $\bar{\pl}_n^*$, respectively.  $\gamma_n$ is defined by (\ref{defn_gamma_n}).
\end{lem}
\begin{proof}
With the APS boundary condition, $\bar{\pl}_n$ and $\bar{\pl}_n^*$ are adjoint operators.  Performing integration by parts, we have
\begin{align}\begin{split}\label{eqn_Drn_by_parts}
\int_\Sigma |\mathcal{D}_{r,n}(\hat{\alpha}_n,\hat{\beta}_n)|^2 &= \int_\Sigma (\frac{r-\gamma_n}{2})^2 |\hat{\alpha}_n|^2 + |\bar{\pl}_n^*\hat{\beta}_n|^2 - \langle\hat{\bar}{\pl}_n^*\hat{\beta}_n,\hat{\alpha}_n\rangle\\
&\qquad + |\bar{\pl}_n\hat{\alpha}_n|^2+ (\frac{r-\gamma_n-2}{2})^2|\hat{\beta}_n|^2 - \langle\bar{\pl}_n\hat{\alpha}_n,\hat{\beta}_n\rangle ~.
\end{split}
\end{align}
With the Cauchy--Schwarz inequality, it completes the proof of this lemma.
\end{proof}

We then consider the case when $|r-\gamma_n|\leq r^\oh$.

\begin{prop}\label{prop_Drn_rn_small}
There exists a constant $c$ such that the following holds.  For any $r\geq c$ and integer $n$ with $|r-\gamma_n|\leq r^\oh$, suppose that $\hat{\alpha}_n$ and $\hat{\beta}_n$ vanish near $\pl\Sigma$.  Then
\begin{align*} \int_\Sigma \big|\hat{\alpha}_n-{\hat{\rm pr}_n}(\hat{\alpha}_n)\big|^2 + \big|\hat{\beta}_n\big|^2 \leq {c}{r^{-1}} \int_\Sigma \big|\mathcal{D}_{r,n}(\hat{\alpha}_n,\hat{\beta}_n)\big|^2 \end{align*}
where ${\hat{\rm pr}_n}$ is the $L^2$-orthogonal projection onto the kernel of $\bar{\pl}_n$.  Moreover, if $r\neq\gamma_n$,
\begin{align*} \int_\Sigma \big|{\hat{\rm pr}_n}(\hat{\alpha}_n) \big|^2 \leq {4}{(r-\gamma_n)^{-2}} \int_\Sigma \big|({{\rm pr}_1}\circ{\mathcal{D}_{r,n}})(\hat{\alpha}_n,\hat{\beta}_n)\big|^2 \end{align*}
where ${\rm pr}_1$ is the projection onto the first component.
\end{prop}
\begin{proof}
The condition $|r-\gamma_n|\leq r^\oh$ implies that $n\geq\frac{V}{4}r$, provided $r\geq 4$.  According to lemma \ref{lem_beta_vanish_00}, there exists a constant $c_1>0$ such that
\begin{align*} \int_\Sigma |\hat{\beta}_n|^2 \leq {c_1}{r^{-1}}\int_\Sigma |\bar{\pl}_n^*\hat{\beta}_n|^2 \end{align*}
for any $r\geq c_1$.  According to \cite[p.51 and p.56]{ref_APS}, $\bar{\pl}_n\bar{\pl}_n^*$ and $\bar{\pl}_n^*\bar{\pl}_n$ have the same non-zero eigenvalues.  Hence,
\begin{align}\begin{split}\label{eqn_dbarn_estimate_00} \int_\Sigma |\hat{\alpha}_n-{\hat{\rm pr}_n}(\hat{\alpha}_n)|^2 &\leq {c_1}{r^{-1}}\int_\Sigma \big|\bar{\pl}_n(\hat{\alpha}_n-{\hat{\rm pr}_n}(\hat{\alpha}_n))\big|^2 \\
&= {c_1}{r^{-1}}\int_\Sigma |\bar{\pl}_n\hat{\alpha}_n|^2 ~. \end{split}\end{align}
It follows from (\ref{eqn_Drn_by_parts}) and the above two inequalities that
\begin{align*}
\int_\Sigma |\mathcal{D}_{r,n}(\hat{\alpha}_n,\hat{\beta}_n)|^2 &\geq \int_\Sigma (\frac{r-\gamma_n}{2})^2 |\hat{\alpha}_n|^2 + \frac{r}{c_1}|\hat{\beta}_n|^2 - 2|\hat{\beta}_n|^2\\
&\qquad + \frac{r}{2c_1} |\hat{\alpha}_n-{\hat{\rm pr}_n}(\hat{\alpha}_n)|^2+ (\frac{r-\gamma_n-2}{2})^2|\hat{\beta}_n|^2 ~.
\end{align*}
This proves the first assertion of this lemma.

For the second assertion, note that
\begin{align*}
\mathcal{D}_{r,n}\big({\hat{\rm pr}_n}(\hat{\alpha}_n),0\big) &= \Big(\frac{r-\gamma_n}{2}{\hat{\rm pr}_n}(\hat{\alpha}_n),0\Big) ~, \\
\mathcal{D}_{r,n}\big(\hat{\alpha}_n-{\hat{\rm pr}_n}(\hat{\alpha}_n),\hat{\beta}_n\big) &= \Big(\frac{r-\gamma_n}{2}\big(\hat{\alpha}_n-{\hat{\rm pr}_n}(\hat{\alpha}_n)\big) + \bar{\pl}_n^*\check{\beta}_n,\cdots\Big) ~.
\end{align*}
Therefore, $\mathcal{D}_{r,n}$ preserves the $L^2$-orthogonality between $\big({\hat{\rm pr}_n}(\hat{\alpha}_n),0\big)$ and $\big(\hat{\alpha}_n-{\hat{\rm pr}_n}(\hat{\alpha}_n),\hat{\beta}_n\big)$.  The desired estimate on ${\hat{\rm pr}_n}(\hat{\alpha}_n)$ follows.
\end{proof}

In next section, we will study the zero modes on $Y$ through $\Sigma\times S^1$, $\check{Y}$ and $\tilde{Y}$.  The cut-off function causes some overlaps of these models.  To tackle this issue, we need to study $\ker\bar{\pl}_n$ carefully.

As discussed in section \ref{sec_page_id_00} and (\ref{eqn_sln_extension_00}), any solution of $\bar{\pl}_n$ on $\Sigma_{-11\epsilon}$ naturally extends to a solution on $\Sigma$, and still satisfies the corresponding APS boundary condition.  Let $\ker_{0}{\bar{\pl}_n}$ be the subspace of $\ker{\bar{\pl}_n}$ which are extended from $\Sigma_{-11\epsilon}$.  Consider the following sections which peak on $\Sigma\backslash\Sigma_{-11\epsilon}$.

\subsubsection*{Adjacent to the tubular neighborhood of the binding}
For any integers $n>0$ and $m$ with $\frac{n}{V}(1-11\epsilon)\leq m < \frac{n}{V}$, let
\begin{align}\label{eqn_almost_nm_page_00} \check{\zeta}_{n,m} = \chi\big(\epsilon(\rho-2+\frac{Vm}{n})\big) \big(\frac{2n}{V\pi^3}\big)^{\frac{1}{4}} \exp\big(-\frac{n}{V}(\rho-2+\frac{Vm}{n})^2\big)e^{imt} ~. \end{align}
\subsubsection*{Adjacent to the Dehn-twist region}
For any integers $n>0$ and $m$ with $\frac{2n}{V}(v+31\epsilon)\geq m > \frac{2n}{V}(v+20\epsilon)$ or $\frac{2n}{V}(v-31\epsilon)\leq m<\frac{2n}{V}(v-20\epsilon)$, let
\begin{align}\label{eqn_almost_nm_page_01} \tilde{\zeta}_{n,m} = \chi\big(\epsilon(\rho-v+\frac{Vm}{2n})\big) \big(\frac{n}{2V\pi^3}\big)^{\frac{1}{4}} \exp\big(-\frac{2n}{V}(\rho-v+\frac{Vm}{2n})^2\big)e^{imt} ~. \end{align}\smallskip

\begin{lem}\label{lem_decomp_sigma_00}
There exists a constant $c>0$ with the following significance.  For any $n\geq c$, the kernel of $\bar{\pl}_n$ has the orthonormal basis
\begin{align*}
\{\text{orthornomal basis of } \ker_{0}{\bar{\pl}_n}\} \oplus \{ \check{\mathfrak{p}}_{n,m}\check{\zeta}_{n,m} + \check{\zeta}_{n,m}^{\text{\rm rem}} \} \oplus \{ \tilde{\mathfrak{p}}_{n,m}\tilde{\zeta}_{n,m} + \tilde{\zeta}_{n,m}^{\text{\rm rem}} \}
\end{align*}
with respect to the $L^2$-inner product on $\Sigma$.  The range of $m$ for the second summand is $\{\frac{n}{V}(1-11\epsilon)\leq m < \frac{n}{V}\}$; the range of $m$ for the third summand is $\{\frac{2n}{V}(v+31\epsilon)\geq m > \frac{2n}{V}(v+20\epsilon)\}$ and $\{\frac{2n}{V}(v-31\epsilon)\leq m<\frac{2n}{V}(v-20\epsilon)\}$.  The elements in the decomposition have the following features:
\begin{enumerate}
\item $\check{\mathfrak{p}}_{n,m}$ is a constant between $\oh$ and $2$, and $\int_\Sigma |\check{\zeta}_{n,m}^{\text{\rm rem}}|^2 \leq c\exp(-\frac{n}{c})$;\smallskip
\item $\tilde{\mathfrak{p}}_{n,m}$ is a constant between $\oh$ and $2$, and $\int_\Sigma |\tilde{\zeta}_{n,m}^{\text{\rm rem}}|^2 \leq c\exp(-\frac{n}{c})$;\smallskip
\item for any $\alpha_n\in\ker_0\bar{\pl}_n$, $\int_{\Sigma\backslash\Sigma_{-2\epsilon}} |\alpha_n|^2 \leq c\exp(-\frac{n}{c})\int_\Sigma |\alpha_n|^2$.
\end{enumerate}
\end{lem}
\begin{proof}
Consider the orthogonal set:
\begin{align}\label{eqn_number_kern}
\{\text{orthonormal basis of }\ker_0\bar{\pl}_n\}\oplus\{\check{\zeta}_{n,m}\}\oplus\{\tilde{\zeta}_{n,m}\}~.
\end{align}
By the APS index theorem \cite{ref_APS}, the total number of (\ref{eqn_number_kern}) is equal to the dimension of $\ker\bar{\pl}_n$:  The first summand form a basis for $\ker\bar{\pl}_n$ on $\Sigma_{-11\epsilon}$ with the corresponding boundary condition.  Thus, the total number of the first summand can be computed by the APS index formula (\ref{eqn_APS_00}) and lemma \ref{lem_beta_vanish_00}.  Similar to (\ref{eqn_sln_extension_00}), the last two summands (without the cut-off function) form a basis for $\ker\bar{\pl}_n$ on $\Sigma\backslash\Sigma_{-11\epsilon}$.  The total number can also be computed by the APS index formula.  If we sum up the index formulae, the boundary contribution from $\pl\Sigma_{-11\epsilon}$ cancels with each other, and it turns out to be the index formula of $\ker\bar{\pl}_n$ on $\Sigma$.

The elements in the last two summand have $L^2$-norm between $\frac{1}{\sqrt{3}}$ and $1$.  They are not annihilated by $\bar{\pl}_n$.  We modify them by the following procedure.

Start with an orthonormal basis of $\ker_0{\bar{\pl}_n}$, and take any $\check{\zeta}_{n,m}$.  A direct computation shows that there exists a constant $c_2$ such that
\begin{align*} \int_\Sigma|\bar{\pl}_n \check{\zeta}_{n,m}|^2 \leq c_2\exp(-\frac{n}{c_2}) ~. \end{align*}
Let ${\check{\rm pr}_n}(\check{\zeta}_{n,m})$ be the $L^2$-orthogonal projection of $\check{\zeta}_{n,m}$ onto $\ker\bar{\pl}_n$.  By (\ref{eqn_dbarn_estimate_00}), there exists a constant $c_3$ such that
\begin{align*} \int_\Sigma \big|(\check{\zeta}_{n,m} - {\check{\rm pr}_n}(\check{\zeta}_{n,m}))\big|^2\leq c_3\exp(-\frac{n}{c_3}) ~. \end{align*}
If we apply the Gram--Schmidt process on
$$ {\check{\rm pr}_n}(\check{\zeta}_{n,m}) = \check{\zeta}_{n,m} + ({\check{\rm pr}_n}(\check{\zeta}_{n,m})-\check{\zeta}_{n,m}) $$
with respect to the orthonormal basis of $\ker_0{\bar{\pl}_n}$, the output would be $\check{\mathfrak{p}}_{n,m}\check{\zeta}_{n,m} + \check{\zeta}_{n,m}^{\text{\rm rem}}$ satisfying property (i).

We can keep doing this projection and Gram--Schmidt process.  Since the total number of steps is less than $n$, the error term is always less than $c_4\exp(-\frac{n}{c_4})$ for some constant $c_4$.  It produces an orthonormal basis for $\ker\bar{\pl}_n$ with property (i) and (ii).

The proof of property (iii) is the same as that for proposition \ref{prop_page_sln_00}
\end{proof}

%%%%%%%%
\section{Upper bound of the spectral flow}
What follows is the main theorem of this section.  With lemma \ref{lem_bound_counting_00}, it implies the upper bound of the spectral flow function asserted by theorem \ref{thm_sf_main_00}.
\begin{thm}\label{thm_sf_upper_bound}
There exist a constant $c$ and a sequence of numbers $\{s_n\}_{n\in\mathbb{N}}$ which have the following significance:
\begin{enumerate}
\item for all $n\geq 2$,
\begin{align*} {\rm sf}_a(s_n) - {\rm sf}_a(s_{n-1}) \leq I_\Sigma(n) + \check{I}(s_{n-1},s_n) + \tilde{I}(s_{n-1},s_n) + c ~; \end{align*}
\item for each $n\in\mathbb{N}$, $ |s_n - \gamma_n - \frac{1}{V}|\leq \frac{1}{4V} $, where $\gamma_n$ is defined by (\ref{defn_gamma_n}).
\end{enumerate}
\end{thm}

The strategy is to project true zero modes onto the vector space spanned by certain almost eigensections.  The following lemma is the technical tool to do the counting after the projection.  It is implied by the Welch bound (\cite{ref_Welch}).  We include its proof for completeness.

\begin{lem}\label{lem_Welch_dim_00}
For any $\delta_6>0$, there exists a constant $c$ with the following significance.  For any integer $L_o>2\delta_6$, suppose that $\{u_l\}_{l=1}^L$ is a set of unit vectors in $\mathbb{C}^{L_o}$ such that their inner product satisfying
\begin{align*} \big|\langle u_l,u_{l'}\rangle\big| &< \frac{\delta_6}{L_o} &\text{for all}\quad &l\neq l' ~. 
\end{align*}
Then $L\leq L_o + c$.
\end{lem}
\begin{proof}
We may assume $L > L_o$.  Let $U$ be the $L\times L_o$ matrix whose $l$-th column is the vector $u_l$.  Consider the matrix $H = U^* U$.  The matrix $H$ is Hermitian, and its kernel has dimension no less than $L-L_o$.  Suppose that the eigenvalues of $H$ are $\{\lambda_1,\cdots,\lambda_{L_o},0,\cdots,0\}$, then $\lambda_1+\cdots+\lambda_{L_o} = L $.  By the Cauchy--Schwarz inequality,
\begin{align*}
L^2 &= (\lambda_1+\cdots+\lambda_{L_o})^2 \\
&\leq L_o(\lambda_1^2+\cdots+\lambda_{L_o}^2) = L_o \,{\rm trace}(H^*H) \\
&= L_o \big( L + \sum_{l\neq l'} \big|\langle u_l,u_{l'}\rangle\big|^2\big)\\
&\leq L_o ( L + L(L-1)\frac{\delta_6^2}{L_o^2} ) ~.
\end{align*}
It follows that $L\leq L_o + c$.
\end{proof}

The remainder of this section is devoted to the proof of theorem \ref{thm_sf_upper_bound}.

\begin{proof}[Proof of theorem \ref{thm_sf_upper_bound}]
This proof contains ten steps.  Before getting into the details, we briefly outline the strategy.  Write $\psi$ as
$$\psi|_{\text{binding}} + \psi|_{\text{Dehn}} + \psi|_{\Sigma\times S^1} ~. $$
First, regard these sections as being defined on the associated $S^1\times S^2$ or $\Sigma\times S^1$, and project them onto the space spanned by small eigensections of the model manifolds.  Next, multiply the projections by suitable cut-off functions, and regard them as sections on the original $3$-manifold $Y$.  By doing the cut and paste carefully, their inner product is still small after the procedure.  It allows us to invoke lemma \ref{lem_Welch_dim_00} to obtain the upper bound on the spectral flow.

%%%%
\subsection*{Step 1}  In this step, we construct the sequence $\{s_n\}_{ n\in\mathbb{N}}$.  Let $c_1$ be the constant given by corollary \ref{cor_sf_exist_00} for the associated contact forms with $\delta_1=1$.  It follows that for any $r\geq c_1$, if $\check{D}_r$ or $\tilde{D}_r$ has an eigenvalue $\lambda_0$ with $|\lambda_0| \leq r^{-1}$ on $\mathcal{S}_{k,m}$, there is a zero crossing on $\mathcal{S}_{k,m}$ happening somewhere in the interval
$$ [r-\frac{c_1+2}{r}, r+\frac{c_1+2}{r}] ~. $$
Proposition \ref{prop_S2S1_unique_00} says that it is the only zero crossing on $\mathcal{S}_{k,m}$ for $r\geq c_1$.

By lemma \ref{lem_divide_interval_00} with $\delta_3=c_1+2$, there exists a sequence $\{s_{n}\}_{n\in\mathbb{N}}$ and a constant $c_2>0$ such that
\begin{enumerate}
\item the total number of zero crossings of $\check{D}_r$ and $\tilde{D}_r$ happening between $s_{n}-\frac{c_1+2}{s_{n-1}}$ and $s_{n}+\frac{c_1+2}{s_{n-1}}$ is less than $c_2$ for any $n\in\mathbb{N}$;
\item for each $n\in\mathbb{N}$, $ |s_{n} - \gamma_{n} - \frac{1}{V}|\leq \frac{1}{4V} $ where $\gamma_{n} = \frac{2n}{V}$.
\end{enumerate}

%%%%
\subsection*{Step 2}  In this step, we introduce the index sets of the vector space of almost eigensections.  For each $n\in\mathbb{N}$, consider the following condition on $k$ and $m$
\begin{align}\label{eqn_small_eig_condition_00}
\min_{r\in[s_{n-1}, s_{n}]}\min_{\substack{\lambda\in{\rm Spec}\check{D}_r \\ \text{on } \mathcal{S}_{k,m}}} |\lambda| \leq \frac{1}{s_{n}} ~.
\end{align}
In other words, the condition means that for some $r\in[s_{n-1},s_n]$, $\check{D}_r|_{\mathcal{S}_{k,m}}$ has an eigenvalue within $[-\frac{1}{s_n},\frac{1}{s_n}]$.  We define the following index sets:
\begin{enumerate}
\item $\check{E}^o_{n}$ is the set of the $(k,m)\in\mathbb{Z}^2$ satisfying $m\geq\frac{k}{V}$ and condition (\ref{eqn_small_eig_condition_00});\smallskip
\item $\check{E}_{n}$ is the set of the $(k,m)\in\mathbb{Z}^2$ satisfying $m\geq\frac{k}{V}(1-11\epsilon)$ and condition (\ref{eqn_small_eig_condition_00});\smallskip
\item $\tilde{E}^o_{n}$ is the set of the $(k,m)\in\mathbb{Z}^2$ satisfying $$\frac{2(v-20\epsilon)}{V-2(v-20\epsilon)\sigma(20\epsilon)} k \leq m \leq\frac{2(v+20\epsilon)}{V-2(v+20\epsilon)\sigma(-20\epsilon)} k$$
and condition (\ref{eqn_small_eig_condition_00}) for $\tilde{D}_r$;\smallskip
\item $\tilde{E}_{n}$ is the set of the $(k,m)\in\mathbb{Z}^2$ satisfying
$$\frac{2(v-31\epsilon)}{V-2(v-31\epsilon)\sigma(20\epsilon)} k\leq m\leq\frac{2(v+31\epsilon)}{V-2(v+31\epsilon)\sigma(-20\epsilon)} k$$
and condition (\ref{eqn_small_eig_condition_00}) for $\tilde{D}_r$.\smallskip
\end{enumerate}

From the construction in step 1,
\begin{align}\begin{split}\label{eqn_E_total_number_00}
&\check{I}(s_{n-1},s_n) \leq \#\check{E}^o_n \leq \check{I}(s_{n-1},s_n) + 2c_2 ~, \\
&\tilde{I}(s_{n-1},s_n) \leq \#\tilde{E}^o_n \leq \tilde{I}(s_{n-1},s_n) + 2c_2   \\
\end{split}\end{align}
for all $n\geq c'_2$.  See definition \ref{defn_index_part_00} for $\check{I}$ and $\tilde{I}$.  According to lemma \ref{lem_bound_counting_00},
\begin{align}\label{eqn_E_total_number_100}
\check{I}(s_{n-1},s_n)&\leq c_3s_n  &&\text{and}  & \tilde{I}(s_{n-1},s_n)&\leq c_3s_n ~.
\end{align}

With these estimates, theorem \ref{thm_sf_upper_bound} is equivalent to the following \emph{claim}:  there exists a constant $c_4$ such that
\begin{align}\label{claim_upper_bound_00}
{\rm sf}_a(s_n) - {\rm sf}_a(s_{n-1}) \leq I_\Sigma(n) + \#\check{E}^o_n + \#\tilde{E}^o_n + c_4
\end{align}
for any $n\geq c_4$.  The proof of this claim occupies step 3 to step 10.

The differences between the index sets $\check{E}_n\backslash\check{E}^o_n$ and $\tilde{E}_n\backslash\tilde{E}^o_n$ can be completely characterized by lemma \ref{lem_refined_r_neck_00} and lemma \ref{lem_refined_r_neck_01}:  there exists a constant $c_5 > 0$ such that
\begin{align}\begin{split}\label{eqn_E_Eo_description_00}
\check{E}_n\backslash\check{E}^o_n &= \big\{ (k,m) \big| k=n \text{, and } \frac{n}{V}(1-11\epsilon)\leq m<\frac{n}{V} \big\}\quad\text{, and} \\
\tilde{E}_n\backslash\tilde{E}^o_n &= \big\{ (k,m) \big| k+mh(20\epsilon)=n, \frac{2n}{V}(v-31\epsilon)\leq m<\frac{2n}{V}(v-20\epsilon) \big\} \\
&\quad\cup \big\{ (k,m) \big| k+mh(-20\epsilon)=n, \frac{2n}{V}(v+20\epsilon)<m\leq\frac{2n}{V}(v+31\epsilon) \big\}
\end{split}\end{align}
for any $n\geq c_5$.

%%%%
\subsection*{Step 3}  In this step, six cut-off functions are defined.  Let $\check{\chi}^o$ and $\check{\chi}$ be the cut-off functions which are supported only on the tubular neighborhood of the binding, and which depend only on $\rho$ in terms of the coordinate in subsection \ref{subsec_binding_00}, with
\begin{align*}
\check{\chi}^o(\rho) = \begin{cases} 1 &\text{when }\rho\leq1+4\epsilon \\ 0 &\text{when }\rho\geq1+6\epsilon \end{cases}
&~, & \check{\chi}(\rho) = \begin{cases} 1 &\text{when }\rho\leq1+8\epsilon \\ 0 &\text{when }\rho\geq1+10\epsilon \end{cases} ~.
\end{align*}
Note that $\check{\chi}\circ\check{\chi}^o = \check{\chi}^o$.

Let $\tilde{\chi}^o$ and $\tilde{\chi}$ be the cut-off functions which are supported only on the Dehn-twist region, and depend only on $\rho$ in terms of the coordinate in subsection \ref{subsec_Dehn_00}, with
\begin{align*}
\tilde{\chi}^o(\rho) = \begin{cases} 1 &\text{when }|\rho|\leq24\epsilon \\ 0 &\text{when }|\rho|\geq26\epsilon \end{cases}
&~, & \tilde{\chi}(\rho) = \begin{cases} 1 &\text{when }|\rho|\leq28\epsilon \\ 0 &\text{when }|\rho|\geq30\epsilon \end{cases} ~.
\end{align*}
Note that $\tilde{\chi}\circ\tilde{\chi}^o = \tilde{\chi}^o$.

Let $\hat{\chi}^o$ be $1-\check{\chi}^o-\tilde{\chi}^o$.  In terms of the terminology introduced by definition \ref{defn_Sigma_ext_curt}, $\hat{\chi}^o = 1$ on $\Sigma_{-6\epsilon}\times S^1$, and $\hat{\chi}^o = 0$ on $Y\backslash(\Sigma_{-4\epsilon}\times S^1)$.  Let $\hat{\chi}$ be the similar cut-off function depending only on $\rho$ near $\pl\Sigma\times S^1$, with
\begin{align*}
\hat{\chi} = \begin{cases} 1 &\text{on }\Sigma_{-2\epsilon}\times S^1\\ 0 &\text{on }Y\backslash(\Sigma\times S^1) \end{cases} ~.
\end{align*}
Also, $\hat{\chi}\circ\hat{\chi}^o = \hat{\chi}^o$.

%%%%
\subsection*{Step 4}  For any zero mode $\psi$ of $D_r$, we study $\check{\chi}\psi$ and $\tilde{\chi}\psi$ in terms of the eigensections on $\check{Y}$ and $\tilde{Y}$.

With the results in section \ref{subsec_2nd_order_S2S1}, there exists a constant $c_6$ with the following property.  Suppose that $n\geq c_6$ and $(k,m)\in\check{E}_n$.  Then for each $r\in(s_{n-1},s_n]$, $\check{D}_r$ has a unique eigenvalue $\lambda_0$ on $\mathcal{S}_{k,m}$ with $|\lambda_0|<1$.  To be more precise, uniqueness follows from proposition \ref{prop_small_eigensection_00} and \ref{prop_small_eigensection_01}.  Existence follows from lemma \ref{lem_exist_sl_00} and \ref{lem_exist_sl_01}.

With this understood, let ${\check{\rm pr}_r}$ be the $L^2$-orthogonal projection onto the eigenspaces of small eigenvalues arising from $\check{E}_n$.  Similarly, let ${\tilde{\rm pr}_r}$ be the $L^2$-orthogonal projection onto the eigenspaces of small eigenvalues arising from $\tilde{E}_n$.

\begin{lem}\label{lem_check_chi_psi_00}
There exists a constant $c$ such that the following holds.  For any $n\geq c$ and $r\in(s_{n-1},s_n]$, suppose that $\psi$ is a zero mode of $D_r$ with unit $L^2$-norm.  Regard $\check{\chi}\psi$ as being defined on $\check{Y}$, and let $\check{\psi}^{\text{\rm err}} = \check{\chi}\psi - {\check{\rm pr}_r}(\check{\chi}\psi)$.  Then
\begin{align*}
\int_{\check{Y}}|\check{\psi}^{\text{\rm err}}|^2&\leq c r^{-1} &\text{ and}&
&\int_{\check{Y}}\langle \check{\psi}^{\text{\rm err}}, \check{\eta} \rangle &=0
\end{align*}
for any $\check{\eta}$ in the image of $\check{\rm pr}_r$.  The assertion also holds for $\check{\chi}^o\psi$, and the untwisting (\ref{def_untwisting}) of $\tilde{\chi}\psi$ and $\tilde{\chi}^o\psi$ on $\tilde{Y}$ with the projection $\tilde{\rm pr}_r$.
\end{lem}
\begin{proof}[Proof of lemma \ref{lem_check_chi_psi_00}]
The orthogonality between $\check{\psi}^{\text{\rm err}}$ and $\check{\eta}$ follows directly from the construction.

With the spectral decomposition given by $\check{D}_r$, let
\begin{itemize}
\item $\check{\psi}^{\text{\rm err}}_{\text{\rm L}}$ be the $L^2$-orthogonal projection of $\check{\psi}^{\text{\rm err}}$ onto the subspace spanned by the eigensections whose eigenvalue $\lambda$ satisfies $|\lambda|\geq\sqrt{\frac{r}{2}}$;\smallskip
\item $\check{\psi}^{\text{\rm err}}_{\text{\rm M}}$ be the $L^2$-orthogonal projection of $\check{\psi}^{\text{\rm err}}$ onto the subspace spanned by the eigensections whose eigenvalue $\lambda$ satisfies $\frac{1}{10V}\leq|\lambda|<\sqrt{\frac{r}{2}}$, and does not arise from $\check{E}_n$;\smallskip
\item $\check{\psi}^{\text{\rm err}}_{\text{\rm S}}$ be the $L^2$-orthogonal projection of $\check{\psi}^{\text{\rm err}}$ onto the subspace spanned by the eigensections whose eigenvalue $\lambda$ satisfies $\frac{1}{s_{n}}\leq|\lambda|<\frac{1}{10V}$, and does not arise from $\check{E}_n$.\smallskip
\end{itemize}
It follows that $\check{\chi}\psi = {\check{\rm pr}_r}(\check{\chi}\psi) + \check{\psi}^{\text{\rm err}}_{\text{\rm L}} + \check{\psi}^{\text{\rm err}}_{\text{\rm M}} + \check{\psi}^{\text{\rm err}}_{\text{\rm S}}$.  It is an orthogonal decomposition with respect to the $L^2$-inner product, and $\check{D}_r$ preserves the orthogonality.  We are going to estimate the size of the three error-components.  Note that
\begin{align}\label{eqn_check_chi_psi_01}
\check{D}_r(\check{\chi}\psi) &= \check{\chi}'{{\rm cl}}(\dd\rho)\psi ~.
\end{align}
The function $\check{\chi}'$ is supported only on the region where $1+8\epsilon\leq\rho\leq1+10\epsilon$, and the Clifford action of $\dd\rho$ switches the two components of $\psi$, see (\ref{eqn_Dirac_neck_01}).

\subsubsection*{The component with large eigenvalue}  The estimate on $\check{\psi}^{\text{\rm err}}_{\text{\rm L}}$ is easy to come by.  By (\ref{eqn_check_chi_psi_01}), there exists a constant $c_8$ such that
\begin{align*}
\int_{\check{Y}} |\check{\psi}^{\text{\rm err}}_{\text{\rm L}}|^2 \leq 2r^{-1} \int_{\check{Y}} |\check{D}_r(\check{\psi}^{\text{\rm err}}_{\text{\rm L}})|^2 \leq 2r^{-1} \int_{\check{Y}} |\check{D}_r(\check{\chi}\psi)|^2 \leq c_8r^{-1} ~.
\end{align*}

\subsubsection*{The component with medium eigenvalue}  By lemma \ref{lem_S2S1_hot_00}, all the eigenvalues $\lambda$ with $|\lambda|<\sqrt{\frac{r}{2}}$ correspond to different $\mathcal{S}_{k,m}$'s.  Let $\check{\psi}^{\text{\rm eig}}_{k,m}$ be the corresponding eigensection with $\int_{\check{Y}} |\check{\psi}^{\text{\rm eig}}_{k,m}|^2 = 1$.  By (\ref{eqn_check_chi_psi_01}), there exists a constant $c_9$ so that
\begin{align*}
&\big| \int_{\check{Y}}\langle\check{\chi}\psi,\check{\psi}^{\text{\rm eig}}_{k,m}\rangle \big|^2 \\
=&\; \frac{1}{\lambda^2}\big| \int_{\check{Y}}\langle(\check{\chi})'{\rm cl}(\dd\rho)\psi,\check{\psi}^{\text{\rm eig}}_{k,m}\rangle \big|^2 \\
\leq&\; 200V^2 \big( \int_{\check{Y}} |(\check{\chi})'\beta_{k,m}|^2\int_{\check{Y}}|\check{\alpha}^{\text{\rm eig}}_{k,m}|^2 + \int_{\check{Y}} |(\check{\chi})'\alpha_{k,m}|^2\int_{\check{Y}}|\check{\beta}^{\text{\rm eig}}_{k,m}|^2 \big) \\
\leq&\; c_9 \big( \int_{\check{Y}} |(\check{\chi})'\beta_{k,m}|^2 + r^{-1} \int_{\check{Y}} |(\check{\chi})'\alpha_{k,m}|^2 \big)
\end{align*}
where $\big( (\check{\chi})'\alpha_{k,m},(\check{\chi})'\beta_{k,m}\big)$ is the $\mathcal{S}_{k,m}$-component of $(\check{\chi})'\psi$.  The first inequality follows from the fact that ${\rm cl}(\dd\rho)$ switches the components.  The second inequality follows from lemma \ref{lem_beta_estimate_00} on $\check{\psi}^{\text{\rm eig}}_{k,m}$.

After summing up the above inequality over all $(k,m)$ involving in $\check{\psi}^{\text{\rm err}}_{\text{\rm M}}$, we have
\begin{align*}
\int_{\check{Y}} |\check{\psi}^{\text{\rm err}}_{\text{\rm M}}|^2 \leq c_9 \big( \int_{\check{Y}} |(\check{\chi})'\beta|^2 + r^{-1} \int_{\check{Y}} |(\check{\chi})'\alpha|^2 \big) ~.
\end{align*}
With proposition \ref{prop_beta_estimate_00}, we obtain the estimate on $\check{\psi}^{\text{\rm err}}_{\text{\rm M}}$.

\subsubsection*{The component with small eigenvalue}  By the same token, all the eigenvalues $\lambda$ with $|\lambda|<\frac{1}{10V}$ correspond to different $\mathcal{S}_{k,m}$.  Proposition \ref{prop_small_eigensection_00} and proposition \ref{prop_small_eigensection_01} give the approximation of the corresponding eigensections
\begin{align*}
\check{\psi}^{\text{\rm eig}}_{k,m} &= \check{\mathfrak{q}}_{k,m}\check{\psi}_{k,m} + \check{\psi}_{k,m}^{(3)} ~.
\end{align*}
If $m>\frac{k}{V}$ or $m<\frac{k}{V}(1-11\epsilon)$, the support of $\check{\psi}_{k,m}$ and $(\check{\chi})'$ are disjoint.  By (\ref{eqn_check_chi_psi_01}), there exists a constant $c_8$ such that
\begin{align*}
\big| \int_{\check{Y}} \langle\check{\chi}\psi, \check{\psi}^{\text{eig}}_{k,m}\rangle \big|^2 &= \frac{1}{\lambda^2} \big| \int_{\check{Y}}\langle(\check{\chi})'{{\rm cl}}(\dd\rho)\psi, \check{\psi}^{(3)}_{k,m} \rangle \big|^2\\
&\leq c_{10} s_{n}^2 r^{-3} \int_{\check{Y}} |(\check{\chi})'\psi_{k,m}|^2 < 10c_{10} r^{-1} \int_{\check{Y}} |(\check{\chi})'\psi_{k,m}|^2
\end{align*}
for all $r\geq c_{10}$.

If $\frac{k}{V}(1-11\epsilon)\leq m\leq\frac{k}{V}$, lemma \ref{lem_refined_r_neck_00} with the condition that $|\lambda|<\frac{1}{10V}$ implies that $k = n$.  However, lemma \ref{lem_refined_r_neck_00} also implies that such $(k,m)$ belongs to $\check{E}$.  Therefore, these $(k,m)$ do not involve in $\check{\psi}^{\text{\rm err}}_{\text{\rm S}}$.\\

By summing up the above inequality over the $(k,m)$ involving in $\check{\psi}^{\text{\rm err}}_{\text{\rm S}}$, we obtain the estimate on $\check{\psi}^{\text{\rm err}}_{\text{\rm S}}$.

It is clear that the assertion also holds for $\check{\chi}^o\psi$.
The discussion for $\tilde{\chi}\psi$ and $\tilde{\chi}^o\psi$ are parallel to the above argument, and is omitted.
\end{proof}

From now on, we will implicitly apply the untwisting (\ref{def_untwisting}) when working with $\tilde{\chi}\psi$ on $\tilde{Y}$.  On the other hand, when we regard some section on $\tilde{Y}$ as being defined on $Y$, we will implicitly undo the untwisting.

%%%%
\subsection*{Step 5}  For any zero mode $\psi$ of $D_r$, consider ${\hat{\rm pr}_n}(\hat{\chi}\psi)$ where $\hat{\rm pr}_n$ is the composition of the following three projection
\begin{align*}
\hat{\chi}\psi \longmapsto \hat{\chi}\alpha \longmapsto \hat{\chi}\alpha_ne^{in\phi}(2\pi V)^{-\oh} \longmapsto {\hat{\rm pr}_n}(\hat{\chi}\alpha_n) e^{in\phi}(2\pi V)^{-\oh} ~.
\end{align*}
The first map is the projection onto the first component of $\psi$.  The second map is the projection onto the frequency $n$ component with respect to the $S^1$-action in $e^{i\phi}$, and $(2\pi V)^{-\oh}$ is simply a normalizing constant.  The last map is the $L^2$-orthogonal projection onto the kernel of $\bar{\pl}_n$ as discussed in section \ref{sec_dirac_page_00}.  We use the same notation for the last projection and the composition of the three projections.
\begin{lem}\label{lem_hat_chi_psi_00}
There exists a constant $c$ such that the following holds.  For any $n\geq c$ and $r\in(s_{n-1},s_n]$, suppose that $\psi$ is a zero mode of $D_r$ with unit $L^2$-norm.  Regard $\hat{\chi}\psi$ as defined on $\Sigma\times S^1$, and let $\hat{\psi}^{\text{\rm err}} = \hat{\chi}\psi - {\hat{\rm pr}_n}(\hat{\chi}\psi)$.  Then
\begin{align*}
\int_{\Sigma\times S^1}|\hat{\psi}^{\text{\rm err}}|^2&\leq c r^{-1} ~, &\text{and}&
&\int_{\Sigma\times S^1}\langle \hat{\psi}^{\text{\rm err}}, \hat{\eta} \rangle &=0
\end{align*}
for any $\hat{\eta}$ in the image of $\hat{\rm pr}_n$.  The assertion also holds for $\hat{\chi}^o\psi$.
\end{lem}
\begin{proof}[Proof of lemma \ref{lem_hat_chi_psi_00}]
The orthogonality between $\hat{\psi}^{\text{\rm err}}$ and $\hat{\eta}$ follows directly from the construction.

To estimate the size of $\hat{\psi}^{\text{\rm err}}$, consider the Fourier expansion of $\hat{\chi}\psi$:
\begin{align*} \hat{\chi}\alpha &= \sum_{\mathfrak{n}\in\mathbb{Z}}\hat{\chi}\alpha_{\mathfrak{n}} e^{i\mathfrak{n}\phi}(2\pi V)^{-\oh} &\text{and}& &\hat{\chi}\beta &= \sum_{\mathfrak{n}\in\mathbb{Z}}\hat{\chi}\beta_{\mathfrak{n}} e^{i(\mathfrak{n}+1)\phi}(2\pi V)^{-\oh}
\end{align*}
where $\alpha_{\mathfrak{n}}$ are functions on $\Sigma$, and $\beta_{\mathfrak{n}}$ are sections of $K^{-1}_{\Sigma}$ over $\Sigma$.  It follows that
\begin{align*} \int_{\Sigma\times S^1} |\hat{\chi}\alpha|^2 &= \sum_{\mathfrak{n}\in\mathbb{Z}}\int_{\Sigma} |\hat{\chi}\alpha_{\mathfrak{n}}|^2 &\text{and}& &\int_{\Sigma\times S^1} |\hat{\chi}\beta|^2 &= \sum_{\mathfrak{n}\in\mathbb{Z}}\int_{\Sigma} |\hat{\chi}\beta_{\mathfrak{n}}|^2 ~. \end{align*}
Let $\mathcal{D}_{r,\mathfrak{n}}$ be the operator (\ref{eqn_Dirac_page_00}) with $n$ replaced by $\mathfrak{n}$.  Since $D_r(\hat{\chi}\psi) = {\rm cl}(\dd\hat{\chi})\psi$, $\mathcal{D}_{r,\mathfrak{n}}(\hat{\chi}\alpha_{\mathfrak{n}},\hat{\chi}\beta_{\mathfrak{n}})$ is supported only on $\Sigma\backslash\Sigma_{-6\epsilon}$, and
\begin{align}\label{eqn_D_rn_hat_chi_00}
\mathcal{D}_{r,\mathfrak{n}}(\hat{\chi}\alpha_{\mathfrak{n}},\hat{\chi}\beta_{\mathfrak{n}}) &= \big( -(\hat{\chi})'{\beta}_{\mathfrak{n}}, (\hat{\chi})'{\alpha}_{\mathfrak{n}}  \big) .
\end{align}
The main task is to estimate $\hat{\chi}\alpha_{\mathfrak{n}}$.  The argument is separated into three cases according to the value of $\mathfrak{n}$.  Remember that $\gamma_{\mathfrak{n}} = \frac{2\mathfrak{n}}{V}$.

\subsubsection*{Case 1: $|\gamma_{\mathfrak{n}}-r|\geq r^{\oh}$}  With lemma \ref{lem_Drn_rn_large} and (\ref{eqn_D_rn_hat_chi_00}), there exists a constant $c_9$ such that
\begin{align*} \int_\Sigma |\hat{\chi}\alpha_{\mathfrak{n}}|^2 + |\hat{\chi}\beta_{\mathfrak{n}}|^2 \leq c_{11} r^{-1}\big( \int_{\Sigma\backslash\Sigma_{-6\epsilon}} |\alpha_{\mathfrak{n}}|^2+|\beta_{\mathfrak{n}}|^2\big) \end{align*}
for all $r\geq c_{11}$.

\subsubsection*{Case 2: the $(\ker\bar{\pl}_{\mathfrak{n}})^{\perp}$-component of $\hat{\chi}\alpha_{\mathfrak{n}}$ when $|\gamma_{\mathfrak{n}}-r|< r^{\oh}$}  With the first inequality of proposition \ref{prop_Drn_rn_small} and (\ref{eqn_D_rn_hat_chi_00}), there exists a constant $c_{12}$ such that
\begin{align*} \int_\Sigma |\hat{\chi}\alpha_{\mathfrak{n}} - {\hat{\rm pr}_{\mathfrak{n}}}(\hat{\chi}\alpha_{\mathfrak{n}})|^2 + |\hat{\chi}\beta_{\mathfrak{n}}|^2 \leq c_{12} r^{-1} \big( \int_{\Sigma\backslash\Sigma_{-6\epsilon}} |\alpha_{\mathfrak{n}}|^2+|\beta_{\mathfrak{n}}|^2 \big) \end{align*}
for all $r\geq c_{12}$.

\subsubsection*{Case 3: the $\ker\bar{\pl}_{\mathfrak{n}}$-component of $\hat{\chi}\alpha_{\mathfrak{n}}$ when $\mathfrak{n}\neq n$ and $|\gamma_{\mathfrak{n}}-r|< r^{\oh}$}  For any $\mathfrak{n}\neq n$, $|r-\gamma_{\mathfrak{n}}|$ is no less than $\frac{1}{10V}$.  By the second inequality of proposition \ref{prop_Drn_rn_small} and (\ref{eqn_D_rn_hat_chi_00}), there exists a constant $c_{13}$ such that
\begin{align*} \int_\Sigma |{\hat{\rm pr}_{\mathfrak{n}}}(\hat{\chi}\alpha_{\mathfrak{n}})|^2\leq c_{13} \int_{\Sigma\backslash\Sigma_{-6\epsilon}} |\beta_{\mathfrak{n}}|^2 \end{align*}
for all $r\geq c_{13}$ and $\mathfrak{n}\neq n$.\\

After summing up the estimates of these three cases, we have
\begin{align*} \int_{\Sigma\times S^1} |\hat{\psi}^{\text{\rm err}}|^2 \leq \int_{(\Sigma\backslash\Sigma_{-6\epsilon})\times S^1} \big( (c_{11}+c_{12})r^{-1}|\psi|^2 + c_{13} |\beta|^2 \big) . \end{align*}
With proposition \ref{prop_beta_estimate_00}, it completes the proof of lemma \ref{lem_hat_chi_psi_00}.  The proof for $\hat{\chi}^o\psi$ is the same.
\end{proof}

%%%%
\subsection*{Step 6}  In this step, we project the zero modes onto a vector space spanned by certain almost eigensections.  The projection will depend on $r$.  In other words, the zero modes are \emph{not} projected onto the \emph{same} vector space.

\begin{prop}\label{prop_proj_finite_00}
There exists a constant $c$ which has the following significance.  For any $n\geq c$ and $r\in(s_{n-1},s_n]$, suppose that $\psi$ is a zero mode of $D_r$ with unit $L^2$-norm.  Let
\begin{align*}
{\Pi_r}(\psi) = \hat{\chi}{\hat{\rm pr}_n}(\hat{\chi}^o\psi) + \check{\chi}{\check{\rm pr}_r}(\check{\chi}^o\psi) + \tilde{\chi}{\tilde{\rm pr}_r}(\tilde{\chi}^o\psi) ~,
\end{align*}
then
\begin{enumerate}
\item $\int_Y |{\Pi_r}(\psi)|^2\geq 1-c s_n^{-1}$;\smallskip
\item if there are two such zero modes $\psi_1$ and $\psi_2$ at $r_1$ and $r_2$ with $s_{n-1}<r_1<r_2\leq s_n$, then
\begin{align*}
\big|\int_Y\langle {\Pi_{r_1}}(\psi_1),{\Pi_{r_2}}(\psi_2) \rangle\big| \leq c s_n^{-1} ~;
\end{align*}
\item if there are two zero modes $\psi_1$ and $\psi_2$ at the same $r\in(s_{n-1},s_n]$, and $\int_Y\langle\psi_1,\psi_2\rangle=0$, then
\begin{align*}
\big|\int_Y\langle {\Pi_{r_1}}(\psi_1),{\Pi_{r_2}}(\psi_2) \rangle\big| \leq c s_n^{-1} ~.
\end{align*}
\end{enumerate}
\end{prop}
\begin{proof}[Proof of proposition \ref{prop_proj_finite_00}]  We start with assertion (ii).  Suppose that there are two such zero modes $\{\psi_j\}_{j=1,2}$.  We apply lemma \ref{lem_hat_chi_psi_00} on $\hat{\chi}^o\psi_j$, and lemma \ref{lem_check_chi_psi_00} on $\check{\chi}^o\psi_j$ and $\tilde{\chi}^o\psi_j$.  After multiplying them by the cut-off functions with larger support, we have
\begin{align*}
\hat{\chi}^o\psi_j = \hat{\chi}\hat{\chi}^o\psi_j &= \hat{\chi}{\hat{\rm pr}_n}(\hat{\chi}^o\psi_j) + \hat{\chi}\hat{\psi}^{\text{\rm err}_o}_j ~, \\
\check{\chi}^o\psi_j = \check{\chi}\check{\chi}^o\psi_j &= \check{\chi}{\check{\rm pr}_r}(\check{\chi}^o\psi_j) + \check{\chi}\check{\psi}^{\text{\rm err}_o}_j ~, \\
\tilde{\chi}^o\psi_j = \tilde{\chi}\tilde{\chi}^o\psi_j &= \tilde{\chi}{\tilde{\rm pr}_r}(\tilde{\chi}^o\psi_j) + \tilde{\chi}\tilde{\psi}^{\text{\rm err}_o}_j ~.
\end{align*}
All terms can be regarded as being defined on $Y$.

We also apply lemma \ref{lem_hat_chi_psi_00} on $\hat{\chi}\psi_j$, and lemma \ref{lem_check_chi_psi_00} on $\check{\chi}\psi_j$ and $\tilde{\chi}\psi_j$.  We have $\hat{\chi}\psi_j = {\hat{\rm pr}_n}(\hat{\chi}\psi_j) + \hat{\psi}^{\text{\rm err}}_j$, $\check{\chi}\psi_j = {\check{\rm pr}_r}(\check{\chi}\psi_j) + \check{\psi}^{\text{\rm err}}_j$ and $\tilde{\chi}\psi_j = {\tilde{\rm pr}_r}(\tilde{\chi}\psi_j) + \tilde{\psi}^{\text{\rm err}}_j$, without the cut-off function on the error-term.  Not all these terms can be regarded as being defined on $Y$.  

The inner product between ${\Pi_{r_1}}(\psi_1)$ and ${\Pi_{r_2}}(\psi_2)$ is equal to
\begin{align*}
\int_{Y} \langle \psi_1-\hat{\chi}\hat{\psi}_1^{\text{\rm err}_o}-\check{\chi}\check{\psi}_1^{\text{\rm err}_o}-\hat{\chi}\hat{\psi}_1^{\text{\rm err}_o}, \psi_2-\hat{\chi}\hat{\psi}_2^{\text{\rm err}_o}-\check{\chi}\check{\psi}_2^{\text{\rm err}_o}-\hat{\chi}\hat{\psi}_2^{\text{\rm err}_o} \rangle ~.
\end{align*}
We would like to show that the magnitude of all these sixteen pairings is less than $c_{14}s_n^{-1}$ for some constant $c_{14}$.  There are four types of these pairings.

\subsubsection*{Type 1: $\int_Y\langle\psi_1,\psi_2\rangle$}  The estimate on this term is given by proposition \ref{prop_small_ip_00}.

\subsubsection*{Type 2: the pairings between $\psi_j$ and the error term on $\Sigma\times S^1$}  By lemma \ref{lem_hat_chi_psi_00},
\begin{align*}
\int_Y\langle\psi_1,\hat{\chi}\hat{\psi}_2^{\text{\rm err}_o}\rangle &= \int_{\Sigma\times S^1}\langle\hat{\chi}\psi_1,\hat{\psi}_2^{\text{\rm err}_o}\rangle = \int_{\Sigma\times S^1}\langle{\hat{\rm pr}_n}(\hat{\chi}\psi_1) + \hat{\psi}^{\text{\rm err}}_1,\hat{\psi}_2^{\text{\rm err}_o}\rangle \\
&= \int_{\Sigma\times S^1} \langle\hat{\psi}^{\text{err}}_1,\hat{\psi}_2^{\text{\rm err}_o}\rangle ~.
\end{align*}
We conclude that $\big| \int_Y\langle\psi_1,\hat{\chi}\hat{\psi}_2^{\text{\rm err}_o}\rangle \big| \leq c_{15}(r_1 r_2)^{-\oh} \leq 2 c_{15} s_n^{-1}$ for some constant $c_{15}$.

\subsubsection*{Type 3: the pairings between $\psi_j$ and the error term on $\check{Y}$ or $\tilde{Y}$}  Similarly,
\begin{align*}
\int_Y\langle\psi_1,\check{\chi}\check{\psi}_2^{\text{\rm err}_o}\rangle &= \int_{\check{Y}}\langle{\check{\rm pr}_{r_1}}(\check{\chi}\psi_1),\check{\psi}_2^{\text{\rm err}_o}\rangle + \langle\check{\psi}^{\text{\rm err}}_1,\check{\psi}_2^{\text{\rm err}_o}\rangle ~.
\end{align*}
Lemma \ref{lem_check_chi_psi_00} implies that the second term $\big| \int_{\check{Y}}\langle\check{\psi}^{\text{\rm err}}_1,\check{\psi}_2^{\text{\rm err}_o}\rangle \big|  \leq 2 c_{16} s_n^{-1}$ for some constant $c_{16}$.  Unlike type 2, the first pairing might not be zero.  However, corollary \ref{cor_eigensection_differ_r_00} and corollary \ref{cor_eigensection_differ_r_01} imply that there exists a constant $c_{17}$ such that
\begin{align*} \big| \int_{\check{Y}} \langle {\check{\rm pr}_{r_1}}(\check{\chi}\psi_1),\check{\psi}_2^{\text{\rm err}_o} \rangle \big| \leq c_{17} s_n^{-\frac{3}{2}} (\#\check{E}_n) \big(\int_{\check{Y}}|\check{\psi}_2^{\text{\rm err}_o}|^2\big)^\oh ~. \end{align*}
By (\ref{eqn_E_total_number_00}) and (\ref{eqn_E_total_number_100}), $\check{E}_n$ is less than some multiple of $s_n$, and $\check{\psi}_2^{\text{\rm err}_o}$ is estimated by lemma \ref{lem_check_chi_psi_00}.  Hence, there exists a constant $c_{18}$ such that
$$ \big| \int_{\check{Y}}\langle{\check{\rm pr}_{r_1}}(\check{\chi}\psi_1),\check{\psi}_2^{\text{\rm err}_o}\rangle \big|\leq c_{18}s_n^{-1} ~.  $$

\subsubsection*{Type 4: the pairings between two error terms}  Lemma \ref{lem_check_chi_psi_00}, lemma \ref{lem_hat_chi_psi_00} and the Cauchy--Schwarz inequality implies that these pairings are of $\mathcal{O}(s_n^{-1})$.\\

The proof for assertion (i) and (iii) is the same up to a minor modification for type 1.  This completes the proof of proposition \ref{prop_proj_finite_00}.
\end{proof}

%%%%
\subsection*{Step 7}  This step is a digression on the $r$-dependence of the eigensection approximation constructed in section \ref{subsec_2nd_order_S2S1}.  We start with $\check{Y}$.  Suppose that $\check{D}_r$ on $\mathcal{S}_{k,m}$ has an eigenvalue $\lambda_0$ with $|\lambda_0|\leq 1$, then proposition \ref{prop_small_eigensection_00} and proposition \ref{prop_small_eigensection_01} apply.

\subsubsection*{Case 1}  When $|m|<(\frac{1}{32\epsilon^2}-1)k$, the main term $\check{\psi}_{k,m}$ in proposition \ref{prop_small_eigensection_00} consists of the zeroth, first and second order terms, see (\ref{eqn_hot_approx_01}).
Let $\check{\psi}^{(0)}_{k,m}$ be the section which consists of only the zeroth and first order term of $\check{\psi}_{k,m}$.  Namely, throw away the $\mathfrak{a}_2(x,r)$ and $\mathfrak{b}_2(x)$ terms in (\ref{eqn_hot_approx_01}).  The key feature of these sections $\check{\psi}_{k,m}^{(0)}$ is that they are independent of $r$.  Let $\check{\psi}_{k,m}^{(2)}$ be $ \check{\psi}_{k,m} - \check{\mathfrak{q}}_{k,m}\check{\psi}_{k,m}^{(0)}$.  Namely, $\check{\psi}_{k,m}^{(2)}$ is the sum of the second order term in $\check{\mathfrak{q}}_{k,m}\check{\psi}_{k,m}$ and $\check{\psi}_{k,m}^{(3)}$.  Under the assumption of proposition \ref{prop_small_eigensection_00}, it is easy to see that there exists a constant $c_{19}$ such that
\begin{align*} \int_{\check{Y}} |\check{\psi}_{k,m}^{(2)}|^2 &\leq c_{19}r^{-2} ~. \end{align*}
The coefficients $\check{\mathfrak{q}}_{k,m}$ in proposition \ref{prop_small_eigensection_00} also depends on $r$, but they are only scalars.  Note that $\check{\psi}_{k,m}^{(0)}$ can be regarded as a section on $Y$.

According to the discussion in section \ref{subsec_higher_order_S2S1}, when $\frac{k}{V}(1-11\epsilon)\leq m\leq\frac{k}{V}$, the first and second order term are zero, and $\check{\psi}_{k,m}^o = \check{\psi}_{k,m}$ and $\int_{\check{Y}}|\check{\psi}_{k,m}^{(2)}|^2 \leq c_{20}\exp(-\frac{r}{c_{20}})$.  Moreover, for $k=n$, the first component of $\check{\psi}_{k,m}^{(0)}$ is the same as $\check{\zeta}_{n,m}e^{in\phi}(2\pi V)^{-\oh}$ given by (\ref{eqn_almost_nm_page_00}), and second component of $\check{\psi}_{k,m}^{(0)}$ is zero.

\subsubsection*{Case 2}  When $|m|\geq(\frac{1}{32\epsilon^2}-1)k$, the main term $\check{\psi}_{k,m}$ in proposition \ref{prop_small_eigensection_01} is already independent of $r$.  To unify the notation, let $\check{\psi}_{k,m}^{(0)} = \check{\psi}_{k,m}$ and $\check{\psi}_{k,m}^{(2)} = \check{\psi}_{k,m}^{(3)}$.

\subsubsection*{On $\tilde{Y}$}  The discussion on the $r$-dependence is the same as that for $\check{Y}$.  We just state the result on the special regions discussed in section \ref{subsec_higher_order_S2S1}.  When
\begin{align*}\begin{cases}
k+mh(20\epsilon)=n \;\text{ and }\; \frac{2n}{V}(v-31\epsilon)\leq m\leq\frac{2n}{V}(v-20\epsilon) \text{~, or}&\\
k+mh(-20\epsilon)=n \;\text{ and }\; \frac{2n}{V}(v+20\epsilon)\leq m\leq\frac{2n}{V}(v+31\epsilon) \text{~,}&
\end{cases}\end{align*}
the first component of $\tilde{\psi}_{k,m}^{(0)}$ is the same as the untwisting  of $\tilde{\zeta}_{n,m}e^{in\phi}(2\pi V)^{-\oh}$ given by (\ref{eqn_almost_nm_page_01}), and second component of $\tilde{\psi}_{k,m}^{(0)}$ is zero.  The remainder term satisfies $\int_{\tilde{Y}} |\tilde{\psi}_{k,m}^{(2)}|^2 \leq c_{20}\exp(-\frac{r}{c_{20}})$.

%%%%
\subsection*{Step 8}  In this step, we throw away the $r$-dependent part of $\check{\chi}{{\rm pr}_r}(\check{\chi}^o\psi)$ and $\tilde{\chi}{{\rm pr}_r}(\tilde{\chi}^o\psi)$.  In other words, the projection is modified to be $r$-independent.

\begin{lem}\label{lem_check_r_dep_00}
There exists a constant $c$ with the following property.  For any $n\geq c$ and $r\in(s_{n-1},s_n]$, suppose that $\psi$ is a zero mode of $D_r$ with unit $L^2$-norm.  Then there exists a section $\check{\psi}^{\text{\rm rem}}$ such that
\begin{enumerate}
\item the support of $\check{\psi}^{\text{\rm rem}}$ is contained in the support of $\check{\chi}$.  Thus, $\check{\psi}^{\text{\rm rem}}$ can be regarded as a section either on $Y$ or $\check{Y}$;\smallskip
\item $\int_{Y} |\check{\psi}^{\text{\rm rem}}|^2 \leq c s_n^{-2}$;\smallskip
\item $\check{\chi}{{\rm pr}_r}(\check{\chi}^o\psi) - \check{\psi}^{\text{\rm rem}}$ belongs to the vector space spanned by
\begin{align*} \big\{\check{\psi}_{k,m}^{(0)} \big| (k,m)\in \check{E}_n^o \big\} \oplus \big\{\check{\psi}_{k,m}^{(0)}\big|k=n\text{, and }\frac{n}{V}(1-7\epsilon)\leq m<\frac{n}{V} \big\} ~. \end{align*}
\end{enumerate}
\end{lem}
\begin{proof}[Proof of lemma \ref{lem_check_r_dep_00}]
For each $(k,m)\in\check{E}_n$, $\check{D}_r$ has a unique eigenvalue $\lambda_0$ on $\mathcal{S}_{k,m}$ with $|\lambda_0|\leq1$.  Let $\check{\psi}_{k,m}^{\text{\rm eig}}$ be the corresponding eigensection given by proposition \ref{prop_small_eigensection_00} and proposition \ref{prop_small_eigensection_01}.  If we apply the projection operator $\check{\rm pr}_r$ introduced in step 4, we have
\begin{align*} {\check{\rm pr}_r}(\check{\chi}^o\psi) &= \sum_{(k,m)\in\check{E}_n} \check{\mathfrak{c}}_{k,m}\check{\psi}_{k,m}^{\text{\rm eig}} \end{align*}
where $\check{\mathfrak{c}}_{k,m} = \int_{\check{Y}}\langle\check{\chi}^o\psi, \check{\psi}_{k,m}^{\text{\rm eig}}\rangle$.

Remember that $\check{E}_n\backslash\check{E}_n^o$ is characterized by (\ref{eqn_E_Eo_description_00}).  Let $\check{\psi}^{\text{\rm rem}}$ be the sum of the following terms:
\begin{itemize}
\item $\check{\chi}\check{\mathfrak{c}}_{k,m}\check{\psi}_{k,m}^{(2)}$ for $(k,m)\in\check{E}_n^o$;\smallskip
\item $\check{\chi}\check{\mathfrak{c}}_{k,m}\check{\psi}_{k,m}^{(2)}$ for $(k,m)\in\check{E}_n\backslash\check{E}_n^o$ and $m\geq\frac{n}{V}(1-7\epsilon)$;\smallskip
\item $\check{\chi}\check{\mathfrak{c}}_{k,m}\check{\psi}_{k,m}^{\text{\rm eig}}$ for $(k,m)\in\check{E}_n\backslash\check{E}_n^o$ and $m<\frac{n}{V}(1-7\epsilon)$.
\end{itemize}
We now estimate their $L^2$-norm.
\begin{itemize}
\item For the terms of the first and second kind, step 7 gives a constant $c_{21}$ such that
\begin{align*} \int_{\check{Y}}|\check{\chi}\check{\mathfrak{c}}_{k,m}\check{\psi}_{k,m}^{(2)}|^2 &\leq c_{21}r^{-2}\int_{\check{Y}} |(\check{\chi}^o\psi)_{k,m}|^2 \end{align*}
where $(\check{\chi}^o\psi)_{k,m}$ is the $\mathcal{S}_{k,m}$-component of $\check{\chi}^o\psi$.
\item For the terms of the third kind,
\begin{align*}
\int_{\check{Y}} |\check{\chi}\check{\mathfrak{c}}_{k,m}\check{\psi}_{k,m}^{\text{\rm eig}}|^2 &\leq |\check{\mathfrak{c}}_{k,m}|^2 = \big|\int_{\check{Y}}\langle\check{\chi}^o\psi, \check{\psi}_{k,m}^{\text{\rm eig}}\rangle\big|^2 \\
&= \big| \int_{\check{Y}}\langle(\check{\chi}^o\psi)_{k,m},\check{\mathfrak{q}}_{k,m}\check{\psi}_{k,m}^{(0)}+\check{\psi}_{k,m}^{(2)} \rangle \big|^2 \\
&\leq c_{20}\exp(-\frac{r}{c_{20}})\int_{\check{Y}}|(\check{\chi}^o\psi)_{k,m}|^2 ~.
\end{align*}
For the last inequality, note that the support of $\check{\chi}^o$ and $\check{\psi}_{k,m}$ are disjoint, and the $L^2$-norm of $\check{\psi}_{k,m}^{(2)}$ is exponentially small by step 7.
\end{itemize}
After summing up the above inequalities, we obtain the assertion (ii) of the lemma.  Assertion (i) follows from the construction of $\check{\psi}^{\text{\rm rem}}$.

For the last assertion, it is clear that $\check{\chi}{{\rm pr}_r}(\check{\chi}^o\psi) - \check{\psi}^{\text{\rm rem}}$ belongs to the vector space spanned by those elements in assertion (iii), multiplied by $\check{\chi}$.  However, the support of $\check{\chi}{{\rm pr}_r}(\check{\chi}^o\psi) - \check{\psi}^{\text{\rm rem}}$ is contained in the support of $\check{\chi}$.  Since $\check{\chi}$ is equal to $1$ on the support of those elements in assertion (iii), it completes the proof of lemma \ref{lem_check_r_dep_00}.
\end{proof}

The argument for the Dehn-twist region is the same.  We just state the result.
\begin{lem}\label{lem_tilde_r_dep_00}
There exists a constant $c$ with the following significance.  For any $n\geq c$ and $r\in(s_{n-1},s_n]$, suppose that $\psi$ is a zero mode of $D_r$ with unit $L^2$-norm.  Then there exists a section $\tilde{\psi}^{\text{\rm rem}}$ such that
\begin{enumerate}
\item the support of $\tilde{\psi}^{\text{\rm rem}}$ is contained in the support of $\tilde{\chi}$;  hence, up to the untwisting (\ref{def_untwisting}), $\tilde{\psi}^{\text{\rm rem}}$ can be regarded as defined either on $Y$ or $\tilde{Y}$;\smallskip
\item $\int_{\tilde{Y}} |\tilde{\psi}^{\text{\rm rem}}|^2 \leq c s_n^{-2}$;\smallskip
\item $\tilde{\chi}{{\rm pr}_r}(\tilde{\chi}^o\psi) - \tilde{\psi}^{\text{\rm rem}}$ belongs to the vector space spanned by
\begin{align*}
&\big\{\tilde{\psi}_{k,m}^{(0)} \big| (k,m)\in \tilde{E}_n^o \big\} \\
&\oplus \big\{\tilde{\psi}_{k,m}^{(0)}\big|k+mh(20\epsilon)=n\text{, and }{\textstyle \frac{2n}{V}(v-27\epsilon)\leq m<\frac{2n}{V}(v-20\epsilon) } \big\} \\
&\oplus \big\{\tilde{\psi}_{k,m}^{(0)}\big|k+mh(20\epsilon)=n\text{, and }{\textstyle \frac{2n}{V}(v+20\epsilon)<m\leq\frac{2n}{V}(v+27\epsilon) } \big\} ~.
\end{align*}
\end{enumerate}
\end{lem}

%%%%
\subsection*{Step 9}  In this step, we apply lemma \ref{lem_decomp_sigma_00} to study $\hat{\chi}{{\rm pr}_n}(\hat{\chi}^o\psi)$.

\begin{lem}\label{lem_hat_r_dep_00}
There exists a constant $c$ such that the following holds.  For any $n\geq c$ and $r\in(s_{n-1},s_n]$, suppose that $\psi$ is a zero mode of $D_r$ with unit $L^2$-norm.  Then there exists a section $\hat{\psi}^{\text{\rm rem}}$ such that
\begin{enumerate}
\item the support of $\hat{\psi}^{\text{\rm rem}}$ is contained in the support of $\hat{\chi}$;\smallskip
\item $\int_Y |\hat{\psi}^{\text{\rm rem}}|^2\leq c s_n^{-2}$;\smallskip
\item the second component of $\hat{\chi}{{\rm pr}_n}(\hat{\chi}^o\psi) - \hat{\psi}^{\text{\rm rem}}$ is zero.  The first component of $\hat{\chi}{{\rm pr}_n}(\hat{\chi}^o\psi) - \hat{\psi}^{\text{\rm rem}}$ belongs to the vector space spanned by
\begin{align*}
&\Big\{ \hat{\chi}e^{in\phi}\cdot\ker_0\bar{\pl}_n \Big\} \oplus \Big\{ \check{\zeta}_{n,m}e^{in\phi} \big| {\textstyle \frac{n}{C}(1-11\epsilon)\leq m \leq \frac{n}{C}(1-3\epsilon) } \Big\} \\
&\quad\oplus \Big\{ \tilde{\zeta}_{n,m}e^{in\phi} \big| {\textstyle \frac{2n}{V}(v+23\epsilon)\leq m\leq\frac{2n}{V}(v+31\epsilon) } \text{, or } \\ 
&\qquad\qquad\qquad\quad {\textstyle \frac{2n}{V}(v-31\epsilon)\leq m\leq\frac{2n}{V}(v-23\epsilon) } \Big\} ~.
\end{align*}
\end{enumerate}
\end{lem}
\begin{proof}[Proof of lemma \ref{lem_hat_r_dep_00}]
According to the decomposition given by lemma \ref{lem_decomp_sigma_00}, the first component of ${{\rm pr}_n}(\hat{\chi}^o\psi)$ can be expressed as
\begin{align*}
\zeta_0 + \sum \check{\mathfrak{c}}_{n,m} (\check{\mathfrak{p}}_{n,m}\check{\zeta}_{n,m} + \check{\zeta}_{n,m}^{\text{\rm rem}}) + \sum \tilde{\mathfrak{c}}_{n,m} (\tilde{\mathfrak{p}}_{n,m}\tilde{\zeta}_{n,m} + \tilde{\zeta}_{n,m}^{\text{\rm rem}})
\end{align*}
where $\zeta_0\in\ker_0\bar{\pl}_n$.  Let $\hat{\psi}^{\text{\rm rem}}$ be the sum of the following terms:
\begin{itemize}
\item $\hat{\chi}\check{\mathfrak{c}}_{n,m} (\check{\mathfrak{p}}_{n,m}\check{\zeta}_{n,m} + \check{\zeta}_{n,m}^{\text{\rm rem}})$ for $\frac{n}{V}(1-3\epsilon)<m<\frac{n}{V}$;\smallskip
\item $\hat{\chi}\check{\mathfrak{c}}_{n,m} \check{\zeta}_{n,m}^{\text{\rm rem}}$ for $\frac{n}{V}(1-11\epsilon)\leq m\leq\frac{n}{V}(1-3\epsilon)$;\smallskip
\item $\hat{\chi}\tilde{\mathfrak{c}}_{n,m} (\tilde{\mathfrak{p}}_{n,m}\tilde{\zeta}_{n,m} + \tilde{\zeta}_{n,m}^{\text{\rm rem}})$ for $\frac{2n}{V}(v+20\epsilon)< m <\frac{2n}{V}(v+23\epsilon)$ or $\frac{2n}{V}(v-23\epsilon)< m<\frac{2n}{V}(v-20\epsilon)$;\smallskip
\item $\hat{\chi}\check{\mathfrak{c}}_{n,m} \tilde{\zeta}_{n,m}^{\text{\rm rem}}$ for $\frac{2n}{V}(v+23\epsilon)\leq m\leq\frac{2n}{V}(v+31\epsilon)$ or $\frac{2n}{V}(v-31\epsilon)\leq m\leq\frac{2n}{V}(v-23\epsilon)$.
\end{itemize}
For the terms of the first and third kind, lemma \ref{lem_decomp_sigma_00} gives a constant $c_{22}$ such that their $L^2$-norm is less than $c_{22}\exp(-\frac{n}{c_{22}})$.  For the terms of the second and fourth kind, the supports of $\check{\zeta}_{n,m}$ and $\tilde{\zeta}_{n,m}$ are disjoint from the support of $\hat{\chi}^o$.  A similar argument as in the proof of lemma \ref{lem_check_r_dep_00} shows that their $L^2$-norm is less than $c_{22}\exp(-\frac{n}{c_{22}})$.  With the triangle inequality,
\begin{align*} \int_{Y} |\hat{\psi}^{\text{\rm rem}}|^2 \leq \big( \frac{110\epsilon}{V}n \,c_{22}\exp(-\frac{n}{c_{22}})\big)^2 \leq c_{23}s_n^{-2} \end{align*}
for some constant $c_{23}$.  This proves assertion (i) and (ii) of the lemma.

For the last assertion, it is clear that $\hat{\chi}{{\rm pr}_n}(\hat{\chi}^o\psi) - \hat{\psi}^{\text{rem}}$ belongs to the vector space spanned by those elements in assertion (iii), multiplied by $\hat{\chi}$.  However, for those $\check{\zeta}_{n,m}$ and $\tilde{\zeta}_{n,m}$ in the last two summand of (iii), $\hat{\chi}$ is equal to $1$ on their supports.  This completes the proof of lemma \ref{lem_hat_r_dep_00}.
\end{proof}

%%%%
\subsection*{Step 10}  In this step, we combine all the results to prove the claim (\ref{claim_upper_bound_00}).

\begin{prop}\label{prop_final_proj_00}
There exists a constant $c$ with the following significance.  For any $n\geq c$ and $r\in(s_{n-1},s_n]$, suppose that $\psi$ is a zero mode of $D_r$ of unit $L^2$-norm.  With proposition \ref{prop_proj_finite_00} and lemma \ref{lem_check_r_dep_00}, \ref{lem_tilde_r_dep_00} and \ref{lem_hat_r_dep_00}, let 
\begin{align*}
{\Pi}(\psi) = {\Pi_r}(\psi) - \hat{\psi}^{\text{\rm rem}} - \check{\psi}^{\text{\rm rem}} - \tilde{\psi}^{\text{\rm rem}} ~,
\end{align*}
then
\begin{enumerate}
\item $\int_Y |{\Pi}(\psi)|^2\geq 1-cs_n^{-1}$;\smallskip
\item ${\Pi}(\psi)$ belongs to a vector space whose dimension is $$I_\Sigma(n)+\#\check{E}_n^o+\#\tilde{E}_n^o ~;$$
\item if there are two such zero modes $\psi_1$ and $\psi_2$ at $r_1$ and $r_2$ with $s_{n-1}<r_1<r_2\leq s_n$, then
\begin{align*}
\big|\int_Y\langle {\Pi}(\psi_1),{\Pi}(\psi_2) \rangle\big| \leq c s_n^{-1} ~;
\end{align*}
\item if there are two such zero modes $\psi_1$ and $\psi_2$ both at $r\in(s_{n-1},s_n]$, and $\int_Y\langle\psi_1,\psi_2\rangle=0$, then
\begin{align*}
\big|\int_Y\langle {\Pi}(\psi_1),{\Pi}(\psi_2) \rangle\big| \leq c s_n^{-1} ~.
\end{align*}
\end{enumerate}
\end{prop}
\begin{proof}[Proof of proposition \ref{prop_final_proj_00}]
Assertion (i), (iii) and (iv) follows from proposition \ref{prop_proj_finite_00} and lemma \ref{lem_check_r_dep_00}, \ref{lem_tilde_r_dep_00} and \ref{lem_hat_r_dep_00}.

With the observation in step 7, ${\Pi}(\psi)$ belongs to the vector space spanned by
\begin{align}\begin{split}\label{eqn_space_decomp_00}
&\Big\{ \hat{\chi}e^{in\phi}\cdot\ker_0\bar{\pl}_n \Big\} \oplus \Big\{ \check{\zeta}_{n,m}e^{in\phi} \big| {\textstyle \frac{n}{V}(1-11\epsilon)\leq m<\frac{n}{V}} \Big\} \\
&\oplus \Big\{ \tilde{\zeta}_{n,m}e^{in\phi} \big| {\textstyle \frac{2n}{V}(v+20\epsilon)<m\leq\frac{2n}{V}(v+1+31\epsilon)} \text{, or }\\
&\qquad\qquad\qquad {\textstyle \frac{2n}{V}(v-31\epsilon)\leq m<\frac{2n}{V}(v-20\epsilon)} \Big\} \\
&\quad \oplus \Big\{ \check{\psi}_{k,m}^{(0)} \big| (k,m)\in\check{E}_n^o \Big\} \oplus \Big\{ \tilde{\psi}_{k,m}^{(0)} \big| (k,m)\in\tilde{E}_n^o \Big\} ~.
\end{split}\end{align}
To be more precise, the first three summands consist of sections whose first component is given by those elements and the second component is zero.  By dimension counting, the dimension of the subspace spanned by the first three summands is $I_{\Sigma}(n)$.  This completes the proof of proposition \ref{prop_final_proj_00}.
\end{proof}

According to lemma \ref{lem_beta_vanish_00}, (\ref{eqn_APS_00}), (\ref{eqn_E_total_number_00}) and (\ref{eqn_E_total_number_100}), there exists a constant $c_{25}$ such that
\begin{align*} I_\Sigma(n) + \#\check{E}_n^o + \#\tilde{E}_n^o \leq c_{25}s_n \end{align*}
for all $n\geq c_{25}$.

Let $L_n$ be the index set $\big\{1,2,\cdots,{\rm sf}_a(s_n) - {\rm sf}_a(s_{n-1})\big\}$.  We may assume that there are only positive zero crossings.  For each zero crossing happening between $(s_{n-1},s_n]$, choose a zero eigensection with unit $L^2$-norm.  If there are more than one zero crossings happening at some $r\in(s_{n-1},s_n]$, choose $L^2$-orthonormal zero eigensections.  Let $\{\psi_l\}_{l\in L_n}$ be the set of these zero modes.

According to proposition \ref{prop_final_proj_00}, there exists a constant $c_{26}$ such that
\begin{enumerate}
\item $\int_Y |{\Pi}(\psi_l)|^2 \geq 1 - c_{26}s_n^{-1}$;\smallskip
\item $\big| \int_Y\langle{\Pi}(\psi_l),{\Pi}(\psi_{l'})\rangle \big| \leq c_{26}s_n^{-1}$ for any $l\neq l'$;\smallskip
\item ${\Pi}(\psi_l)$ belongs to a vector space (\ref{eqn_space_decomp_00}), whose dimension is
\begin{align*} I_\Sigma(n) + \#\check{E}_n^o + \#\tilde{E}_n^o \leq c_{25}s_n \end{align*}
\end{enumerate}
for all $n\geq c_{26}$.  After normalizing the $L^2$-norm of ${\Pi}(\psi_l)$, lemma \ref{lem_Welch_dim_00} applies.  This proves claim (\ref{claim_upper_bound_00}), and completes the proof of theorem \ref{thm_sf_upper_bound}.
\end{proof}

%%%%%%%%

\end{document}